\documentclass[a4paper]{amsart}
\usepackage[textwidth=15cm,top=3cm, bottom=3cm, hcentering]{geometry}
\usepackage[colorlinks=true,linkcolor=red,citecolor=blue]{hyperref}
\usepackage{bbm}
\usepackage{bm}
\usepackage{amsmath,amscd,amsthm,amssymb,mathtools}
\usepackage{color}
\usepackage[dvipsnames,table]{xcolor}
\usepackage[utf8]{inputenc}
\normalfont
\usepackage[T1]{fontenc}
\usepackage{tikz-cd}
\usetikzlibrary{cd,shapes,math}
\usepackage{tikz}
\usetikzlibrary{positioning,calc,decorations.pathreplacing,arrows.meta}
\usepackage{enumerate}
\usepackage{longtable}
\usepackage{graphicx}
\usepackage{relsize}
\usepackage{verbatim}
\usepackage{enumitem}
\usepackage{quiver}
\usepackage{shuffle}

 \usepackage[all]{xy}

\definecolor{light-gray}{HTML}{E5E4E2}

\usepackage[export]{adjustbox}

\newtheorem{theo}{Theorem}[section]
\newtheorem{lemm}[theo]{Lemma}
\newtheorem{prop}[theo]{Proposition}
\newtheorem{coro}[theo]{Corollary}
\newtheorem*{theo*}{Theorem}

\newtheoremstyle{named}{}{}{\itshape}{}{\bfseries}{.}{.5em}{#1 \thmnote{#3}}
\theoremstyle{named}
\newtheorem*{namedtheorem}{Lemma}

\theoremstyle{definition}
\newtheorem{defi}[theo]{Definition}
\newtheorem{rema}[theo]{Remark}

\newtheorem{exam}[theo]{Example}

\theoremstyle{plain}
\newtheorem{thmintro}{Theorem}

\newcommand{\CC}{\mathbb{C}}

\newcommand{\HH}{\mathbb{H}}
\newcommand{\II}{\mathbb{I}}

\renewcommand{\SS}{\mathbb{S}}
\newcommand{\TT}{\mathbb{T}}

\newcommand{\ZZ}{\mathbb{Z}}
\newcommand{\Aa}{\mathcal{A}}
\newcommand{\Bb}{\mathcal{B}}
\newcommand{\Cc}{\mathcal{C}}
\newcommand{\Dd}{\mathcal{D}}

\newcommand{\Mm}{\mathcal{M}}
\newcommand{\Nn}{\mathcal{N}}
\newcommand{\Oo}{\mathcal{O}}
\newcommand{\Pp}{\mathcal{P}}
\newcommand{\Qq}{\mathcal{Q}}
\newcommand{\Rr}{\mathcal{R}}

\newcommand{\Ho}{\mathrm{Ho}}
\newcommand{\Ker}{\mathrm{Ker}}
\newcommand{\Coker}{\mathrm{Coker}}
\newcommand{\Img}{\mathrm{Im}}
\newcommand{\Hom}{\mathrm{Hom}}

\newcommand{\ov}{\overline}

\newcommand{\p}{\mathfrak{p}}
\newcommand{\q}{\mathfrak{q}}
\newcommand{\mup}{\mathfrak{m}}

\newcommand{\kk}{\mathbf{k}}
\newcommand{\del}{\partial}
\newcommand{\delb}{{\bar \partial}}
\newcommand{\A}{H_{\Aa}}
\newcommand{\BC}{H_{\Bb\Cc}}
\newcommand{\ABC}{H_{\Aa\Bb\Cc}}

\newcommand{\lra}{\longrightarrow}
\newcommand{\Lpic}{\llcorner}

\newcommand{\ele}{\tikz[baseline=0mm]{ 
\draw[thick] (0,0.015) -- (0.2,0.015); 
\draw[thick]  (0,0) -- (0,0.22);
}}

\newcommand{\antiele}{\tikz[baseline=0mm]{ 
\draw[thick] (0,0.2) -- (0.2,0.2); 
\draw[thick]  (0.2,0) -- (0.2,0.215);
}}

\newcommand\doq{%
  \mathord{%
    \mspace{1mu}%
    \text{\hspace{0.5ex}\Doq}%
    \mspace{1mu}%
  }%
}
\newcommand\Doq{%
    \tikz[line cap=round,x=1ex,y=1ex,line width=0.3pt]
    {\draw (1,0) |- (0,1) (0.55,0);}%
}

\newcommand{\ops}{\mathrm{Op}(\bico)}
\newcommand{\coops}{\mathrm{CoOp(\bico)}}
\newcommand{\bico}{\mathrm{BiCo}}
\newcommand{\bicok}{\mathrm{BiCo}_{\kk}}
\newcommand{\Ch}{\mathrm{Ch}}

\newcommand{\smod}{\SS\,\text{-}\mathrm{Mod}}
\newcommand{\dgsmod}{\SS\,\mep\Ch^{\leq 0}}
\newcommand{\bicosmod}{\SS\,\mep\mathrm{BiCo}}
\newcommand{\bbsmod}{\mathrm{bb}\text{-}\SS\,}
\newcommand{\bicosmodn}{\SS\,\mep\mathrm{BiCo}^{\doq}}
\newcommand{\pinfal}{\Pp_\infty^{pp}\text{-}\mathrm{Alg}}
\newcommand{\pinfalm}{\infty\text{\,-}\Pp^{pp}_\infty\text{-}\mathrm{Alg}}
\newcommand{\Palg}{\Pp\text{-}\mathrm{Alg}}

\newcommand{\ppinfa}{A^{pp}_\infty}
\newcommand{\ppinfm}{\scalebox{0.80}{\(\stackrel{pp}\infty\)}}
\newcommand{\Endop}{\mathcal{E}nd}
\newcommand{\ppinfc}{C^{pp}_\infty}

\newcommand{\id}{\mathrm{id}}

\newcommand{\vp}{\varphi}

\newcommand{\tens}[1]{
  \mathbin{\mathop{\otimes}\displaylimits_{#1}}
}
\newcommand{\br}{\mathrm{B}}
\newcommand{\cbr}{\Omega}

\newcommand{\varsh}{\varphi^{10}}
\newcommand{\varsv}{\varphi^{01}}

\newcommand{\Tw}{\mathrm{Tw}}

\newcommand{\mep}{\text{-}}

\newcommand{\antshrk}{\mbox{!`}}
\newcommand{\antshrkind}{\text{¡}}

\newcommand{\Inf}{\mathrm{Inf}}

\title{Pluripotential Operadic Calculus}

\author{Anna Sopena-Gilboy}

\thanks{
 Financial support from AGAUR, Govern de Catalunya, through the FI Program and from the Spanish State Research Agency through projects PID2020-117971GB-C22, PID2024-155646NB-I00 and EUR2023-143450.
}

\begin{document}

\begin{abstract}
This paper explores a new context for operadic calculus and  Koszul duality arising in the study of complex manifolds or, more abstractly, of bidifferential bigraded algebras up to pluripotential weak equivalence. In particular, we establish a Pluripotential Homotopy Transfer Theorem. This provides a homotopical framework for the theory of Massey products in Bott-Chern and Aeppli cohomologies, defined for any complex manifold.
\end{abstract}
\maketitle 

\tableofcontents
\newpage

\section{Introduction}
The complexified de Rham algebra of differential forms of a complex manifold 
admits a decomposition into $(p,q)$-differential forms. This decomposition splits the exterior differential into two components $d=\del+\delb$, making it into a bicomplex. From such an object, one can compute several cohomological invariants aside from the cohomology with respect to the total differential: most notably Dolbeault cohomology \cite{Dol53}, as well as Aeppli \cite{Ae64} and Bott–Chern cohomology \cite{BottChern65}, given respectively by:
\[H_\delb:=\frac{\ker \delb}{\Img\, \delb},\quad \BC:=\frac{\ker\del\cap\ker\delb}{\Img\, \del\delb},\quad \A:=\frac{\ker\del\delb}{\Img\,\del+\Img\,\delb}.\]
These cohomologies have been intensively studied in complex geometry as algebraic invariants carrying 
complex-analytic information. In particular, Bott-Chern cohomology plays a key role in the study of non-Kähler geometry, deformation theory, and special Hermitian metrics (see for instance \cite{Angella14}).

If we additionally consider the multiplicative structure of the differential forms one obtains a commutative bidifferential bigraded algebra (cbba for short).
Motivated by the behavior of wedge products with respect to harmonic forms of complex manifolds,
Angella and Tomassini  \cite{AnTo15} introduced
triple Massey-like products defined for Bott-Chern and Aeppli cohomologies, called \textit{ABC-Massey products}. In \cite{MiSte24}, Milivojevic and Stelzig extend these products to higher arities and explore notions of formality obstructed by these invariants. 
This leads to the consideration of \textit{pluripotential weak equivalences}, defined by those maps inducing
isomorphisms in both Bott-Chern and Aeppli cohomologies. 
 Remarkably, compact Kähler manifolds can have non-trivial ABC-Massey products, showing that the pluripotential weak equivalence class of the complexified de Rham algebra is a finer invariant than its cohomology with the real Hodge structure \cite{PSZ24}.

In \cite{Ste25}, Stelzig studies the homotopy theory of cbba's over a field of characteristic zero with respect to pluripotential weak equivalences, through the  construction of minimal models of cbba's à la Sullivan. With this theory, Stelzig obtains various refinements of
the homotopy groups of complex manifolds, sensitive to the analytic structure. 

A complementary  method to study homotopy types of differential graded algebras is through the so-called Homotopy Transfer Theory: over a field, one can always construct a contraction connecting a dga  to its cohomology, which induces an $A_\infty$-structure on cohomology, unique up to $A_\infty$-isomorphism \cite{Ka80}. This $A_\infty$-structure is closely related to Massey products and encodes the homotopy type of the dga. When applied to the differential forms of a smooth manifold or, more generally, to the piece-wise linear forms of a topological space, the $A_\infty$-structure is in fact a $C_\infty$-structure and encodes the rational homotopy type of the space. 

In this paper, we develop homotopy transfer theory in the pluripotential setting.
One main advantage of our approach with respect to Stelzig's theory is that our setting works without the connectivity assumptions. Moreover,
it situates ABC-Massey products in a robust homotopical framework in which new higher operations emerge naturally.

To formalize our theory, we establish a new operadic framework: we work with operads in the symmetric monoidal category of bicomplexes over a field of characteristic zero, and construct minimal models of operads with respect to pluripotential weak equivalences in two complementary ways. First, following the more classical approach of Markl, Shnider, and Stasheff \cite{MaShniSta02} we obtain minimal models for any reduced operad using free extensions. In our setting, the notion of minimal operad requires cofibrancy
together with the condition that the composition $\del\delb$ is decomposable. We prove:
\begin{thmintro}
    Let $\Pp$ be an operad in bicomplexes such that $\Pp(0)=0$ and $\Pp(1)=\kk$. Then, there exists a minimal operad $\Mm$ and a pluripotential weak equivalence $\Mm\xrightarrow[]{\sim}\Pp$.
\end{thmintro}

The second way is through Koszul duality. We use the \emph{inflation functor}, defined in \cite{MaSo}, from cochain complexes to bicomplexes
\(\Inf:\Ch\to \bico\)
as part of a Quillen adjunction, with right adjoint given by 
a well-known sheaf-theoretic construction in complex geometry \cite{De97}, \cite{Sch07}, \cite{Piovani}.
This functor allows us to recycle many of the computations of the dg case. For an operad \(\Pp\) concentraded in bidegreee \((0,0)\) with trivial differentials we define \(\Pp_\infty^{pp}\) as the free operad generated by \(\Inf(s\ov \Pp^{\antshrk})\) with two additional differentials. Here, \(s\ov \Pp^{\antshrk}\) denotes the suspension of the Koszul dual cooperad of \(\Pp\) viewed as a dg-operad. We show:
\begin{thmintro}
    Let $\Pp$ be a Koszul operad in bicomplexes. The operad
    \(\Pp_\infty^{pp}\) is the pluripotential minimal model of \(\Pp\).
\end{thmintro}

Algebras over the minimal model of the associative operad are called \emph{pluripotential \(A_\infty\)-algebras} (\(\ppinfa\)-algebras for short). As expected, these are algebraic structures in which associativity holds up to higher homotopies in the pluripotential sense: there is a family of \(2n-3\) operations of arity \(n\geq 2\) and, in particular, 
there is a single binary operation $\mu_2$ of bidegree zero satisfying
\[
\mu_2(\mu_2\otimes\mathrm{id})-\mu_2(\mathrm{id}\otimes\mu_2)=[\del,[\delb,\mu_3]].
\]
Here, \(\mu_3\) is a ternary operation of bidegree \((-1,-1)\). Note that this equation states that $\mu_3$ is a homotopy measuring the failure of $\mu_2$ to be associative. Indeed,
a pluripotential homotopy between morphisms of bicomplexes $f$ and $g$ is given by a map $h$ satisfying 
\[
[\del,[\delb,h]]=f-g,
\]
where $\del$ and $\delb$ denote the two differentials of the bicomplexes. 

In the classical setting, homotopy transfer is carried into a minimal chain complex, characterized by having trivial differential.
Here, minimality is achieved
by bicomplexes such that the composition of the two differentials is trivial: $\del\delb=0$.
We prove:

\begin{thmintro}\label{theointro:HTT}
    Let $(A,\del,\delb)$ be a bidifferential bigraded algebra.
\begin{enumerate}
    \item There is a contraction into a minimal bicomplex
\[
\begin{tikzcd}[ampersand replacement = \&]
(A,\del,\delb) \arrow[r, rightarrow, shift left, "f"] \arrow[r, shift right, leftarrow, "  g" swap] \arrow[loop left, "h"] \& (H,\del,\delb)
\end{tikzcd};\, \text{ with } \del\delb=0\text{ on }H.
\]
\item There is a $\ppinfa$-structure on $(H,\del,\delb)$ and a pluripotential $\infty$-weak equivalence $  g:H\rightsquigarrow A$.
\end{enumerate}
\end{thmintro}

Note in the classical setting, minimal complexes are 
isomorphic to the cohomology. In the pluripotential case, $H$ is a bicomplex with possibly non-trivial differentials, and the minimality ensures that this is isomorphic to 
$H_{\Bb\Cc}(H)\oplus \tilde{H}_{\Aa}(H)$, where $\tilde{H}_{\Aa}(H)$ is the cokernel of the map $H_{\Bb\Cc}(H)\to H_{\Aa}(H)$ from Bott-Chern to Aeppli cohomology. Morally, this result states that, under a choice of sections, any bicomplex admits its 'pluripotential cohomology' as a deformation retract, and one can transfer the multiplicative structure along this retract.

In arity three, Theorem~\ref{theointro:HTT} recovers the 
ABC-Massey products of \cite{AnTo15}. In the case of complex manifolds, this gives new biholomorphic invariants. In fact, we prove a commutative version of the above theorem (Theorem~\ref{theo:HTTcom}) which gives $\ppinfc$-algebra structures encoding the pluripotential homotopy type of the complex manifold. More generally, we prove:
\begin{thmintro}
Let $\Pp$ be a Koszul operad. There is an equivalence of homotopy categories
    \[\mathrm{Ho}(\Palg)\cong\mathrm{Ho}(\pinfalm).\]
\end{thmintro}
Here, $\Palg$ denotes the category of $\Pp$-algebras while 
$\pinfalm$ denotes the category of $\Pp^{pp}_\infty$-algebras whose morphisms are $\infty$-morphisms, arising from pluripotential Koszul duality theory. The localization $\Ho$ is with respect to pluripotential weak equivalences.  

When applied to the complexified de Rham algebra of differential forms of a complex manifold, the results of this paper establish a robust homotopical setting that translates analytical information of complex structures into algebraic data. This offers a new framework to study the so-called strong and weak formality and, more generally, pluripotential homotopy types of complex manifolds. This is an active area of research, central in the recent works mentioned at the beginning of this introduction, as well as in \cite{SfeTo22}, \cite{PSZ25} and \cite{Ru26}, to name a few.

\medskip 
We briefly explain the structure of this paper.
To make the paper accessible to geometers, we begin the first section by giving explicit definitions for the case of \(\ppinfa\)‑algebras and sketch the connections to geometry, avoiding the harder operadic machinery. In the subsequent sections we justify that these definitions are the natural ones from a homotopical point of view and develop the surrounding theory. Namely, the paper is organized as follows: in Section~\ref{sec:preliminars} we define
pluripotential \(\ppinfa\)-algebras and prove the Homotopy Transfer Theorem in the associative and commutative cases (Theorem~\ref{theo:HTT} and Theorem~\ref{theo:HTTcom}). We also give an explicit relation between transferred structures and triple ABC-Massey products. In Section~\ref{sec:BasicOperadThry} we review some constructions for operads and cooperads in the category of bicomplexes. In Section~\ref{sec:MinimalModSul} we show existence and uniqueness of minimal models of operads in bicomplexes with respect to pluripotential weak equivalences (Theorem~\ref{theo:MinimalModelsSullivan}). In Section~\ref{sec:koszulduality} we develop Koszul duality for operads in bicomplexes in bidegree \((0,0)\) with trivial differentials and prove Theorem~\ref{theo:MinimalModelsKoszul} on minimal models. Lastly, in Section
\ref{sec:HomThPalg} we define the bar and cobar constructions for operads in bicomplexes with respect to a twisting morphism and prove Theorem~\ref{theo:EqHoCats} on the equivalence of homotopy categories.
\medskip

\subsection*{Notation and conventions}
Througout this work, $\kk$ denotes a field of characteristic zero. For a pure
(bi)degree element \(v\in V\) of a bigraded object we use the notation \(|v|\) for the 
total degree.

\subsection*{Acknowledgments}
I would like to thank my advisors Joana Cirici and Jonas Stelzig for their guidance and support. I am grateful to Muriel Livernet and Bruno Vallette for their dedication and discussions when I visited them in Paris. Thanks also to Ricardo Campos and Geoffroy Horel for useful comments.

\section{Pluripotential \texorpdfstring{$A_\infty$}{A-infinity}-algebras and homotopy transfer theorem}\label{sec:preliminars}

In this section we introduce the category of pluripotential $A_\infty$-algebras and prove the Pluripotential Homotopy Transfer Theorem for bidifferential bigraded algebras, giving explicit formulas. At the end of the section, we relate Pluripotential \texorpdfstring{$A_\infty$}{A-infinity}-structures with ABC-Massey products defined for complex manifolds \cite{AnTo15}.

\subsection{Bicomplexes and pluripotential weak equivalences}\label{subsec:preliminarsbico}
We review the main properties of the category
 $\bicok$ of bicomplexes over $\kk$ (also called double complexes in the literature). All results claimed without proof in this section, can be found in \cite{Ste25}. Objects in this category are given by bigraded vector spaces 
\[A=\bigoplus_{p,q\in\ZZ} A^{p,q}\]
together with two linear maps $\del$ and $\delb$
of bidegrees $(1,0)$ and $(0,1)$, respectively, and
satisfying the relations
\[\del^2=0,\quad\delb^2=0\quad\text{and}\quad\del\delb+\delb\del=0.\]
Morphisms in this category are given by bidegree-preserving linear maps compatible with both $\del$ and $\delb$.

\begin{rema}
Although we use the notation $\del$ and $\delb$ for the differentials of a bicomplex, borrowed from the notation of complex manifolds, our bicomplexes are defined over an arbitrary field, so we do not assume any real structure.
\end{rema}

The category $\bicok$ is a symmetric monoidal category, with the tensor product given by 
\[(A\otimes B)^{p,q}=\bigoplus_{r,s\in\ZZ} A^{r,s}\otimes B^{p-r,q-s}\]
with differentials given by the Leibniz rule. The unit is given by the field $\kk$ concentrated in bidegree $(0,0)$.

The internal Hom is defined by 
\[\underline{\mathrm{Hom}}(A,B)^{p,q}=
\prod_{r,s\in\ZZ}\Hom_\kk(A^{r,s},B^{r+p,s+q}),\]
with differentials 
\[\del f:=[\del,f]\quad\text{ and }\quad\delb f:=[\delb,f],
\]
where $[-,-]$ is the graded commutator
\[[f,g]:=f\circ g-(-1)^{|f|\cdot |g|}g\circ f.\]

The internal Hom and tensor product form an adjunction
\[
(- \otimes A) : \mathrm{\bicok} \;\rightleftarrows\; \mathrm{\bicok} : \underline{\mathrm{Hom}}(A,-).
\]

The category \(\bicok\) is then a closed symmetric monoidal category. 

Associated to a bicomplex $(A,\del,\delb)$ there are various functorial cohomologies:
aside from the $\del$- and $\delb$-cohomologies
\[
H^{*,*}_\del(A):={\frac{\Ker(\del)}{\Img(\del)}}\quad\text{ and }\quad H^{*,*}_\delb(A):=\frac{\Ker(\delb)}{\Img(\delb)},
\]
we also have the Bott-Chern $\BC$ and Aeppli $\A$ cohomologies, given respectively by
\[H_{\Bb\Cc}^{*,*}(A):={\frac{\Ker(\delb)\cap \Ker(\del)}{\Img(\delb\del) }}\quad\text{ and }\quad
H_{\Aa}^{*,*}(A):={\frac{\Ker(\delb\del)}{\Img(\delb)+\Img(\del)}}.\]
Note that the identity induces natural transformations
\[\begin{tikzcd}[ampersand replacement=\&]
	\& {\BC} \\
	{H_{\del}} \& {H_{dR}} \& {H_{\delb}} \\
	\& {\A}
	\arrow[from=1-2, to=2-1]
	\arrow[from=1-2, to=2-2]
	\arrow[from=1-2, to=2-3]
	\arrow[from=2-1, to=3-2]
	\arrow[from=2-2, to=3-2]
	\arrow[from=2-3, to=3-2]
\end{tikzcd}\]
where $H_{dR}$ denotes the total cohomology. Moreover, the row and column filtrations on a bicomplex \(A\) induce two spectral sequences relating the total cohomology with \(\del\)- and \(\delb\)- cohomologies, respectively:
\[
H_\del(A)\Rightarrow H_{dR}(A)\Leftarrow H_{\delb}(A).
\]

A bicomplex \(A\) is said to \emph{satisfy the \(\del\delb\)-property} if the map of bigraded vector spaces
\[
\BC(A) \to \A(A)
\]
is injective (and hence this map, and all the maps in the above diagram, are isomorphisms).

Let $\widetilde{H}_{\Bb\Cc}$ and $\widetilde{H}_{\Aa}$ be the functors which send a bicomplex $A$ to the kernel and cokernel of $\BC(A)\to \A(A)$, respectively. Explicitly, this gives:
\[\widetilde{H}_{\Bb\Cc}=\frac{\Img\del\cap\Ker\delb + \Img\delb \cap\Ker\del}{\Img\del\delb},\qquad\widetilde{H}_{\Aa}=\frac{\Ker\del\delb}{\Img\del+\Img\delb+\Ker\del\cap\Ker\delb}.\]
The \emph{ABC-cohomology} of a bicomplex $(A,\del,\delb)$ is the bicomplex defined by
\[\ABC(A):= \BC(A) \oplus \widetilde{H}_{\Aa}(A)\]
together with the two differentials $\del$ and $\delb$ induced from those of $A$. Note that this bicomplex is \emph{pluripotentially minimal}, in the sense that $\del\delb\equiv 0$.

We will use the following result:
\begin{lemm}\label{lemm:long-exact-seq}
    Any short exact sequence of bicomplexes
    \[0\to A\to B\to C\to 0\]
    induces a long exact sequence
    \begin{align*}
        \dots \to H_{\Bb_{p,q}}^{p+q-2}(B)\to H_{\Bb_{p,q}}^{p+q-2}(C)\to \A^{p-1,q-1}(A)\to\A^{p-1,q-1}(B)\to \A^{p-1,q-1}(C)\to\\
        \to\BC^{p,q}(A)\to \BC^{p,q}(B)\to \BC^{p,q}(C)\to H_{\Bb_{p,q}}^{p+q+1}(A)\to H_{\Bb_{p,q}}^{p+q+1}(B)\to\dots
    \end{align*}
    where \(H_{\Bb_{p,q}}^{k}\) denotes the cohomology of the Bigolin complex in degree \(k\).
\end{lemm}

Any bicomplex $A$ admits a (noncanonical) decomposition 
\[A=A_{\mathrm{zig}}\oplus A_{\mathrm{sq}}\]
where $A_{\mathrm{sq}}$ is a direct sum of squares, of the form
\[\begin{tikzcd}
	\kk & \kk \\
	\kk & \kk,
	\arrow["\delb", from=2-1, to=1-1]
	\arrow["\del", from=1-1, to=1-2]
	\arrow["\del"', from=2-1, to=2-2]
	\arrow["\delb"', from=2-2, to=1-2]
\end{tikzcd}\]
and $A_{\mathrm{zig}}$ is a bicomplex satisfying $\del\delb=0$. Therefore, \(A_{\mathrm{zig}}\) is a direct  sum of bicomplexes of the form 

\[\begin{tikzcd}
\kk
\end{tikzcd},\qquad\quad
\begin{tikzcd}
	\kk \\
	\kk
	\arrow["\delb", from=2-1, to=1-1]
\end{tikzcd},\qquad\quad
\begin{tikzcd}
	\kk & \kk
	\arrow["\del", from=1-1, to=1-2]
\end{tikzcd},\qquad\quad
\begin{tikzcd}
	\kk \\
	\kk & \kk
	\arrow["\del", from=2-1, to=2-2]
	\arrow["\delb", from=2-1, to=1-1]
\end{tikzcd},\qquad\quad
\begin{tikzcd}
	\kk \\
	\kk & \kk \\
	& \kk
	\arrow["\del", from=2-1, to=2-2]
	\arrow["\delb", from=2-1, to=1-1]
	\arrow["\delb", from=3-2, to=2-2]
\end{tikzcd}\quad...
\]
In squares and zig-zags, $\del$ and $\delb$ are isomorphisms. This result has long been known as a folklore statement. Complete proofs can be found in \cite{Ste21} and \cite{KhoQi20}.

With such a decompostion, we have $A_{\mathrm{zig}}\cong \ABC(A)$.
In particular, the summand $A_{\mathrm{sq}}$ is pluripotentially acyclic. If the bicomplex satisfies the \(\del\delb\)-lemma property, the bicomplex \(A_{\mathrm{zig}}\) has trivial differentials, and hence the bicomplex \(A\) splits as a direct sum of squares and a bigraded vector space.

The category of bicomplexes admits a model structure whose weak equivalences are defined in terms of Bott–Chern and Aeppli cohomologies. The rest of this section reviews this model structure and the relevant definitions.

\begin{defi}
A morphism of bicomplexes $f\colon A\to B$ is said to be a \emph{pluripotential weak equivalence}
if $\A(f)$ and $\BC(f)$ are isomorphisms.
\end{defi}

\begin{rema}\label{Asqacyclic}
Being a pluripotential weak equivalence $f\colon A\to B$ is equivalent to inducing an isomorphism 
$\ABC(f)$
at the level of ABC-cohomology. In \cite{Ste25}, pluripotential weak equivalences are also called \emph{pluripotential quasi-isomorphisms} or \emph{bigraded quasi-isomorphisms}.
\end{rema}

\begin{rema}\label{rema:boundedbicos}
    A bicomplex \(A\) is called \emph{bounded} if \(A^{p,q}=0\) for all but finitely many \((p,q)\in\ZZ\). If \(f:A\to B\) is a map between bounded bicomplexes, then \(f\) is a pluripotential weak equivalence if and only if it induces an isomorphism in cohomology with respect to \(\del\) and with respect to \(\delb\). 
\end{rema}

The category $\bicok$ admits a cofibrantly generated model structure,
where:
\begin{itemize}
    \item fibrations are surjective morphisms,
    \item weak equivalences are pluripotential weak equivalences,
    \item cofibrations are injective morphisms.
\end{itemize}

We will refer to this as the \textit{pluripotential model structure}. It is easy to see that every object is fibrant and cofibrant. The tensor product and internal hom make \(\bicok\) into a symmetric monoidal model
category. The notion of homotopy associated with this model structure is the following.

\begin{defi}
Let $f,g:A\to B$ be two maps of bicomplexes. A \emph{homotopy from $f$ to $g$} is a linear map $h:A\to B$ of bidegree $(-1,-1)$ 
satisfying 
\[[\del,[\delb,h]]=f-g.\]
If such a homotopy exists, we write \(f\simeq g\).
\end{defi}

Since we are working over a field, the class of homotopy equivalences induced from the above definition coincides with the class of pluripotential weak equivalences.

\begin{prop}\label{prop:HomotopicMapsAreEqualInHomology}
    Any two homotopic maps $f,g:A\to B$ of bicomplexes induce the same map in cohomology:
    \[[f]_{\Bb\Cc}=[g]_{\Bb\Cc}:\BC(A)\to \BC(B)\quad\text{ and }\quad[f]_{\Aa}=[g]_{\Aa}:\A(A)\to \A(B).\]
\end{prop}

We will use the following notion of shift functor $\ele: \bicok\to \bicok$.
    For a bicomplex $A$, let
    \[\ele(A):=\ele\otimes A,\]
    where
    \[\ele:=\begin{tikzcd}[ampersand replacement=\&]
	\kk \\
	\kk \& \kk,
	\arrow["1", from=2-1, to=1-1]
	\arrow["{-1}"', from=2-1, to=2-2]
\end{tikzcd}\]
and the bottom left corner is sitting in bidegree $(-1,-1).$ A homotopy inverse $\antiele: \bicok\to\bicok$ of $\ele$ is given by 
\[\antiele(A):=\antiele\otimes A,\]
where
\[\antiele:=\begin{tikzcd}[ampersand replacement=\&]
	\kk \& \kk \\
	\& \kk, \\
	\arrow["{-1}", from=1-1, to=1-2]
	\arrow["1"', from=2-2, to=1-2]
\end{tikzcd}\]
and the top right corner is sitting in bidegree $(1,1)$.

The following definition introduces the pluripotential analogue of the mapping cone for cochain complexes. It will be used to detect pluripotential weak equivalences.

\begin{defi}
    The \emph{mapping cone}  of a morphism of bicomplexes $f\colon A\to B$ is the bicomplex defined by
    \[C(f):=\ele(A)\oplus B,\]
    with differentials given by
    \begin{align*}
            &\del(a,a',a'',b)=(\del a,- \del a', - \del a''-a, \del b +f(a')),\\
            &\delb(a,a', a'',b)=(\delb a, -\delb a'+a,- \delb a'', \delb b+f(a'')).
    \end{align*}
    \label{cone}
\end{defi}
There is a short exact sequence 
\[B\to C(f)\to \ele A\]
and
 $f$ is a pluripotential weak equivalence if and only if $C(f)$ is acyclic: $\BC(C(f))=0$.

We next state Künneth formulae for Bott-Chern and Aeppli cohomologies. There are natural maps
\[
\BC(A) \otimes \BC(B) \longrightarrow 
\BC(A \otimes B)
\]
and
\[
\BC(A) \otimes \A(B)
\;\oplus\;
\A(A) \otimes \BC(B)
\longrightarrow
\A(A \otimes B).
\]

These maps are neither injective nor surjective in general.
Below, we provide Künneth-type formulas for \(\ABC\) cohomology, although the resulting isomorphisms are not natural.
For canonical descriptions of the kernels and cokernels of the maps above, see Theorem 1.35 in \cite{Ste25}.

\begin{prop}\label{prop:Kunneth}
    Let \(A\) and \(B\) be two bicomplexes. Then, there exists an isomorphism
    \[
    \ABC(A\otimes B)\cong\ABC(\ABC(A)\otimes\ABC(B)).
    \]
\end{prop}
\begin{proof}
    If \(A \simeq B\) is a pluripotential homotopy equivalence, then 
    \(A \otimes C \simeq B \otimes C\) as well. One may take a homotopy of the form 
    \(h \otimes \id_{C}\), where \(h\) is a homotopy witnessing \(A \simeq B\). Therefore, by Theorem 1.21 in \cite{Ste25}, for any choice of decomposition into squares and zig-zags, we obtain
    \[
    A\otimes B \simeq A_{{\mathrm{zig}}}\otimes B_{\mathrm{zig}}.
    \]
    Choosing an isomorphism \(\ABC(A)\cong A_{{\mathrm{zig}}}\) and \(\ABC(B)\cong B_{\mathrm{zig}}\) and using Proposition \ref{prop:HomotopicMapsAreEqualInHomology} we obtain an isomorphism
    \[
    \ABC(A\otimes B)\cong\ABC(\ABC(A)\otimes\ABC(B)).
    \qedhere\]
\end{proof}

\begin{rema}
    The proposition above relies on the fact that any bicomplex \(A\) splits (non-uniquely) as a 
    direct sum of squares and zig-zags,
    \[
    A = A_{{\mathrm{zig}}} \oplus A_{\mathrm{sq}}.
    \]
    The hypotheses required for such a decomposition can be weakened to bicomplexes in other 
    abelian categories, for instance, to bicomplexes of characteristic-zero representations of a 
    finite group \(G\), see Remark 5 in  \cite{Ste21}.  
    Consequently, Proposition~\ref{prop:Kunneth} also applies to bicomplexes in 
    \(\kk[\SS_n]\)-modules, a fact we will use later. 
\end{rema}

\subsection{The category of \texorpdfstring{$A^{pp}_\infty$}{ppA-infinity}-algebras}\label{subsec:AinftyAlgs}
We now introduce the category of pluripotential $A_\infty$-algebras. As we will see in Section~\ref{sec:koszulduality}, these arise as algebras over the pluripotential minimal model of the associative operad. 

Originally, \(A_\infty\)-algebras were introduced by Stasheff in \cite{Sta63} in his study of loop spaces and their higher homotopy associativity. We briefly recall the classical definition:
an \emph{$A_\infty$-algebra} consists of a cochain complex $(A, d)$ together with a family of linear maps
\[
\mu_n: A^{\otimes n} \to A
\]
of degree $2-n$ for every \(n\geq 2\), satisfying the following relation:
\[
[d, \mu_n] = \sum_{\substack{k + \ell = n + 1 \\1\leq r\leq k}} (-1)^{r(1 + \ell) + n} \mu_k \left( \id^{\otimes r} \otimes \mu_\ell \otimes \id^{\otimes s} \right)
\]
where \[[d,\mu_n]:=d\mu_n-\sum_i(-1)^n\mu_n(\id^{\otimes i-1}\otimes d\otimes \id^{n-i}).\]

For low arities we have the following relations:
\begin{itemize}[leftmargin=*]
    \item For \(n=2\) we have
    \[d\mu_2=\mu_2(d\otimes\id)+\mu_2(\id\otimes d)\]
    \item For \(n=3\) we have
    \[[d,\mu_3]=\mu_2(\mu_2\otimes\id)-\mu_2(\id\otimes\mu_2).\]
\end{itemize}

An \(A_\infty\)-algebra can be viewed as a chain complex equipped with a binary operation that is associative up to a coherent system of higher homotopies. In this sense, a pluripotential \(A_\infty\)-algebra is analogous, with ordinary homotopy replaced by pluripotential homotopy. Maybe surprisingly, these two definitions are quite different as one can see below.

\begin{defi}\label{defi:ppAinftyA}
A \emph{pluripotential \(A_\infty\)-algebra}, $\ppinfa$-algebra for short, consists of a bicomplex $(A,\del,\delb)$ together $n$-ary operations of bidegree $(p,q)$
\[
\mu_n^{p,q}: A^{\otimes n} \to A
\]
for every  \(n\geq 2\), where the operations are nonzero only in the following cases:
\begin{itemize}
    \item when $p+q=2-n$ and $p,q\leq 0$,
    \item when $p+q=1-n$ and $p,q\leq -1$.
\end{itemize}
These operations must satisfy the following relations:
\begin{equation}
    \tag{$\mathcal{A}^{p,q}_n$} \label{delmu}
    [\del,\mu_n^{p,q}]=q\mu_n^{p+1,q}+\sum_{\substack{k+\ell=n+1\\1\leq i\leq k}}\sum_{\substack{p_1+p_2=p+1\\q_1+q_2=q}}(-1)^{i(\ell+1)+n+\varepsilon}\mu_{k}^{p_1,q_1}(\id^{\otimes i-1}\otimes\mu_\ell^{p_2,q_2}\otimes\id^{\otimes k-i}),
\end{equation}
and
\begin{equation}
    \tag{$\mathcal{\overline{A}}^{p,q}_n$} 
    [\delb,\mu_n^{p,q}]=-p\mu_n^{p,q+1}+\sum_{\substack{k+\ell=n+1\\1\leq i\leq k}}\sum_{\substack{p_1+p_2=p\\q_1+q_2=q+1}}(-1)^{i(\ell+1)+n+\varepsilon}\mu_{k}^{p_1,q_1}(\id^{\otimes i-1}\otimes\mu_\ell^{p_2,q_2}\otimes\id^{\otimes k-i}),
\end{equation}
where $\varepsilon=(p_1+q_1)(\ell+p_2+q_2)+(p+q+n).$
\end{defi}

\begin{rema}
For a fixed $n$, there are $2n-3$ operations: $n-1$ of them have total degree $2-n$, and $n-2$ of them have total degree $1-n$. The relations can be represented as follows:
\begin{center}
\begin{tikzpicture}
    \matrix (m) [matrix of math nodes,
      nodes in empty cells,
      nodes={minimum width=5ex, minimum height=5ex, outer sep=-3pt},
      column sep=1ex, row sep=1ex]{
       \, & \, & \, &  \, & \, & \, & \, & \, & \, & \,\\
       \, & \, &  \, & \, & \, &\, &  \mu_\ell^{p_2,q_2} & \,  & \, & \,  \\
       \left[\del,\mu_n^{p,q}\right] & = \, &  q\mu_n^{p+1,q} & + & \mathlarger{\sum}\pm & \, & \, & \, & \, & \, \\
       \, & \, &  \, & \, & \, &\, &  \mu_{k}^{p_1,q_1} & \, & \, & \, \\
       \, & \, &  \, & \, & \, & \, & \, &  \, & \, & \, \\};

  % edges
  \draw[-] (m-4-7) -- (m-1-5);
  \draw[-] (m-1-6) -- (m-2-7);
  \draw[-] (m-1-7) -- (m-2-7);
  \draw[-] (m-1-8) -- (m-2-7);
  \draw[-] (m-1-9) -- (m-4-7);
  \draw[-] (m-1-10) -- (m-4-7);
  \draw[-] (m-2-7) -- (m-4-7);
  \draw[-] (m-4-7) -- (m-5-7);
  \draw[-] (m-1-2) -- (m-3-3);
  \draw[-] (m-1-3) -- (m-3-3);
  \draw[-] (m-1-4) -- (m-3-3);
  \draw[-] (m-3-3) -- (m-5-3);

  % overbraces inside the picture
  % Left brace above the inputs in columns 2--4 (row 1)
  \draw[decorate, decoration={brace, amplitude=5pt}]
    ($(m-1-2)+(0,3mm)$) -- node[above=4pt] {$n$} ($(m-1-4)+(0,3mm)$);

  % Right brace above the inputs in columns 6--10 (row 1)
  \draw[decorate, decoration={brace, amplitude=5pt}]
    ($(m-1-5)+(0,3mm)$) -- node[above=4pt] {$n$} ($(m-1-10)+(0,3mm)$);
\end{tikzpicture}
\end{center}
Here, the sum runs over all compositions of arity $n$ (as in $A_\infty$-algebras) and bidegree $(p+1,q)$. Notice that if $p+q=2-n$, the first tree in the picture vanishes, and all operations in the summation satisfy $p_i+q_i=2-n$ by degree reasons. On the other hand, if $p+q=1-n$, then in each composition in the summation, exactly one operation satisfies $p_i+q_i=1-n$.
\end{rema}

\begin{rema}
Note that the tuples \((A,\del,\{\mu_\bullet^{*,0}\})\) and \((A,\delb,\{\mu_\bullet^{0,*}\})\), arising from a \(\ppinfa\)-algebra \((A,\del,\delb,\{\mu_\bullet^{*,*}\})\), are \(A_\infty\)-algebras.
\end{rema}
 
\begin{exam}
A \emph{bidifferential bigraded algebra} is a $\ppinfa$-algebra with $\mu_n^{p,q} =0$ for all $n\geq 3$.
\end{exam}

\begin{exam} Let us break down the previous definition in the case of low arities.
\begin{itemize}[leftmargin=*]
    \item For $n=2$, there is a unique bidegree-preserving binary operation
\[
\mu_2 : A^{\otimes 2} \to A
\]
satisfying \[\del\mu_2=\mu_2(\del,\id)+\mu_2(\id,\del)\quad\text{ and }\quad\delb\mu_2=\mu_2(\delb,\id)+\mu_2(\id,\delb).\] This is equivalent to say that
 both $\del$ and $\delb$ are derivations with respect to $\mu_2$. 
 \item For $n=3$, we have three ternary operations
\[
\mu_3^{p,q} : A^{\otimes 3} \longrightarrow A,
\qquad
(p,q)\in\{(-1,0),\,(0,-1),\,(-1,-1)\},
\]
satisfying
\begin{align*}
\bigl[\del,\mu_3^{\mep1,0}\bigr] &= \mu_2(\mu_2\otimes \id)-\mu_2(\id\otimes \mu_2), 
&\quad
\bigl[\delb,\mu_3^{\mep1,0}\bigr] &= 0,\\[4pt]
\bigl[\del,\mu_3^{0,\mep1}\bigr] &= 0, 
&\quad
\bigl[\delb,\mu_3^{0,\mep1}\bigr] &= \mu_2(\mu_2\otimes \id)-\mu_2(\id\otimes \mu_2), \\[4pt]
\bigl[\del,\mu_3^{\mep1,\mep1}\bigr] &= -\,\mu_3^{0,\mep1}, 
&\quad
\bigl[\delb,\mu_3^{\mep1,\mep 1}\bigr] &= \mu_3^{\mep1,0}. 
\end{align*}
In particular, these imply
\[
\bigl[\del,\,\bigl[\delb,\mu_3^{\mep1,\mep1}\bigr]\bigr]
\;=\;
\mu_2(\mu_2\otimes \id)-\mu_2
(\id\otimes \mu_2).
\]
Thus, $\mu_2$ is associative up to a pluripotential homotopy given by $\mu_3^{\mep1,\mep1}$. Note that $\mu^{\mep1,0}_3$ and $\mu^{0,\mep1}_3$ are determined by $\mu^{\mep1,\mep1}_3$, in fact this holds in general, see the next remark for a detailed explanation.
\end{itemize}
\end{exam}
\begin{rema}\label{rema:dolentes-determinen-bones}
    Note that
    \[[\del,\mu_n^{p-1,q}]-[\delb,\mu_n^{p,q-1}]=(p+q)\mu_n^{p,q},\]
    for any \(n\geq 2\) and \(p,q\in\ZZ_{\leq 0}\). This means that, in arity \(n\), the \(n-2\) operations of total degree \(1-n\) determine the \(n-1\) operations of total degree \(2-n\).
\end{rema}

\begin{defi}
Let $(A,\{\mu_\bullet^{*,*}\})$ and $(B,\{\nu_\bullet^{*,*}\})$ be two $\ppinfa$-algebras. A \emph{$\ppinfm$-morphism} \(f:A\rightsquigarrow B\) is a family of linear maps 
\[
f_n^{p,q}: A^{\otimes n} \to B,\quad n\geq 2,
\]
of bidegree $(p,q)$
that are non-zero only in
 the following cases:
\begin{itemize}
    \item when $p+q=1-n$ and $p,q\leq 0$,
    \item when $p+q=-n$ and $p,q\leq -1$
\end{itemize}
such that
\begin{align*}
[\del, f_n^{p,q}] =& qf_n^{p+1,q} -\sum_{i_1+\cdots+i_k=n} \sum_{\substack{p_0+p_{i_1}+\dots+p_{i_k}=p+1\\q_0+q_{i_1}+\dots+q_{i_k}=q}}(-1)^{\theta(i_1,\dots,i_k)+\varepsilon_1}\, \nu^{p_0,q_0}_k\Bigl(f_{i_1}^{p_{i_1},q_{i_1}}\otimes \cdots \otimes f_{i_k}^{p_{i_k},q_{i_k}}\Bigr)\\[1mm]
&-\sum_{\substack{k+\ell=n+1\\ r+s+1=k}} \sum_{\substack{p_1+p_2=p+1\\q_1+q_2=q}}(-1)^{r(1+\ell)+k+\varepsilon_2}\, f_k^{p_1,q_1}\Bigl(\id^{\otimes r}\otimes\mu^{p_2 ,q_2}_\ell\otimes \id^{\otimes s}\Bigr),
\end{align*}
and
\begin{align*}
[\delb, f_n^{p,q}] =& -pf_n^{p,q+1}-\sum_{i_1+\cdots+i_k=n} \sum_{\substack{p_0+p_{i_1}+\dots+p_{i_k}=p\\q_0+q_{i_1}+\dots+q_{i_k}=q+1}}(-1)^{\theta(i_1,\dots,i_k)+\varepsilon_1}\, \nu^{p_0,q_0}_k\Bigl(f_{i_1}^{p_{i_1},q_{i_1}}\otimes \cdots \otimes f_{i_k}^{p_{i_k},q_{i_k}}\Bigr)\\[1mm]
&-\sum_{\substack{k+\ell=n+1\\ r+s+1=k}} \sum_{\substack{p_1+p_2=p\\q_1+q_2=q+1}}(-1)^{r(1+\ell)+k+\varepsilon_2}\, f_k^{p_1,q_1}\Bigl(\id^{\otimes r}\otimes\mu^{p_2 ,q_2}_\ell\otimes \id^{\otimes s}\Bigr).
\end{align*}
Here, 
\[\theta:= \theta(i_1,\dots,i_k):=\sum_{1\leq r<s\leq k}i_r(i_s+1),\]
\[\varepsilon_1 := \sum_{j=1}^k \left( \sum_{r=0}^{j-1}(p_r+q_r) \right)(p_j+q_j+i_j+1)+p+q+n+1,\]
and
\[\varepsilon_2=(p_1+q_1)(p_2+q_2+\ell)+p+q+n+1.\]
\end{defi}
\begin{rema}
    For every \(n\geq 1\) there are \(2n-1\) linear maps, \(n\) of total degree \(1-n\) and \(n-1\) of total degree \(-n\). 
\end{rema}
\begin{exam}
Again, let us break down the definition of $\ppinfm$-morphisms in the case of low arities.
\begin{itemize}[leftmargin=*]
    \item For $n=1$, we have a single degree-preserving map
    \[
    f_1 : A \to B,
    \]
    such that
    \[
    [\del, f_1] = 0 \quad\text{and}\quad [\delb, f_1] = 0.
    \]
    In other words, $f_1$ is a map of bicomplexes.
    \item For $n=2$, we have three linear maps
\[
f_2^{p,q} : A^{\otimes 2} \to B
\]
of bidegrees $(-1,0)$, $(0,-1)$, and $(-1,-1)$ satisfying
\begin{align*}
[\del, f_2^{\mep1,0}] &= -\mu_2^B\bigl(f_1\otimes f_1\bigr) + f_1\bigl(\mu_2^A\bigr),\\[1mm]
[\del, f_2^{0,\mep1}] &= 0,\\[1mm]
[\del, f_2^{\mep1,\mep1}] &= -f_2^{0,\mep1},
\end{align*}
and similarly,
\begin{align*}
[\delb, f_2^{\mep1,0}] &= 0,\\[1mm]
[\delb, f_2^{0,\mep1}] &= -\mu_2^B\bigl(f_1\otimes f_1\bigr) + f_1\bigl(\mu_2^A\bigr),\\[1mm]
[\delb, f_2^{\mep1,\mep1}] &= f_2^{\mep1,0}.
\end{align*}
In particular, we have
\[
[\del,[\delb, f_2^{\mep1,\mep1}]] = -\mu_2^B\bigl(f_1\otimes f_1\bigr) + f_1\bigl(\mu_2^A\bigr).
\]
\end{itemize}
\end{exam}
\begin{rema}
    Note again that
    \[
    [\del,f_n^{p-1,q}]+[\delb,f_n^{p,q-1}]=(p+q)f_n^{p,q},
    \]
    so, in arity \(n\), the \(n-1\) linear maps of total degree \(-n\) determine the \(n\) linear maps of total degree \(1-n\).
\end{rema}

\begin{defi}\label{defi:inftymorph}
A $\ppinfm$-morphism $f : A \rightsquigarrow B$ is called a \emph{\(\ppinfm\)-weak equivalence} if $f_1$ is a pluripotential weak equivalence. We say that $f$ is \emph{strict} if $f_n = 0 \text{ for all } n \ge 2.$ The \emph{identity $\ppinfm$-morphism} is the strict morphism with $f_1 = \id_A$.
\end{defi}

The composition of two $\infty$-morphisms $g : A \rightsquigarrow B$ and $f : B \rightsquigarrow C$, is given by
\[
(f\circ g)_n = \sum_{i_1+\dots+i_k=n}\sum_{\substack{p_0+\dots+p_k=p\\q_0+\dots+q_k=q}}(-1)^{\theta(i_1,\dots, i_k)+\varepsilon}\, f_k^{p_0,q_0}\Bigl(g_{i_1}^{p_1,q_1}\otimes \cdots \otimes g_{i_k}^{p_k,q_k}\Bigr),
\]
where the sign ${\theta(i_1,\dots, i_k)}$ is defined as in the $\ppinfm$-morphism relations and
\[\varepsilon:=\sum_{j=1}^k \left( \sum_{r=0}^{j-1}(p_r+q_r) \right)(p_j+q_j+i_j+1).\]
This composition is associative, since \(\ppinfm\)-morphisms are precisely morphisms between the images of the source and the target under the bar construction \(\br_\iota\), see Section~\ref{subsec:BarCobar}.

\subsection{Homotopy transfer theorem}\label{subsec:HTT}

In this section, we prove that every bicomplex admits a contraction into a minimal bicomplex. Furthermore, we show that any bba induces a pluripotential $A_\infty$-algebra structure on any homotopy equivalent bicomplex, with explicit formulas.

\begin{defi}A \emph{contraction} of a bicomplex
 $(A,\del,\delb)$ into a bicomplex  $(B,\del,\delb)$ is a diagram
\[
\begin{tikzcd}[ampersand replacement = \&]
(A,\del,\delb) \arrow[r, rightarrow, shift left, "f"] \arrow[r, shift right, leftarrow, "  g" swap] \arrow[loop left, "h"] \& (B,\del,\delb)
\end{tikzcd}
\]
where $f$ and $  g$ are bicomplex morphisms such that
 $fg=\id_B$ and
    $h$ is a pluripotential homotopy from the \(gf\) to the identity:
      \[g f-\id_A=[\del,[\delb,h]].\]
\end{defi}

\begin{prop}\label{prop:MinimalContraction}
Every bicomplex  $(A,\del,\delb)$ admits a contraction 
into a minimal bicomplex:
\[
\begin{tikzcd}[ampersand replacement = \&]
(A,\del,\delb) \arrow[r, rightarrow, shift left, " f"] \arrow[r, shift right, leftarrow, " g" swap] \arrow[loop left, "h"] \& (H,\del,\delb)
\end{tikzcd};\, \text{ with } \del\delb=0\text{ on }H.
\]
\end{prop}

\begin{proof} 
By Section \ref{subsec:preliminarsbico}, we can decompose (non-uniquely) every bicomplex $A$ as \[A= A_{\mathrm{sq}}\oplus A_{\mathrm{zig}}.\]
Let 
\(H:=A_{\mathrm{zig}}\simeq A.\)
We define $ f:A\twoheadrightarrow H$ and $ g:H \hookrightarrow A$ via the decomposition. Next, we define the homotopy
 \[h:A_{\mathrm{sq}}\oplus A_{\mathrm{zig}}\to A_{\mathrm{sq}}\oplus A_{\mathrm{zig}}\] as follows. Given a pair $(a,b)\in  A_{\mathrm{sq}}\oplus A_{\mathrm{zig}}$,
we let $h(a,b)=(a',0)$ if $a=-\del\delb a'$ and zero otherwise. Note that since we are considering a splitting into squares and zig-zags, $\del$ and $\delb$ are isomorphisms, so $a'$ is uniquely determined. This homotopy satisfies \[g f-\id_A=[\del,[\delb,h]].\]

First, observe that for any element $ (a, b) \in A_{\text{sq}} \oplus A_{\text{zig}} $, we have:
\[
(g f-\id_A)(a, b) = -(a, 0).
\]
Next, note that \[[\del,[\delb,h]](A_{\mathrm{zig}})=0,\] by definition.
Now, consider the case when $a\in A_{\mathrm{sq}}$, we will compute
\[
[\del, [\delb, h]]a = (\del \delb h - \del h \delb + \delb h \del - h \delb \del)a.
\]
Assume $a$ is a degree-homogeneous element in a square, as depicted below:

% https://q.uiver.app/?q=WzAsNCxbMSwwLCJcXGRlbFxcZGVsYlxcYWxwaGEiXSxbMSwxLCJcXGRlbFxcYWxwaGEiXSxbMCwxLCJcXGFscGhhIl0sWzAsMCwiXFxkZWxiIFxcYWxwaGEiXSxbMiwxLCJcXGRlbCIsMl0sWzEsMCwiXFxkZWxiIiwyXSxbMiwzLCJcXGRlbGIiXSxbMywwLCJcXGRlbCJdXQ==
\[\begin{tikzcd}
	{\delb \alpha} & \del\delb\alpha \\
	\alpha & \del\alpha
	\arrow["\del"', from=2-1, to=2-2]
	\arrow["\delb"', from=2-2, to=1-2]
	\arrow["\delb", from=2-1, to=1-1]
	\arrow["\del", from=1-1, to=1-2]
\end{tikzcd}\]
There are four possible cases for $a$ based on its position in the square:
\begin{enumerate}
    \item If $a$ is the top right corner of a square, so that $a=\del\delb\alpha$, then the only non zero term is \[\del\delb h a= -\del\delb \alpha= -a.\]
    \item If $a$ is the bottom right corner of a square, so that $a=\del\alpha$, then the only non zero term is \[-\del h\delb a=-\del h \delb\del \alpha=-\del(\alpha)=-a.\]
    \item If $a$ is the top left corner of a square, so that $a=\delb \alpha$, then \[\delb h\del a=\delb h\del\delb \alpha=-\delb\alpha=-a\] is the only non-zero term.
    \item If $a$ is the bottom left corner of a square, then
    \[-h\delb\del a=-\alpha=-a\] is the only non-zero term. 
\end{enumerate}

In all cases, we obtain $[\del, [\delb, h]]a = -a$. Therefore, we conclude
\[
[\del, [\delb, h]](a, b) = -(a, 0),
\]
which completes the proof.
\end{proof}

Let $(A,\del,\delb,\mu)$ be a bidifferential bigraded algebra. For the rest of the section, fix a contraction (not necessarily minimal)
    \[
    \begin{tikzcd}[ampersand replacement = \&]
    (A,\del,\delb) \arrow[r, rightarrow, shift left, "f"] \arrow[r, shift right, leftarrow, "g" swap] \arrow[loop left, "h_{11}"] \& (B,\del,\delb)
    \label{transferdata}
    \tag{$\Dd$}
    \end{tikzcd}
    \]
and define $h_{10}:=[\delb, h_{11}]$ and $h_{01}:=-[\del, h_{11}]$. Note that these maps are chain homotopies with respect to $\del$ and $\delb$, respectively: 
\[\id_A-gf=[\del,h_{10}]\quad\text{ and }\quad\id_A-gf=[\delb,h_{10}].\]
One can define a binary operation $\nu_2:B\otimes B\to B$ by the formula \[\nu_2:=f\mu(g,g).\]
This operation is only associative up to pluripotential homotopy. Indeed, 
the ternary operation $\nu_3^{\text{-}1,\text{-}1}:B^{\otimes 3}\to B$, defined by the formula 
\[\nu_3^{\text{-}1,\text{-}1}:=f\mu(h_{11}\mu(g,g),g)-f\mu(g,h_{11}\mu(g,g)),\]
satisfies the relation
\[\nu_2(\nu_2,\id)-\nu_2(\id,\nu_2)=[\del,[\delb,\nu_3^{\text{-}1,\text{-}1}]].\]

\begin{rema} This operation satisfies
$\delb \nu_3^{\text{-}1,\text{-}1}-\nu_3^{\text{-}1,\text{-}1}\delb=\nu_3^{\text{-}1,0}$ and $\del \nu_3^{\text{-}1,\text{-}1}-\nu_3^{\text{-}1,\text{-}1}\del=-\nu_3^{0,\text{-}1}$, where 
\[
\nu_3^{\text{-}1,0}=f\mu(h_{10}\mu( g, g), g)-f\mu( g,h_{10}\mu( g, g)),
\] 
and
\[
\nu_3^{0,\text{-}1}=f\mu(h_{01}\mu( g, g), g)-f\mu( g,h_{01}\mu( g, g)).
\]
\end{rema}

In general, one can define the following abstract maps associated to a contraction, called $\p$\emph{-kernels}, which will be used to transfer the algebraic structure of a bba through a minimal contraction.

The associated \emph{$\p$-kernels} are the linear maps
\[\{\mup_n^{p,q}:A^{\otimes n}\to A\}_{n\geq 2}\]
of bidegree $(p,q)$
where
\begin{itemize}
    \item $p+q=2-n$ and $p,q\leq 0$, or
    \item $p+q=1-n$ and $p,q\leq -1$.
\end{itemize} 
These are defined recursively by
\begin{equation*}
\mup_n^{p,q}:=\sum_{\substack{k+\ell=n\\k,\ell\geq 1}}\sum_{\substack{p_1+p_2-\alpha_1-\alpha_2=p\\q_1+q_2-\beta_1-\beta_2=q}}(-1)^{k(\ell+1)+\varepsilon}\mu(h_{\alpha_1\beta_1}\mup_k^{p_1,q_1}\otimes h_{\alpha_2\beta_2}\mup_\ell^{p_2,q_2}),
    \label{mup}
\end{equation*}
where $\varepsilon=(\alpha_1+\beta_1)(d_1+k)+(\alpha_1+\beta_1+d_1)(\alpha_{2}+\beta_{2}+1)+(\alpha_1+\beta_1+d_1+\alpha_2+\beta_2)(d_2+\ell)$ and \(d_x=p_x+q_x\) for \(x=1,2\). We formally write $h_{11}\mup_1=\id$, $h_{01}\mup_1=0$ and $h_{10}\mup_1=0$.

\begin{rema}\label{rema:TransferTrees}
These maps can be defined as linear combinations of planar binary trees with $n$ leaves, where each internal edge is labeled by $h_{10}$, $h_{01}$ or $h_{11}$. Due to degree constraints, only one leaf decorated with $h_{11}$ is permitted per tree. Explicitly, let $\text{PBT}_n$ the set of all planar binary trees with $n$ leaves. Then, we have 
    \[\mup_n^{p,q}=\sum_{t\in \text{PBT}_n}\pm\mup_t^{p,q},\]
where
\begin{itemize}
    \item if $p+q=2-n$, $\mup_t^{p,q}$ consists of the sum (with signs) of $\frac{(n-2)!}{|p|!|q|!}$ copies of $t$ such that each copy is a different way of decorating the internal edges with $|p|$ homotopies $h_{10}$ and $|q|$ homotopies $h_{01}$,
    \item  if $p+q=1-n$, $\mup_t^{p,q}$ consists of the sum (with signs) of $\frac{(n-2)!}{(|p|-1)!(|q|-1)!}$ copies of $t$ such that each copy is a different way of decorating the internal edges with $|p|-1$ homotopies $h_{10}$, $|q|-1$ homotopies $h_{01}$ and one homotopy $h_{11}$.
\end{itemize}
\end{rema}

\begin{exam} Below, the operation \(\mup_4^{-2,-1}\) is depicted as a sum of planar binary trees with \(n\) leaves:

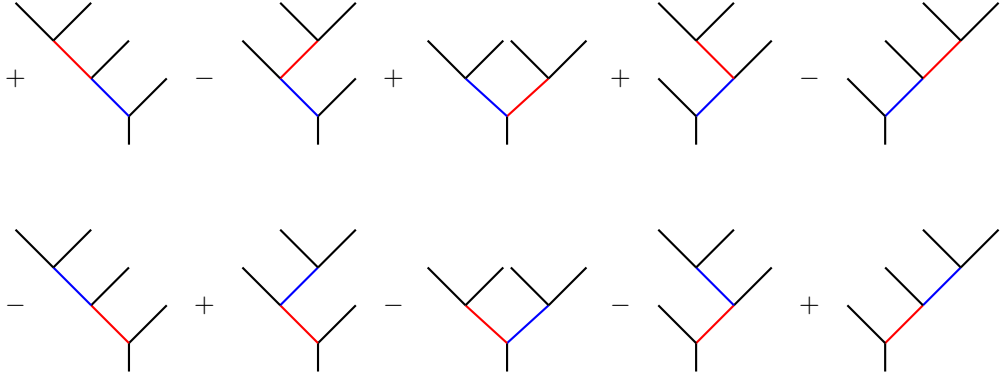
\vspace{10pt}

\begin{center}
\begin{tikzpicture}

% ROW 1

\node at (-1,1) {\(+\)};
% Tree 1
\node at (0.5,0) {} [grow = up, level distance = 0.5cm, sibling distance = 1cm, thick]
    child{
        child {}
        child[edge from parent/.style={draw=blue}]{
            child[edge from parent/.style={draw=black}]{}
            child[edge from parent/.style={draw=red}]{
                child[edge from parent/.style={draw=black}]{}
                child[edge from parent/.style={draw=black}]{}
            }
        }
    };
\node at (1.5,1) {\(-\)};

% Tree 2
\node at (3,0) {} [grow = up, level distance = 0.5cm, sibling distance = 1cm, thick]
    child{
        child {}
        child[edge from parent/.style={draw=blue}]{
            child[edge from parent/.style={draw=red}]{                
                child[edge from parent/.style={draw=black}]{}
                child[edge from parent/.style={draw=black}]{}
            }
            child[edge from parent/.style={draw=black}]{}
        }
    };
\node at (4,1) {\(+\)};

% Tree 3
\node at (5.5,0) {} [
    grow = up,
    level distance = 0.5cm,
    sibling distance = 1cm,
    thick
]
child{
    child[xshift=0.05cm,edge from parent/.style={draw=red}] {                
        child[edge from parent/.style={draw=black}]{}
        child[edge from parent/.style={draw=black}]{}
    }
    child[xshift=-0.05cm,edge from parent/.style={draw=blue}] {
        child[edge from parent/.style={draw=black}]{}
        child[edge from parent/.style={draw=black}]{}
    }
};
\node at (7,1) {\(+\)};

% Tree 4
\node at (8,0) {} [grow = up, level distance = 0.5cm, sibling distance = 1cm, thick]
    child{
        child[edge from parent/.style={draw=blue}]{
            child[edge from parent/.style={draw=black}]{}
            child[edge from parent/.style={draw=red}]{
                child[edge from parent/.style={draw=black}]{}
                child[edge from parent/.style={draw=black}]{}
            }
        }
        child[edge from parent/.style={draw=black}]{}
    };
\node at (9.5,1) {\(-\)};

% Tree 5
\node at (10.5,0) {} [grow = up, level distance = 0.5cm, sibling distance = 1cm, thick]
    child{
        child[edge from parent/.style={draw=blue}]{
            child[edge from parent/.style={draw=red}]{
                child[edge from parent/.style={draw=black}]{}
                child[edge from parent/.style={draw=black}]{}
            }
            child[edge from parent/.style={draw=black}]{}
        }
        child{}
    };

% ROW 2

\node at (-1,-2) {\(-\)};

% Tree 6
\node at (0.5,-3) {} [grow = up, level distance = 0.5cm, sibling distance = 1cm, thick]
    child{
        child {}
        child[edge from parent/.style={draw=red}]{
            child[edge from parent/.style={draw=black}]{}
            child[edge from parent/.style={draw=blue}]{
                child[edge from parent/.style={draw=black}]{}
                child[edge from parent/.style={draw=black}]{}
            }
        }
    };
\node at (1.5,-2) {\(+\)};

% Tree 7
\node at (3,-3) {} [grow = up, level distance = 0.5cm, sibling distance = 1cm, thick]
    child{
        child {}
        child[edge from parent/.style={draw=red}]{
            child[edge from parent/.style={draw=blue}]{                
                child[edge from parent/.style={draw=black}]{}
                child[edge from parent/.style={draw=black}]{}
            }
            child[edge from parent/.style={draw=black}]{}
        }
    };
\node at (4,-2) {\(-\)};

% Tree 8
\node at (5.5,-3) {} [
    grow = up,
    level distance = 0.5cm,
    sibling distance = 1cm,
    thick
]
child{
    child[xshift=0.05cm,edge from parent/.style={draw=blue}] {                
        child[edge from parent/.style={draw=black}]{}
        child[edge from parent/.style={draw=black}]{}
    }
    child[xshift=-0.05cm,edge from parent/.style={draw=red}] {
        child[edge from parent/.style={draw=black}]{}
        child[edge from parent/.style={draw=black}]{}
    }
};
\node at (7,-2) {\(-\)};

% Tree 9
\node at (8,-3) {} [grow = up, level distance = 0.5cm, sibling distance = 1cm, thick]
    child{
        child[edge from parent/.style={draw=red}]{
            child[edge from parent/.style={draw=black}]{}
            child[edge from parent/.style={draw=blue}]{
                child[edge from parent/.style={draw=black}]{}
                child[edge from parent/.style={draw=black}]{}
            }
        }
        child[edge from parent/.style={draw=black}]{}
    };
\node at (9.5,-2) {\(+\)};

% Tree 10
\node at (10.5,-3) {} [grow = up, level distance = 0.5cm, sibling distance = 1cm, thick]
    child{
        child[edge from parent/.style={draw=red}]{
            child[edge from parent/.style={draw=blue}]{
                child[edge from parent/.style={draw=black}]{}
                child[edge from parent/.style={draw=black}]{}
            }
            child[edge from parent/.style={draw=black}]{}
        }
        child{}
    };

\end{tikzpicture}
\end{center}

Here, the edges that are colored in red are labeled by the homotopy \(h_{10}\) and the edges that are colored in blue are labeled by the homotopy \(h_{11}\).

\end{exam}

\begin{lemm}\label{lemm:HTT}
Let $(A,\del,\delb,\mu)$ be a bidifferential bigraded algebra. 
Given a contraction     
\[
    \begin{tikzcd}[ampersand replacement = \&]
    (A,\del,\delb) \arrow[r, rightarrow, shift left, "f"] \arrow[r, shift right, leftarrow, "g" swap] \arrow[loop left, "h_{11}"] \& (B,\del,\delb),
    \end{tikzcd}
\]
define, for all $n\geq 2$,
\[\nu_n^{p,q}:=f\mup_n^{p,q} g^{\otimes n}\quad\text{ and }\quad
    g_n^{p,q}:=\sum_{\substack{i-\alpha=p\\j-\beta=q}}(-1)^{(\alpha+\beta)(i+j+n)} h_{\alpha\beta} \mup_n^{i,j}g^{\otimes n}.
\]
Then, $(B,\del,\delb,\{\nu_\bullet^{*,*}\})$ is a $\ppinfa$-algebra, and $g=(g,\{g_\bullet^{*,*}\})$
is a $\ppinfm$-weak equivalence $g:B\rightsquigarrow A$.
\end{lemm}
\begin{proof}
    In the proof of the analogous result for the classical case of $A_\infty$-algebras, Markl \cite{Mar06} states: “A straightforward but awfully technical induction shows that...” However, the corresponding awfully technical induction is not explicitly presented in the paper. As one might expect, the proof of Lemma~\ref{lemm:HTT} is even worse. But since I’m just a PhD student, I had no choice but to write out the proof, which can be found in Appendix~\ref{app:induc}.
\end{proof}

\begin{rema}\label{rema:formules-HTT-using-g}
    Note that, by the definition of \(\mup_n^{p,q}\), we can rewrite \(\nu_n^{p,q}\) and \(g_n^{p,q}\) as
    \[
    \nu_n^{p,q}=\sum_{k+\ell=n}\sum_{\substack{p_1+p_2=p\\q_1+q_2=q}}(-1)^{k(\ell+1)+d_1(d_2+\ell+1)}f\mu(g_k^{p_1,q_1}\otimes g_\ell^{p_2,q_2}),
    \]
    and
    \[
    g_n^{p,q}=\sum_{k+\ell=n}\sum_{\substack{p_1+p_2-\alpha=p\\q_1+q_2-\beta=q}}(-1)^{\varepsilon} h_{\alpha\beta}\mu(g_k^{p_1,q_1}\otimes g_\ell^{p_2,q_2}),
    \]
    where \(\varepsilon=k(\ell+1)+d_1(d_2+\ell+1)+(\alpha+\beta)(d_1+d_2+k+\ell)\) and \(d_x=p_x+q_x\) for \(x=1,2\). This expression will be used to prove the commutative version of the above lemma.
\end{rema}

We now state the main result of this section.
\begin{theo}[Homotopy Transfer Theorem]\label{theo:HTT}
    Let $(A,\del,\delb,\mu)$ be a bidifferential bigraded algebra.
\begin{enumerate}
    \item There is a contraction into a minimal bicomplex
\[
\begin{tikzcd}[ampersand replacement = \&]
(A,\del,\delb) \arrow[r, rightarrow, shift left, "f"] \arrow[r, shift right, leftarrow, " g" swap] \arrow[loop left, "h"] \& (H,\del,\delb)
\end{tikzcd};\, \text{ with } \del\delb=0\text{ on }H.
\]
\item There is a $\ppinfa$-structure on $(H,\del,\delb)$ and a $\ppinfm$-weak equivalence $ g:H\rightsquigarrow A$.
\end{enumerate}
\end{theo}
\begin{proof}
    This is now a direct consequence of Proposition~\ref{prop:MinimalContraction} together with Lemma~\ref{lemm:HTT}.
\end{proof}

\begin{rema}
    Since a bba is, in particular, a bicomplex, we can consider its Bott-Chern and Aeppli cohomologies. The product induces an associative algebra structure on the Bott-Chern cohomology and equips the Aeppli cohomology with the structure of a module over the Bott-Chern cohomology. 
    
    The bicomplex \(H\) of the above theorem is isomorphic to the ABC-cohomology \[\ABC(A)=\BC(A)\oplus\Tilde{\A}(A),\] see Section~\ref{subsec:preliminarsbico}. Consequently, the induced product does not make the ABC-cohomology into a bba. However, Theorem~\ref{theo:HTT} provides a way to endow this cohomology with an algebra structure that is not strictly associative, but associative up to homotopy.

    If the bba \(A\) satisfies the \(\del\delb\)-property, then the bicomplex \(H\) has trivial differentials. Hence, by Remark~\ref{rema:dolentes-determinen-bones}, we have \(\mu_n^{p,q}=0\) whenever \(p+q=2-n\). Moreover, in this case the bicomplex \(H\) is isomorphic to \(\BC(A)\), and the product \(\mu_2\) of Theorem~\ref{theo:HTT} gives precisely the induced product on cohomology.
\end{rema}

\subsection{Commutative algebras}

We next define the analogues of commutative algebras in the setting of \(\ppinfa\)-algebras, i.e. \(\ppinfc\)-algebras, and we show that if we start with a commutative bidifferential bigraded algebra, then the formulas for the transferred \(\ppinfa\)-structure of Lemma~\ref{lemm:HTT} give a \(\ppinfc\)-structure.

To define \(\ppinfc\)-structures, we first recall the definition of shuffle permutations: A \emph{\((i,j)\)-shuffle} is an element \(\sigma\) of the symmetric group \(\SS_{i+j}\) such that
\[\sigma(1)<\sigma(2)<\dots<\sigma(i)\quad\text{ and }\quad \sigma(i+1)<\sigma(i+2)<\dots<\sigma(i+j)\] hold.

Denote by \(\mathrm{Sh}_{i,j}\) the set of \((i,j)\)-shuffles. Define the element \(\tau_{i,j}\in \kk[\SS_{i+j}]\) by
\[\tau_{i,j}=\sum_{\sigma\in\mathrm{Sh}_{i,j}}\mathrm{sign}(\sigma)\sigma.\]

\begin{defi}
    A \emph{pluripotential \(C_\infty\)-algebra}, \(C^{pp}_\infty\)-algebra for short, is a \(A_\infty^{pp}\)-algebra such that 
    \[\mu_{i+j}^{p,q}\circ \tau_{i,j}=0\]
    for every \(i,j\geq 1\) and every \(p,q\in\ZZ_{\leq 0}\).
\end{defi}

Here, \(\tau_{i,j}\) acts on \(A^{\otimes i+j}\) by permuting the tensor factors. A \(\ppinfc\)-algebra is then a \(\ppinfa\)-algebra satisfying commutativity relations. Likewise, a \(\ppinfm\)-morphism of \(\ppinfc\)-algebras is required to satisfy the same commutativity relations.

\begin{defi}
    Let $(A,\{\mu_\bullet^{*,*}\})$ and $(B,\{\nu_\bullet^{*,*}\})$ be two $C_\infty^{pp}$-algebras. A \emph{$\ppinfm$-morphism of \(\ppinfc\)-algebras} \(f:A\rightsquigarrow B\) is a \(\ppinfm\)-morphism of \(\ppinfa\)-algebras such that \[f_{i+j}^{p,q}\circ \tau_{i,j}=0\]
    for every \(i,j\geq 1\) and every \(p,q\in\ZZ_{\leq 0}\).
\end{defi}

We next show that when the initial algebra is commutative, then the transferred structure of Lemma~\ref{lemm:HTT} is, in fact, a \(\ppinfc\)-structure.
\begin{lemm}\label{lemm:HTTCom}
    Let $(A,\del,\delb,\mu)$ be a commutative bidifferential bigraded algebra. 
Given a contraction     
\[
    \begin{tikzcd}[ampersand replacement = \&]
    (A,\del,\delb) \arrow[r, rightarrow, shift left, "f"] \arrow[r, shift right, leftarrow, "g" swap] \arrow[loop left, "h"] \& (B,\del,\delb),
    \end{tikzcd}
\]
the formulas defined in Lemma~\ref{lemm:HTT}
\[\nu_n^{p,q}:=f\mup_n^{p,q} g^{\otimes n}\quad\text{ and }\quad
    g_n^{p,q}:=\sum_{\substack{i-\alpha=p\\j-\beta=q}}(-1)^{(\alpha+\beta)(i+j+n)} h_{\alpha\beta} \mup_n^{i,j}g^{\otimes n},
\]
for all $n\geq 2$, define a $\ppinfc$-structure on \(B\) and a \(\ppinfm\)-weak equivalence of \(\ppinfc\)-algebras $g\colon B\rightsquigarrow A$.
\end{lemm}
\begin{proof}
The proof is parallel to the proof of Theorem 12 of \cite{GeXue08} for \(C_\infty\)-algebras in the dg setting. We reproduce it here for the sake of completeness. By Lemma \ref{lemm:HTT}, \(\{\nu_\bullet^{*,*}\}\) defines a \(\ppinfa\)-algebra structure on \(B\) and \(g=(g,\{g_\bullet^{*,*}\})\) defines a \(\ppinfm\)-morphism. It just remains to prove the symmetric conditions
\[
\nu_{i+j}^{p,q}\circ \tau_{i,j}=0,
\quad\text{and}\quad
g_{i+j}^{p,q}\circ \tau_{i,j}=0.
\]
First of all, recall that, by Remark~\ref{rema:formules-HTT-using-g}, we can rewrite \(\nu_n^{p,q}\) and \(g_n^{p,q}\) as
\[
\nu_n^{p,q}=\sum_{k+\ell=n}\sum_{\substack{p_1+p_2=p\\q_1+q_2=q}}(-1)^{k(\ell+1)+d_1(d_2+\ell+1)}f\mu(g_k^{p_1,q_1}\otimes g_\ell^{p_2,q_2}),
\]
and
\[
g_n^{p,q}=\sum_{k+\ell=n}\sum_{\substack{p_1+p_2-\alpha=p\\q_1+q_2-\beta=q}}(-1)^{k(\ell+1)+d_1(d_2+\ell+1)+(\alpha+\beta)(d_1+d_2+k+\ell)} h_{\alpha\beta}\mu(g_k^{p_1,q_1}\otimes g_\ell^{p_2,q_2}),
\]
where \(d_x=p_x+q_x\) for \(x=1,2\).

Let \(\nabla_2:T(B)\to T^2(TB)\) be the map given by
\[
a_1\otimes\dots\otimes a_{n}\mapsto\sum_{k+\ell=n}(a_1\otimes\dots\otimes a_k)\otimes (a_{k+1}\otimes\dots\otimes a_{k+\ell}).
\]
Define the shuffle product
\[
\shuffle:T(B)\otimes T(B)\to T(B)
\]
by 
\[
(a_1, \dots, a_i) \otimes (a_{i+1},\dots, a_{i+j})\mapsto \tau_{i,j}(a_1,\dots,a_i, a_{i+1}, \dots, a_{i+j}).
\]
By Lemma 13 in \cite{GeXue08}, we have
\[
\nabla_2(TB \shuffle TB)
\subset
\tau_{1,1}(T^2(TB))
\;\oplus\;
TB\;\otimes\; (TB \shuffle TB)\;\oplus\; (TB \shuffle TB)\;\otimes\; TB.
\]

We will show by induction on \(n\geq 2\) that \(\nu_n^{p,q}\) and \(g_n^{p,q}\) vanish on \(TB\shuffle TB\), which implies \(\nu_{i+j}^{p,q}\tau_{i,j}=0\) and \(g_{i+j}^{p,q}\tau_{i,j}=0\) for any \(i,j\geq 1\) and any \(p,q\in\ZZ_{\leq 0}\).

On the one hand, by induction hypothesis, the map  
\[
\sum_{k+\ell=n}\sum_{\substack{p_1+p_2=i\\q_1+q_2=j}}g_k^{p_1,q_1}\otimes g_\ell^{p_2,q_2}
\]
annihilates
\[
TB\;\otimes\; (TB \shuffle TB)\;\oplus\; (TB \shuffle TB)\;\otimes\; TB.
\]
On the other hand,
\[
\sum_{k+\ell=n}\sum_{\substack{p_1+p_2=i\\q_1+q_2=j}}g_k^{p_1,q_1}\otimes g_\ell^{p_2,q_2}
\]
takes the subspace \(\tau_{1,1}T^2(TB)\) to \(\tau_{1,1}(T^2(A))\), but this is annihilated by \(\mu\), and we conclude that \(\nu_n^{p,q}(TB\shuffle TB)\) and \(g_n^{p,q}(TB\shuffle TB)\) vanish. 
\end{proof}

As a consequence, we obtain the following commutative version of Theorem~\ref{theo:HTT}.
\begin{theo}\label{theo:HTTcom}
    Let $(A,\del,\delb,\mu)$ be a commutative bidifferential bigraded algebra. There exists a minimal bicomplex \(H\) endowed with a $\ppinfc$-structure and a \(\ppinfm\)-weak equivalence of \(\ppinfc\)-algebras $g\colon H\rightsquigarrow A$.
\end{theo}

\subsection{Massey products in complex geometry}
In their study of the topology of complex manifolds, 
 Angella and Tomasini  \cite{AnTo15} introduced triple Massey-type products for Bott-Chern and Aeppli cohomologies (ABC-Massey products). These Massey products turn out to be obstructions to the notions of formality related to Bott-Chern and Aeppli cohomologies introduced in \cite{MiSte24}. In this section, we review these concepts, present important examples and describe an explicit relation between these products and the \(\ppinfa\)-algebra structure arising from the homotopy transfer theorem.

\medskip

Let \((A,\del,\delb,\cdot)\) be a bidifferential bigraded algebra. Then, $\BC(A)$ has an algebra structure induced by the product of the algebra and $\A(A)$ has the structure of a $\BC$-module.

We next recall the Aeppli-Bott-Chern-Massey (ABC-Massey) products defined in \cite{AnTo15}.

\begin{defi}Let $[x], [y], [z]\in  \BC (A)$ be three Bott-Chern cohomology classes such that
 $[x]\cdot[y] = [y]\cdot[z] = 0$. The \textit{triple ABC-Massey product $\langle [x],[y],[z] \rangle$} is defined as the following set of Aeppli classes:
\[\langle [x],[y],[z] \rangle := \{[x\cdot v - u\cdot z]\,|\,\del\delb u = x\cdot y,\, \del\delb v = y\cdot z\} \subset \A(A).\]
The triple ABC-Massey product is said to be \textit{trivial} if $0\in \langle [x],[y],[z] \rangle$.
\end{defi}

Let us remark that in \cite{AnTo15} a different sign convention is used. We are sticking to the sign convention in \cite{MiSte24} instead. Also, in these references, the triple product is defined as an element in the coset space 
\[
\frac{\A(A)}{{[x]\cdot \A(A)+[z]\cdot \A(A)}}
\]
rather than as a set of Aeppli classes. Both viewpoints are equivalent. In particular, a triple ABC-Massey product is trivial if and only if it is zero in the above coset space.

When applied to the complexified de Rham algebra of a complex manifold, ABC-Massey products yield new holomorphic invariants with a homotopy–theoretic flavor. Moreover, in \cite{MiSte24}, these Massey-type operations were shown to obstruct the notions of \emph{strong} and \emph{weak formality} associated with pluripotential weak equivalences.

Interestingly, ABC-Massey products need not vanish in many situations where ordinary 
Massey products do.  
The standard proof of formality for compact Kähler manifolds of \cite{DGMS75}
relies on the $\del\delb$-property, but
Sferruzza and Tomassini \cite{SfeTo22} constructed an example of a compact manifold satisfying the \(\del\delb\)-property while 
admitting a nontrivial ABC-Massey product (see 
Example~\ref{exam:SferruzzaTomassini}).  
Subsequently, \cite{PSZ24} showed that any closed Riemann surface of genus at least two 
carries a nontrivial ABC-Massey product.  
In particular, even compact Kähler manifolds may exhibit such nontrivial ABC-Massey products.
The examples in \cite{PSZ24} do not have restricted holonomy, leaving open the question of 
whether Calabi-Yau or hyperkähler manifolds are strongly formal.  This was answered negatively for the Calabi-Yau case in the recent work \cite{MMSte25}.  

There are further contexts in which ABC-Massey products have been employed to obstruct 
the existence of certain metrics, for example in \cite{SfeTo23}, \cite{SfeTo24}, and \cite{SfeTo25}.  
Additional results on triple ABC-Massey products may be found in \cite{TaTo17}, \cite{CeTo25}, and 
\cite{Ru25}.

In \cite{MiSte24}, Milivojević and Stelzig introduce higher Aeppli-Bott-Chern Massey products, extending the notion of triple ABC-products, and prove that are invariant under pluripotential weak equivalences. These higher ABC-Massey products take inputs and land in non-standard cohomology groups. Indeed, the possible targets of these products involve cohomology groups of the Bigolin complex (\cite{Bi69}, \cite{De97}, \cite{Sch07}), and their input and landing spaces depend on auxiliary choices. Moreover, they are not very accessible or computable in practice.

In contrast to these partially defined products, the homotopy transfer theorem for bidifferential bigraded algebras provides a robust homotopical framework for understanding higher products in cohomology.  
More precisely, a minimal bicomplex (for instance, the \(H_{\Aa\Bb\Cc}\)-cohomology of a bba) can 
be endowed with a \(\ppinfa\)-algebra structure that encodes the full homotopy type of the original algebra.  
In practice, this yields genuine operations, all acting on a single (typically small) space, which collectively capture all the relevant homotopical information. Moreover, \cite{CiGaSo25} presents a package (\cite{GaVi24}) for computing such arity-3 products (and which can be extended to higher arities), demonstrating that these operations are indeed computable.

We now explain the relation between triple ABC-Massey products and a transferred structure. Using Proposition~\ref{prop:MinimalContraction}, given a contraction into a minimal bicomplex $(H,\del,\delb)$, we may define a new contraction
\[
\begin{tikzcd}[ampersand replacement = \&]
(A,\del,\delb) \arrow[r, rightarrow, shift left, " f"] \arrow[r, shift right, leftarrow, " g" swap] \arrow[loop left, "h"] \& (\ABC(A),\del,\delb),
\end{tikzcd}\label{diag:HabcContraction}\tag{C}
\]
where $ f:A\twoheadrightarrow H\xrightarrow{\cong}\ABC(A)$ and $ g:\ABC(A)\xrightarrow{\cong}H \hookrightarrow A$.

Applying the pluripotential homotopy transfer theorem
(Theorem~\ref{theo:HTT}) to this contraction yields an \(A_\infty^{pp}\)-structure on 
\(H_{\Aa\Bb\Cc}(A)\) together with a \(\ppinfm\)-weak equivalence
\[
H_{\Aa\Bb\Cc}(A) \rightsquigarrow A .
\]
By Theorem~\ref{theo:EqHoCats}, the transferred structure on \(H_{\Aa\Bb\Cc}(A)\) encodes the pluripotential homotopy type of \(A\).

From now on, we will write \(\mu_3:=\mu_3^{\mep1,\mep1},\) for the transferred arity 3 operation of bidegree \((-1,-1).\)

\begin{prop}\label{prop:ABC-transferred}
 Given 
 $[x], [y], [z]\in  \BC(A)$ such that
 $[x]\cdot[y]=[y]\cdot[z]=0$ and a contraction $( f, g,h)$ as \ref{diag:HabcContraction}, we have 
 \[-[ g\mu_3([x],[y],[z])]_\Aa\in \langle [x], [y], [z]\rangle,\]
 where $[-]_\Aa$ denotes taking the Aeppli cohomology class.
 \end{prop}
 \begin{proof}
 Note first that there is an inclusion $\BC(A)\subseteq \ABC(A)$, so we may view such classes in ABC-cohomology. Write $x= g([x])$, $y= g([y])$ and $z= g([z])$. Since
$[x]\cdot[y]=[y]\cdot[z]=0$, we may choose an element $u$ such that \[\del\delb u=x\cdot y\text{ and }h(x\cdot y)=-u.\]
In the same way, we can choose $v$ such that \[\del\delb v =y\cdot z\text{ and }h(y\cdot z)=-v.\]
This gives the following
\[\mu_3([x],[y],[z])= f(h(x\cdot y)\cdot z-x\cdot h(y\cdot z))= f(-u \cdot z+x\cdot v).\]
Thus, we obtain \[[ g\mu_3([x],[y],[z])]_\Aa=[x\cdot v-u \cdot z]_\Aa\in \langle [x], [y], [z]\rangle.\qedhere\]
 \end{proof}

\begin{coro}
Let 
 $[x], [y], [z]\in  \BC(A)$ such that
 $[x]\cdot[y]=[y]\cdot[z]=0$. If 
 $\mu_3([x],[y],[z])=0$ then the triple ABC-Massey product $\langle [x], [y], [z]\rangle$
 is trivial.
\end{coro}

\begin{rema}
    This reflects the situation observed in the ordinary case. Indeed, the triple products of $A_\infty$-structures arising from homotopy transfer theory always detect Massey products. That is, \[
    \pm \mu_3(a,b,c) \in \langle a,b,c \rangle.
    \]
    Consequently, if $\mu_3(a,b,c) = 0$, the corresponding Massey product must also vanish.
\end{rema}

In \cite{BuiMoFerMu20}, an explicit relationship is established between \(A_\infty\)-structures and ordinary higher Massey products.
We expect that an analogous phenomenon should hold for pluripotential \(A_\infty\)-algebras and higher 
ABC-Massey products. 

The following example exhibits the phenomenon described in Proposition~\ref{prop:ABC-transferred}. 
The computations presented here are extracted from \cite{CiGaSo25} and use the Sage package \cite{GaVi24}.

\begin{exam}\label{exam:SferruzzaTomassini}
    The following procedure, explained in \cite{SfeTo22}, gives a compact $\del\delb$-manifold by constructing an orbifold of global quotient type out of the Iwasawa nilmanifold. 

Consider the action $\sigma\colon \CC^3\lra \CC^3$ 
given by $ \sigma(z_1,z_2,z_3)=(iz_1,iz_2,-z_3)$.
This gives a well-defined action on the Iwasawa manifold. The resulting quotient 
$M=\mathbb{I}_3/\langle\sigma\rangle$ is a compact orbifold with 16 isolated singular points. 
By Theorem 7.1 of \cite{SfeTo22}, the orbifold $M$ admits a resolution into a smooth $\del\delb$-compact manifold.
The sub-algebra $\Aa_\sigma(M)$ of $\Aa(\II_3)$ of differential forms which are left invariant for the action of \(\HH(\CC)\) and $\sigma$-invariant  is generated by the elements
\[a\ov a, b\ov b, c\ov c, a\ov b, b\ov a, c\ov a\ov b, ab\ov c,  abc, \ov a\ov b\ov c.\]
The only non-trivial differentials being \[\del(c\ov c)=-ab\ov c, \,\del(c\ov a\ov b)=-ab\ov a\ov b\]
and their complex conjugates. 
This algebra satisfies the $\del\delb$-condition and so in particular, all cohomologies (including Dolbeault, Bott-Chern, Aeppli and ABC) coincide. In Figure \ref{pic:cohos-ST} it is indicated a basis for the cohomology in each bidegree.

\begin{figure}[ht]
\centering
\begin{tabular}{|c|c|c|c|}
\hline
$[\overline{abc}]$ & & & $[abc\overline{abc}]$ \\ \hline
 & & \begin{tabular}[c]{@{}c@{}}$[ac\overline{ac}]$, $[ac\overline{bc}]$,\\ $[bc\overline{ac}]$, $[bc\overline{bc}]$\end{tabular} & \\ 
 \hline
 & \begin{tabular}[c]{@{}c@{}}$[a\overline{a}]$, $[a\overline{b}]$,\\ $[b\overline{a}]$, $[b\overline{b}]$\end{tabular} & & \\ \hline
$1$ & & & $[abc]$ \\ \hline
\end{tabular}
\caption{A basis for the cohomology of \(M\).}
\label{pic:cohos-ST}
\end{figure}

In \cite{SfeTo22}, a non-trivial ABC-Massey product is computed on \(M\):
\[\langle [a\ov a],[b\ov b],[b\ov b]\rangle_{ABC}\in \frac{\A^{2,2}(M)}{[a\ov a]\cdot\A^{1,1}(M)+[b\ov b]\cdot\A^{1,1}(M)},\]
which is represented by the nonzero Aeppli cohomology class \([bc\ov b\ov c] \in H_A^{2,2}(M).\)

We define contraction
\[
\begin{tikzcd}[ampersand replacement = \&]
(\Aa_\sigma,\del,\delb) \arrow[r, rightarrow, shift left, " f"] \arrow[r, shift right, leftarrow, " g" swap] \arrow[loop left, "h"] \& (\ABC(\Aa_\sigma),\del,\delb)
\end{tikzcd}
\]
as follows:
the map $ g$ is induced by the identity on all elements. The map $ f$ is induced by the identity, except for the following elements, which belong to a square and for which $ f$ is zero:
\[\xymatrix{
c\ov{a}\ov{b}\ar[r]&ab\ov{a}\ov{b}\\
c\ov{c}\ar[u]\ar[r]&ab\ov{c}\ar[u]
}\]
The homotopy $h$ is given by 
\[h(ab \ov{a}\ov{b})=c\ov{c}\] 
and zero otherwise.

Applying the homotopy transfer theorem we obtain that the $\mu_3$-product is non-trivial only when all 3 inputs have bidegree $(1,1)$. We obtain
\[\mu_3([a\ov a],[b\ov b],[b\ov b])=[b\ov bc\ov c]\]
which gives precisely the ABC-Massey product computed in \cite{SfeTo22}.
\end{exam}

\section{Basics in operad theory}\label{sec:BasicOperadThry}

 In this section, we recall the definitions of operads and operadic algebras, as well as the dual notions of cooperads and coalgebras, along with some important results, adapted to the pluripotential setting when needed. Most of the material presented here is adapted from \cite{GeJo94}, \cite{LoVa12}, and \cite{Fre04}.  In our case, we will work over the closed symmetric monoidal model category of bicomplexes $\bicok$ over a field $\kk$ of characteristic zero, although most of the constructions are valid over any symmetric monoidal category. 

\subsection{\texorpdfstring{$\SS\,$}{S-modules}-modules}
For multiple constructions in operad theory,
it is convenient to work with the category $\smod$ of $\SS\,$-modules or, more generally, with $\SS\,\mep\mathrm{Ch}$, the category of differential graded $\SS$-modules. In our context, we will consider bidifferential bigraded $\SS\,$-modules. We next review their basic properties and present the Schur functor, a key ingredient to define operadic algebras.

\begin{defi}
    A \textit{bidifferential bigraded $\SS\,$-module} (\(\bbsmod\)-module for short) $M$ is given by a sequence of bicomplexes $\{M(n)\}_{n\in\ZZ_{\geq 0}}$ together with a right $\SS_n$ action, where $\SS_n$ denotes the $n$th symmetric group. 
\end{defi}

If $\mu\in M(n)^{p,q}$, we say that $\mu$ is of \emph{arity} $n$ and bidegree $(p,q)$. 

\begin{defi}
A \emph{morphism of }bb-$\SS\,$\emph{-modules} $f:M\to N$ is a sequence of equivariant bicomplex morphisms $f(n):M(n)\to N(n)$.
\end{defi}

We denote the category of bb-$\SS\,$-modules by $\bicosmod$. 
\begin{rema}
Note we can view every \(\bbsmod\)-module $M$ as a family of bicomplexes $\{M(n)\}_{n\in\ZZ_{\geq 0}}$ where each $M(n)^{p,q}$ is a right $\kk[\SS_n]$-module for all $(p,q)$.
\end{rema}

Let $M$ and $N$ be two \(\bbsmod\)-modules. Define their \emph{tensor product} as 
\[
(M\otimes N)(n):=\bigoplus_{p+q=n}\mathrm{Ind}_{\SS_p\times \SS_q}^{\SS_n}M(p)\otimes N(q).
\]
Here $\mathrm{Ind}_{\SS_p\times \SS_q}^{\SS_n}$ denotes the induced $\SS_n$-representation. That is, 
\[\mathrm{Ind}_{\SS_p\times\, \SS_q}^{\SS_n}M(p)\otimes N(q)=\left(M(p)\otimes N(q)\right)\tens{\SS_p\times \SS_q}\kk[\SS_n].\]

The \emph{composition product} is given by
\[M\circ N = \bigoplus_{n\geq 0}
M(n) \tens{\SS_n} N^{\otimes n}.\]

Explicitly, the $n$-ary component of the composition is:
\[
(M \circ N)(n) = \bigoplus_{k \geq 0} M(k) \tens{\SS_k} 
\left( \bigoplus_{i_1 + \cdots + i_k=n} 
\mathrm{Ind}_{\SS_{i_1} \times \cdots \times \SS_{i_k}}^{\SS_n} 
\left( N(i_1) \otimes \cdots \otimes N(i_k) \right) \right)
\]

where the action of $\SS_k$ is given by permuting the factors in the tensor product \[N(i_1) \otimes \dots \otimes N(i_k)\] via permutation of the indices $(i_1, \dots, i_k)$.
Therefore, an element of $(M \circ N)(n)$ is a linear combination of tuples
\[
(\mu; \mu_1, \dots, \mu_k) \sigma
\]
where $\mu\in M(k)$, $\mu_j\in N(i_j)$, $\sigma\in\SS_n$ and $i_1+\dots+i_k=n$. Such tuples are subject to the following relations:
\begin{equation}\label{eq:equivout}
    (\mu\tau;\mu_1,\dots,\mu_k)=(\mu;\mu_{\tau^{-1}(1)},\dots,\mu_{\tau^{-1}(k)})\tau_{i_1,\dots,i_k},
\end{equation}
    where $\tau_{i_1,\dots,i_k}$ is the block permutation of $\{1, \dots, n\}$ that divides it into $k$ blocks of sizes $i_1, \dots, i_k$, and permutes these blocks according to $\tau$ and
    \begin{equation}\label{eq:equivinn}
        (\mu;\mu_1\tau_1,\dots,\mu_k\tau_k)=(\mu;\mu_{\tau^{-1}(1)},\dots,\mu_{\tau^{-1}(k)})\tau_1\sqcup\dots\sqcup \tau_k,
    \end{equation}
    where each $\tau_j \in \SS_{i_j}$, and $\tau_1 \sqcup \cdots \sqcup \tau_k \in \SS_n$ is the permutation that acts within each block of size $i_j$ by $\tau_j$.

To a bb-$\SS\,$-module $M$ we may associate an endofunctor in $\bicok$ called the \emph{Schur functor}
\[M:\bicok\to\bicok\]
given by 
\[M(A)=\bigoplus_{n\geq 0}M(n)\tens{\SS_n}A^{\otimes n}.\]

This functor satisfies  
  \[(M\otimes N)(-)\cong M(-) \otimes N(-)\text{ and }
 (M \circ N)(-) \cong M(N(-)) 
  \]
(see for instance Section 5.1 of \cite{LoVa12}, or Section 1.2 in \cite{GeJo94}). Moreover, we have:

\begin{prop}\label{prop:tensorcomposition} The triple 
$(\bicosmod,\otimes,\id)$ is a symmetric monoidal category with $\id=(\kk,0,\dots)$. 
The triple 
$(\bicosmod,\circ,\mathrm{I})$ is a monoidal category, with $\mathrm{I}=(0,\kk,0,\dots)$. 
\end{prop}

\subsection{Operads and operadic algebras} In this section, we recall the definitions of operads and operadic algebras along with some useful results, adapted to the pluripotential setting when needed.

\begin{defi}
    An \emph{operad} is a monoid in the monoidal category $(\bicosmod,\circ,\mathrm{I})$, that is, a $\bbsmod$-module $\Pp$ together with morphisms
    \[\gamma:\Pp\circ\Pp\to\Pp\quad\text{ and }\quad\eta:\mathrm{I}\to\Pp\]
    called the \emph{composition} and \emph{unit} maps, respectively, satisfying associativity
\[
\begin{tikzcd}
(\Pp \circ \Pp) \circ \Pp \arrow[r, "\cong"] \arrow[d, "\gamma\circ\id"'] & 
\Pp \circ (\Pp \circ \Pp) \arrow[r, "\id \circ \gamma"] & 
\Pp \circ \Pp \arrow[d, "\gamma"] \\
\Pp \circ \Pp \arrow[rr, "\gamma"] & & \Pp
\end{tikzcd}
\]

and unitality

\[
\begin{tikzcd}
\mathrm{I} \circ \Pp \arrow[dr, "="'] \arrow[r, "\eta \circ \id"] & 
\Pp \circ \Pp \arrow[d, "\gamma"] & 
\Pp \circ \mathrm{I} \arrow[l, "\id \circ \eta"'] \arrow[dl, "="] \\
& \Pp &
\end{tikzcd}
\]
conditions.
\end{defi}

In particular, the operad structure is determined by maps
\[\Pp(k)\otimes\Pp(i_1)\otimes\dots\otimes \Pp(i_k)\to\Pp(i_1+\dots+i_k)\]
subject to associativity, unitality and equivariance conditions.  Moreover, the differentials of $\Pp$, $\del$ and $\delb$, are \emph{derivations} for $\gamma$
    \begin{gather*}
        \del\big(\gamma(\mu;\mu_1,\dots,\mu_k)\big):=\gamma\big(\del(\mu);\mu_1,\dots,\mu_k\big)\\
        +\sum_{i=1}^k(-1)^{\epsilon_i}\gamma\big(\mu;\mu_1,\dots,\del(\mu_i),\dots,\mu_k\big),
    \end{gather*}
where $\epsilon_i=|\mu|+|\mu_1|+\dots+|\mu_{i-1}|$ and $|-|$ stands for the total degree. The same identity holds for $\delb$.

Let $\Pp$ be an operad and let $\mu \in \Pp(m)$ an element in arity \(m\) and $\nu \in \Pp(n)$ and element in arity \(n\). We will sometimes use the \emph{partial composition}
\[
(\mu,\nu)\;\longmapsto\;\mu \circ_i \nu \in \Pp(m-1+n)
\]
defined for $1 \le i \le m$ by:
\[
-\!\circ_i\!-\;:\;\Pp(m)\otimes \Pp(n)\;\longrightarrow\;\Pp(m-1+n),
\qquad
\mu \circ_i \nu \;:=\; \gamma\bigl(\mu;\; \underbrace{\id,\dots,\id}_{i-1},
\nu,
\underbrace{\id,\dots,\id}_{m-i}\bigr).
\]

\begin{defi}
A \emph{morphism of operads} $f:\Pp\to\Qq$ is
a morphism of bb-$\SS\,$-modules which is compatible with the composition and unit maps. 
\end{defi}

We denote the category of operads in \(\bicok\) by $\ops$. 

\begin{exam}
    Let $A$ be a bicomplex. Consider the internal $\hom$ defined in Section~\ref{subsec:preliminarsbico}
    \[\underline{\mathrm{Hom}}(A,B)^{p,q}=
    \prod\Hom_\kk(A^{i,j},B^{i+p,j+q}),\]
    with differentials 
    \[\del f:=[\del,f]\quad\text{ and }\quad\delb f:=[\delb,f].
    \]

    Define the following bb-$\SS\,$-module \[\Endop(n):=\underline{\Hom}(A^{\otimes n}, A)\]
    where the action of $\SS_n$ is induced by the left $\SS_n$ action on $A^{\otimes n}$. We can endow it with an operadic structure given by composition of endomorphisms
    \[
    (f; f_1, \dots , f_k) := f(f_1 \otimes\dots\otimes f_k).
    \]
    The resulting operad is called the \emph{endomorphism} operad.
\end{exam}

If $\Pp$ is an operad, then Proposition~\ref{prop:tensorcomposition} shows that the Schur functor 
\[\Pp:\bicok\to\bicok\]
defines a monad.

\begin{defi}
    A $\Pp$\emph{-algebra} is an algebra over the monad $\Pp$, that is, an object $A$ in $\bicok$ together with maps of bicomplexes $\gamma_A\colon \Pp(A) \to A$ and $\eta_A\colon A\to\Pp(A)$ such that the following diagrams commute:
\[
\begin{tikzcd}[ampersand replacement=\&]
	{\Pp\circ\Pp(A)} \& {\Pp(A)} \\
	{\Pp(A)} \& A
	\arrow["{\gamma(A)}", from=1-1, to=1-2]
	\arrow["{\Pp(\gamma_A)}"', from=1-1, to=2-1]
	\arrow["\gamma_A", from=1-2, to=2-2]
	\arrow["\gamma_A"', from=2-1, to=2-2]
\end{tikzcd}\qquad
\begin{tikzcd}[ampersand replacement=\&]
	A \& {\Pp(A)} \& {} \\
	\& A
	\arrow["{\eta_A}", from=1-1, to=1-2]
	\arrow[shift right, no head, from=1-1, to=2-2]
	\arrow[no head, from=1-1, to=2-2]
	\arrow["\gamma_A", from=1-2, to=2-2]
\end{tikzcd}
\]
\end{defi}

The $\Pp$-algebra structure is determined by maps
\[
\gamma_n:\Pp(n)\tens{\SS_n} A^{\otimes n}\to A.
\]

One should think of $\Pp(n)$ as the space that parametrizes all the possible ways of combining $n$ elements of the $\Pp$-algebra.

Using the fact that there is a natural isomorphism
\[
\Hom_{\SS_n}(\Pp(n),\Hom(A^{\otimes n},A))\cong \Hom (\Pp(n)\tens{\SS_n} A^{\otimes n},A),
\]
one can alternatively define an algebra over an operad $\Pp$ as a morphism of operads 
\[
\Pp\to \Endop(A).
\]
The image of an element $\mu\in\Pp(n)$ is an operation $\mu_n:A^{\otimes n}\to A$. We denote the category of $\Pp$-algebras by $\Palg$.

The following definition will be used to twist differentials on a \(\Pp\)-algebra.
\begin{defi}
    A \emph{derivation} of a \(\Pp\)-algebra \(A\) is a linear map \(d:A\to A\) such that the following diagram commutes
   \[\begin{tikzcd}[ampersand replacement=\&]
    	{\Pp(A)} \& A \\
    	{\Pp(A)} \& A
    	\arrow["\gamma_A", from=1-1, to=1-2]
    	\arrow["{\id\circ'd}"', from=1-1, to=2-1]
    	\arrow["d", from=1-2, to=2-2]
    	\arrow["\gamma_A", from=2-1, to=2-2]
    \end{tikzcd}\]
    where \(\id\circ'd\) denotes the map given by \(\sum_{i=1}^n\id^{i}\otimes d\otimes\id^{n-i}\) in \(\Pp(n)\tens{\SS_n} A^{\otimes n}\).
\end{defi}
If \(A\) is a \(\Pp\)-algebra, the bigraded vector space of derivations of \(A\),
\(\mathrm{Der}(A)\),
is a bidifferential bigraded Lie algebra, with bracket the graded commutator and differentials the graded commutators with the differentials of \(A\).

\begin{defi}
    A \emph{pair of differentials} of a \(\Pp\)-algebra consists in two derivations \((\tau,\ov\tau)\) of degree \((1,0)\) and \((0,1)\), respectively, such that 
    \[(\del + \tau)^2 = 0,\quad(\delb + \ov\tau)^2 = 0\quad\text{ and }\quad(\del + \tau)(\delb + \ov\tau)+(\delb + \ov\tau)(\del + \tau) = 0.\]
\end{defi}

If \((\tau,\ov\tau)\) is a pair of differentials, then the bicomplex
\((A,\del+\tau,\delb+\ov\tau)\) is a \(\Pp\)-algebra with the same structure map \(\gamma: \Pp(A) \to A\) as \(A\).

Since an operad \(\Pp\) is a monad, we have the definition of a free \(\Pp\)-algebra.
\begin{defi}
   Let \(A\) be a bicomplex. The \emph{free \(\Pp\)-algebra} over \(A\) is the algebra given by \(\Pp(A)\) itself, and the algebra structures \(\gamma:\Pp\circ\Pp(A)\to\Pp(A)\) and \(\eta:\Pp(A)\to\Pp\circ\Pp(A)\) are given by the monad structure of \(\Pp\).   
\end{defi}

For any \(\Pp\)-algebra \(A\) and any linear map \(f : V \to A\), there is a
unique \(\Pp\)-algebra extension \(\tilde f : \Pp(V) \to A\) of \(f\):
\[
\begin{tikzcd}
V \arrow[r, "\eta"] \arrow[dr, "f"'] & \Pp(V) \arrow[d, "\tilde f"] \\
& A.
\end{tikzcd}
\]
Note that a free algebra is unique up to a unique isomorphism.

The following proposition is very useful for constructing derivations on free algebras; see for instance Section 6.3.6 in \cite{LoVa12}, \cite{GeJo94}, among others.

\begin{prop}
Any derivation \(d\) on a free \(\Pp\)-algebra \(\Pp(A)\) is uniquely determined by its restriction to the generators, that is, by a linear map \(A \to \Pp(A)\).

More explicitly, given a map \(\varphi : A \to \Pp(A)\), the unique
derivation \(d_{\varphi}\) of the free \(\Pp\)-algebra \(\Pp(A)\) extending
\(\varphi\) is defined by
\[
d_{\varphi} = (\gamma_{(1)} \circ \id_{A}) \,
(\id_{\Pp} \circ' \varphi).
\]
\end{prop}

An algebra is said to be \emph{quasi-free} if when we forget the differentials it is a free algebra.

There is a similar story in the non-symmetric setting. In this case, one considers sequences $M = \{M(n)\}_{n \geq 0}$ of bicomplexes without the $\SS_n$-action. The associated functor is defined by
\[
M(A) = \bigoplus_{n \geq 0} M(n) \otimes A^{\otimes n}.
\]

All the constructions presented above admit an obvious non-symmetric analoge.

\subsection{The free operad}

The \emph{free operad} over a \(\bbsmod\)-module \(M\) is an operad \(\TT(M)\) with a \(\bbsmod\)-module morphism \(\eta(M):M\to\TT(M)\) such that given an operad $\Pp$, any map $f:M\to \Pp$ of $\bbsmod$-modules, extends uniquely into an operad morphism $\Tilde{f}:\TT(M)\to \Pp$:
\[\begin{tikzcd}[ampersand replacement=\&]
	M \& {\TT(M)} \\
	\& \Pp
	\arrow["{\eta(M)}", from=1-1, to=1-2]
	\arrow["f"', from=1-1, to=2-2]
	\arrow["{\Tilde{f}}", from=1-2, to=2-2]
\end{tikzcd}\]

The functor \(\TT:\bicosmod\to\ops\) is the left adjoint to the forgetful functor \(U:\ops\to\bicosmod\). Arity-wise it can be computed as
\[
\TT(M)(n)
  := \mathop{\mathrm{colim}}_{T \in \mathrm{T}(n)} M(T),
\]
where \(\mathrm{T}(n)\) denotes the set of rooted trees with \(n\) leaves. For a tree \(T\), the object \(M(T)\) is the \emph{treewise tensor product} of the 
\(\SS\)-module \(M\) along the vertices of \(T\):
\[
M(T) = \bigotimes_{v \in V(T)} M(\mathrm{in}(v)),
\]
where \(\mathrm{in}(v)\) denotes the number of inputs of the vertex \(v\).  
For details, see Section~5.6.1 of \cite{LoVa12} or \cite{GeJo94}.
\begin{rema}
    When \(M(0)=0\), 
    \(
    \TT(M)(n)
    = \bigoplus_{T \in \mathrm{T}(n)} M(T).
    \)
\end{rema}  
It is helpful to think of an element of \(\TT(M)(n)\) as a sum of rooted trees in which each vertex \(v\) is decorated by an element of \(M(\mathrm{in}(v))\) and each leaf is decorated by an element of \(\{1,\dots,n\}\). The operad structure of the $\SS\,$-module $\TT(M)$ is given by grafting of trees. 

The following property will be very useful.
\begin{prop}
    Let $M$ be a $\mathrm{bb}\text{-}\SS\,$-module. Any derivation $d_\phi:\TT(M)\to\TT(M)$ of the free operad $\TT(M)$ is completely characterized by the image on the generators: $\phi:M\to\TT(M)$.
\end{prop}

When representing an element of $\TT(M)$ by a labeled tree, its image under $d_\phi$ is the sum of labeled trees, where $\phi$ has been applied once and only once to every vertex. Such operads are called \emph{quasi-free} operads. It means that the underlying bigraded operad, forgetting the differentials, is free.

\begin{rema}\label{rema:weight-grading}
   Let \(M\) be a \(\bbsmod\)-module, the free operad on \(M\), \(\TT(M)\), has a weight-grading. 
    By definition, the \emph{weight} \(w(\mu)\) of an operation \(\mu \in \TT(M)\)
    is defined as follows:
    \[
    w(\mathrm{id}) = 0,\qquad
    w(\mu) = 1 \ \text{when}\ \mu \in M,
    \]
    and
    \[
    w(\nu;\,\nu_1,\ldots,\nu_n)
    = w(\nu) + w(\nu_1) + \cdots + w(\nu_n).
    \]
We denote by \(\TT(M)^{(r)}\) the \(\bbsmod\)-module of operations of weight \(r\).
\end{rema}

A completely analogous construction exists for the free non-symmetric operad using planar trees instead, see \cite[Sections 5.9.5 and 5.9.6]{LoVa12} for details. 

\subsection{Cooperads and coalgebras over cooperads}\label{subsec:CoopsCoalgs}

We next review the dual notions of operads and algebras, that is, cooperads and coalgebras.

\begin{defi}
    A \emph{cooperad} is a comonoid in the monoidal category $(\bicosmod, \circ, \mathrm{I})$. Explicitly, this consists of a  $\bbsmod$-module $\Cc$ equipped with morphisms
    \[
    \Delta : \Cc \to \Cc \circ \Cc\quad\text{ and }\quad \varepsilon : \Cc \to \mathrm{I},
    \]
    satisfying coassociativity and counit axioms 
\[\begin{tikzcd}[ampersand replacement=\&]
	\Cc \&\& {\Cc\circ\Cc} \&\&\& \Cc \\
	{\Cc\circ\Cc} \& {(\Cc\circ\Cc)\circ\Cc} \& {\Cc\circ(\Cc\circ\Cc)} \&\& {\mathrm{I}\circ\Cc} \& {\Cc\circ\Cc} \& {\Cc\circ\mathrm{I}}
	\arrow["\Delta", from=1-1, to=1-3]
	\arrow["\Delta"', from=1-1, to=2-1]
	\arrow["{\id\circ\Delta}", from=1-3, to=2-3]
	\arrow["{=}"', from=1-6, to=2-5]
	\arrow["\Delta", from=1-6, to=2-6]
	\arrow["{=}", from=1-6, to=2-7]
	\arrow["{\Delta\circ\id}", from=2-1, to=2-2]
	\arrow["\cong", from=2-2, to=2-3]
	\arrow["{\varepsilon\circ\id}"', from=2-6, to=2-5]
	\arrow["{\id\circ\varepsilon}", from=2-6, to=2-7]
\end{tikzcd}\]
\end{defi}
We denote the category of cooperads by $\coops$. We will sometimes adopt the following notation
\[\Delta(\mu) = \sum(\nu;\, \nu_{1}, \ldots, \nu_{k})\sigma.\]
And sometimes use the shorter notation
\[
\Delta(\nu)
=
\sum 
\nu_{(1)} \circ \nu_{(2)},
\]
where 
\[
\nu_{(2)} = 
(\nu_{1} ,\dots, \nu_{k})\sigma
\in \Cc^{\otimes k}(n).
\]

The differentials \(\del_\Cc\) and \(\delb_\Cc\) are coderivations, under this notation, this means
\[
\Delta\del_{\Cc}(c)
=
\sum(\del_{\Cc}(c)\,;\, c_{1},\ldots,c_{k})\sigma
+
\sum\sum_{i=1}^{k}
(-1)^{\theta_{i}}\,
(c\,;\, c_{1},\ldots,\del_{\Cc}(c_{i}),\ldots,c_{k})\sigma,
\]
where \(\theta_{i} = |c| + |c_{1}| + \cdots + |c_{i-1}|\).

Given a cooperad \((\mathcal{C},\Delta,\eta)\), one may isolate from the full
decomposition map \(\Delta : \mathcal{C} \to \mathcal{C}\circ \mathcal{C}\) the part that
corresponds to decomposing into two parts. Let \(M\) and \(N\) be two \(\bbsmod\)-modules, we denote by 
\(M \circ_{(1)} N
\)
the \(\bbsmod\)-module defined in arity \(n\) by
\[
(M \circ_{(1)} N)(n)
\;=\;
\bigoplus_{k+\ell = n+1}
M(k) \otimes_{\SS_k}
\left(
\bigoplus_{1 \le j \le k}
\mathrm{Ind}_{\SS_{1}^{\times (j-1)}\times \SS_\ell\times \SS_1^{\times (k-j)}}^{\SS_n} \mathrm{I}^{\otimes j-1}\otimes N(\ell)\otimes \mathrm{I}^{k-j}
\right).
\]
This yields the
\emph{infinitesimal decomposition map} of \(\mathcal{C}\), defined as the
composition
\[
\Delta_{(1)} \colon
\Cc \xrightarrow{\;\Delta\;}
\Cc \circ \Cc 
\;\longrightarrow\;
\Cc \circ_{(1)} \Cc.
\]

The infinitesimal decomposition may be written in the form
\[
\Delta_{(1)}(\nu)
=
\sum_{i}
(\nu_{i}' \circ_{e_{i}} \nu_{i}'')
\,\tau_{i},
\]
for 
\(\nu_{i}' \in \Cc(k'_{i})\),
\(\nu_{i}'' \in \Cc(k''_{i})\),
and \(\tau_{i} \in \SS_{n}\),
where
\[
\nu_{i}' \circ_{e_{i}} \nu_{i}'' 
=
\nu_{i}' \circ 
( \id^{e_i-1} \otimes 
\nu_{i}'' \otimes 
\id^{k_i'-e_i}).
\]

If $\Cc$ is a cooperad, then the Schur functor is a comonad, see \cite{GeJo94}. 
\begin{defi}
    Let \(\Cc\) be an operad. A \emph{$\Cc$-coalgebra} is a coalgebra over the comonad $\Cc$. That is, an object $C$ in $\bicok$ together with maps of bicomplexes $\Delta_C:C\to\Cc(C)$ and $\varepsilon:\Cc(C)\to C$ satisfying the coassociativity and counit constraints:
    \[\begin{tikzcd}[ampersand replacement=\&]
	C \& {\Cc(C)} \& C \& {\Cc(C)} \\
	{\Cc(C)} \& {\Cc\circ\Cc(C)} \&\& C
	\arrow["\Delta_C", from=1-1, to=1-2]
	\arrow["\Delta_C"', from=1-1, to=2-1]
	\arrow["{\Cc(\Delta_C)}", from=1-2, to=2-2]
	\arrow["\Delta_C", from=1-3, to=1-4]
	\arrow[no head, from=1-3, to=2-4]
	\arrow[shift left, no head, from=1-3, to=2-4]
	\arrow["\varepsilon", from=1-4, to=2-4]
	\arrow["{\Delta_{\Cc}(C)}", from=2-1, to=2-2]
\end{tikzcd}\]
\end{defi}

The following definition will be used to twist codifferentials on a \(\Cc\)-coalgebra.

\begin{defi}
    A \emph{coderivation} of a \(\Cc\)-coalgebra is a linear map \(d\colon C\to C\) such that the following diagram commutes
   \[
    \begin{tikzcd}[ampersand replacement=\&]
    	C \& {\Cc(C)} \\
    	C \& {\Cc(C)}
    	\arrow["\Delta_C", from=1-1, to=1-2]
    	\arrow["d"', from=1-1, to=2-1]
    	\arrow["{\id\circ' d}", from=1-2, to=2-2]
    	\arrow["\Delta_C", from=2-1, to=2-2]
    \end{tikzcd}
    \]
\end{defi}
The bigraded vector space of coderivations on a \(\Cc\)-coalgebra \(C\), \(\mathrm{Coder}(C)\), is a bidifferential bigraded Lie algebra with bracket the graded commutator and differentials the graded commutators with the differentials of \(C\).

\begin{defi}
    A \emph{pair of codifferentials} of a \(\Cc\)-coalgebra consists in two coderivations \((\tau,\ov\tau)\) of degree \((1,0)\) and \((0,1)\), respectively, such that 
    \[(\del + \tau)^2 = 0,\quad(\delb + \ov\tau)^2 = 0\quad\text{ and }\quad(\del + \tau)(\delb + \ov\tau)+(\delb + \ov\tau)(\del + \tau) = 0.\]
\end{defi}

If \((C,\del,\delb)\) is a \(\Cc\)-coalgebra and \((\tau,\ov\tau)\) is a pair of codifferentials, then the bicomplex
\((C,\del+\tau,\delb+\ov\tau)\) is a \(\Cc\)-coalgebra with the same structure map
\(\Delta_C\colon C \to \Cc(C)\) as \(C\).

Since a cooperad \(\Cc\) is a comonad, there is the notion of cofree coalgebra. The \emph{cofree \(\Cc\)-coalgebra} over a bicomplex \(V\) is the algebra given by \(\Cc(V)\) itself, and the coalgebra structures \(\Delta_C\colon\Cc(V)\to\Cc\circ\Cc(V)\) and \(\eta_C\colon\Cc\circ\Cc(V)\to\Cc(V)\) are given by the comonad structure of \(\Cc\).

For any \(\Cc\)-coalgebra \(C\) and any linear map \(f \colon C \to V\), there is a
unique \(\Cc\)-coalgebra extension \(\tilde f \colon C \to \Cc(V)\) of \(f\):
\[\begin{tikzcd}[ampersand replacement=\&]
	\& {\Cc(V)} \\
	C \& V
	\arrow["\varepsilon", from=1-2, to=2-2]
	\arrow["{\tilde f}", from=2-1, to=1-2]
	\arrow["f"', from=2-1, to=2-2]
\end{tikzcd}\]
Note that a cofree coalgebra is unique up to a unique isomorphism.

The following will be very useful to construct coderivations of cofree coalgebras.
\begin{prop}
    Any coderivation \(d\) on a cofree \(\Cc\)-coalgebra \(\Cc(V)\) is completely characterized by its projection onto the space of cogenerators \(V\).
Explicitly, given a map \(\varphi \colon \Cc(V) \to V\), the unique coderivation
\(d_\varphi\) on the cofree \(\Cc\)-coalgebra \(\Cc(V)\) that extends \(\varphi\)
is given by
\[
d_\varphi
  \;=\;
  \bigl(\id_\Cc \circ (\id_V \,;\, \varphi)\bigr)
  \bigl(\Delta_{(1)} \circ \id_V\bigr)
\]
where \(\Delta_{(1)} : \Cc \to \Cc \circ_{(1)} \Cc\) is the infinitesimal
decomposition map of the cooperad. 
The map
\[
\id_\Cc \circ (\id_V \,;\,\varphi) :
(\Cc \circ_{(1)} \Cc)(V) \longrightarrow \Cc(V)
\]
acts by applying \(\id_V\) to the \(V\)-tensor factor and applying
\(\varphi\) to the \(\Cc(V)\)-tensor factor.
\end{prop}

The \emph{cofree cooperad} over a \(\bbsmod\)-module \(M\) is a cooperad \(\TT^c(M)\) such that 
for any \(\bbsmod\)-module morphism \(\varphi : \Cc \to M\) sending \(\id\) to \(0\), 
there exists a unique cooperad morphism 
\(\widetilde{\varphi} : \Cc \to \TT^{c}(M)\) making the following diagram commute:
\[
\begin{tikzcd}
C \arrow[dr, "\varphi"'] \arrow[r, "\widetilde{\varphi}"] & 
\TT^{c}(M) \arrow[d] \\
& M
\end{tikzcd}
\]
The following property will be very useful to construct coderivations
\begin{prop}
    Let $M$ be a $\mathrm{bb}\text{-}\SS\,$-module. Any coderivation $d_\phi\colon \TT(M)\to\TT(M)$ of the free operad $\TT(M)$ is completely characterized by the image to the cogenerators $\phi\colon \TT(M)\to M$.
    \end{prop}

\section{Minimal models of operads}\label{sec:MinimalModSul}

We prove the existence and uniqueness of minimal models for reduced operads in \(\bicok\) with respect to pluripotential weak equivalences. The models we give are defined inductively on arity, in the spirit of Sullivan’s minimal models and following the approach of Markl–Shnider–Stasheff \cite{MaShniSta02}.

\subsection{Basic pluripotential operad homotopy theory} 
We introduce principal extensions of operads in bicomplexes and show that these satisfy a lifting property with respect to pluripotential weak equivalences. We mainly adapt results from \cite{MaShniSta02}, \cite{Roig22} and \cite{CiRoig19} on basic homotopy theory of operads, to the pluripotential case. In particular, we introduce a functorial path for operads in bicomplexes and study the associated notion of homotopy equivalence. 
 
 We will work with \emph{reduced} operads, i.e. operads \(\Pp\) satisfying \(\Pp(0)=0\) and \(\Pp(1)=\kk\). Reduced operads are \emph{augmented}, that is, there exists an operad morphism \(\varepsilon:\Pp\to\mathrm{I}\).

\begin{defi}
    Let $n\geq 2$ be an integer. Let $(\Pp,\del_\Pp,\delb_\Pp)$ be a quasi-free operad, $\Pp=\TT(M)$ where $M$ is a $\bbsmod$-mod with $M(0)=M(1)=0$. An \emph{arity n principal extension} of $\Pp$ is the quasi-free operad \[\Pp\sqcup_d\TT(E):=\TT(M\oplus E),\]
    where $E$ is a bigraded $\SS\,$-module concentrated in arity $n$ and $d=(\del+\delb):E\to\ker((\del_\Pp+\delb_\Pp)(n))$ a map of degree $(1,0)+(0,1)$. The differentials on $\Pp\sqcup_d\TT(E)$ are determined by the differentials of $\Pp$, $d$ and the Leibniz rule.
\end{defi}

\begin{lemm}\label{lemm:univpropprincext}
    Let $\Pp\sqcup_d\TT(E)$ be an arity $n$ principal extension of a quasi-free operad $\Pp=\TT(M)$ and let $\varphi:\Pp\to \Qq$ be a morphism of operads. 
    A morphism $\varphi'\colon \Pp\sqcup_d\TT(E)\to Q$ extending $\varphi$ is uniquely determined by a morphism of $\SS\,$-modules $f\colon E\to Q(n)$ satisfying $(\del_\Qq+\delb_\Qq) f=\varphi d$.
\end{lemm}

\begin{proof}
    By the universal property of free operads, any such $\varphi'$ is determined by its restriction $\varphi'_{|M\oplus E}$, and $\varphi'_{|M}=\varphi$, set $f:=\varphi'_{|E}$. We have $(\del_\Qq+\delb_\Qq)f=(\del_\Qq+\delb_\Qq)\varphi'_{|E}=\varphi'\circ d_{|E}=\varphi d.$  
\end{proof}

The following is a pluripotential version of Lemma 3.139 in \cite{MaShniSta02}. We adapt the proof given in \cite{CiRoig19}.
In our current setting, 
weak equivalences and surjections are given by arity-wise pluripotential weak equivalences and surjections respectively. 

\begin{lemm}
    Let $\iota\colon\Pp\to \Pp\sqcup_d\TT(E)$ an arity $n$ principal extension and
    \[\begin{tikzcd}
	\Pp & \Qq \\
	\Pp\sqcup_d\TT(E) & \Rr
	\arrow["\iota"', from=1-1, to=2-1]
	\arrow["\psi", from=2-1, to=2-2]
	\arrow["{\rotatebox{90}{$\sim$}}"', "\omega", two heads, from=1-2, to=2-2]
	\arrow["\varphi", from=1-1, to=1-2]
	\arrow["{\psi'}", dashed, from=2-1, to=1-2]
\end{tikzcd}\]
a commutative diagram of operad morphisms, where $\omega$ is a surjective pluripotential weak equivalence. Then, there exists an operad morphism $\psi'$ making both triangles commute.
\label{liftingprop}
\end{lemm}

\begin{proof}
    Consider the diagram of bigraded $\kk$-vector spaces:
   \[
    \begin{tikzcd}
    & {Z(C(\id_\Qq(n)))} \\
    E & {Z(C(\omega(n)))}
    \arrow["Id^{\oplus 3}\oplus\omega(n)", two heads, from=1-2, to=2-2]
    \arrow["\lambda", from=2-1, to=2-2]
    \arrow["\mu", dashed, from=2-1, to=1-2]
    \end{tikzcd}
    \]
    where \(Z(\_)\) denotes the space of Bott-Chern cocycles, \(C(\_)\) denotes the cone of a map of bicomplexes defined in Section~\ref{subsec:preliminarsbico}, and $\lambda = (\varphi\circ\del\delb, -\varphi\circ \del,-\varphi\circ \delb, \psi_{|E})$. We first show this diagram is well defined. Recall that
    \[
    C(\id_\Qq(n)) = \Qq(n)[-1,-1] \oplus \Qq(n)[-1,0] \oplus \Qq(n)[0,-1] \oplus \Qq(n)
    \]
    with differentials given by
    \begin{equation*}
    \begin{aligned}
    \del(a,  a',  a'', b) &= (\del a, - \del a', - \del a'' - a, \del b + a') \\
    \delb(a,  a',  a'', b) &= (\delb a, - \delb a' + a, - \delb a'', \delb b + a'').
    \end{aligned}
    \end{equation*}
    
    In the same way,
    \[
    C(\omega(n)) = \Qq(n)[-1,-1] \oplus \Qq(n)[-1,0] \oplus \Qq(n)[0,-1] \oplus \Rr(n)
    \]
    with differentials given by
    \begin{equation*}
    \begin{aligned}
    \del(a,  a',  a'', b) &= (\del a, - \del a', - \del a'' - a, \del b + \omega(a')), \\
    \delb(a, a',  a'', b) &= (\delb a,  -\delb a' + a, - \delb a'', \delb b + \omega(a'')).
    \end{aligned}
    \end{equation*}

    Let $(a,a',a'',b)\in Z(C(\id_\Qq(n)))$, then $(\del+\delb)(a) = 0$, $\delb(a') = -\del(a'') = a$, $-(\del+\delb)b = a'+ a''$, and
    \[
    (\id^{\oplus 3}\oplus \omega(n))(a,a',a'',b) = (a,a',a'',\omega(b))\in Z(C(\omega(n)))
    \]
    because $-(\del+\delb)\omega(b) = -\omega((\del+\delb)b) = \omega(a' + a'')$. It is also clear that $\Img(\lambda)\subseteq Z(C(\omega(n)))$ by Lemma~\ref{lemm:univpropprincext}.
    
    We now construct a map \(\mu:E\to Z(C(\id_\Qq(n)))\). Since $\omega$ is a  
    pluripotential weak equivalence, we have that \(\BC(C(\omega(n)))=0\) (see Section \ref{subsec:preliminarsbico}). Let $(a, a', a'', c) \in Z(C(\omega(n))),$ then there is an element $(e, e', e'', f)$ such that $\del\delb(e, e', e'', f) = (a, a', a'', c)$, note that
    \[
    \del\delb(e, e', e'', f) = (\del\delb e, \del\delb e' - \del e, \del\delb e'' - \delb e, \del\delb f + \omega(-\delb e' + \del e'' + e)).
    \]
    
    Since $\omega(n)$ is surjective, we may choose $f' \in Q(n)$ such that $\omega(f') = f$. Consider the element $(a, a', a'', \del\delb f' - \delb e' + \del e'' + e) \in Z(\id_\Qq(n)).$ It satisfies 
    \[
    (\id^{\oplus 3}\oplus\omega)(a,a',a'',\del\delb f' - \delb e' + \del e'' + e) = (a, a', a'', c).
    \]
    This proves $\id^{\oplus 3} \oplus \omega(n)$ is surjective. Thus, we can construct $\mu = (\alpha, \alpha', \alpha'', \beta)(e)$ by choosing a preimage of $\lambda(e)$ through $(\id^{\oplus 3} \oplus \omega(n))$. In particular, we obtain a map \(\beta:E\to\Qq(n)\) such that $-(\del+\delb)\beta(e) = (\alpha'' + \alpha')(e) = -\varphi(\del+\delb)(e)$. According to Lemma~\ref{lemm:univpropprincext}, we may obtain $\psi'$ extending $\varphi$ from $\beta: E \to \Qq(n)$. Moreover, $\omega\beta = \psi|_E$.   
\end{proof}

\begin{defi}
We will say that an operad is \emph{cofibrant} if is the colimit of a sequence of principal extensions starting from the initial operad \(\mathrm{I}\). 
\end{defi}

The following lifting lemma justifies the above definition. Its proof follows by induction using 
 Lemma~\ref{liftingprop}.

\begin{prop}\label{liftingprop2}
    Let $\Oo$ be a cofibrant operad. For every diagram of operads
\[\begin{tikzcd}
    	& \Pp \\
    	\Oo & \Qq
    	\arrow["{\rotatebox{90}{$\sim$}}"',"\omega", two heads, from=1-2, to=2-2]
    	\arrow["{\varphi'}", dashed, from=2-1, to=1-2]
    	\arrow["\varphi", from=2-1, to=2-2]
\end{tikzcd}\]
in which $\omega$ is a surjective pluripotential weak equivalence, there exists a morphism of operads $\varphi'$ making the diagram commute.
\label{cofibrantlifting}
\end{prop}

\begin{rema}
    Let $\Pp$ be an operad and $K$ a commutative bb-algebra. 
    Then $\Pp\otimes K=\{\Pp(n)\otimes K\}_{n\geq 0}$ has a natural operad structure given by $(p\otimes a)\circ_i(q\otimes b)=(-1)^{|a||q|}(p\circ_i q)\otimes (ab)$.
\end{rema}

Consider the cbba $R=\kk[t,\del t,\delb t,\del\delb t]=\Lambda(t,\del t,\delb t,\del\delb t),$ where $|t|=(0,0)$.
We have the unit $\iota\colon\kk\to R$ and evaluations $\delta^0$ and $\delta^1$ at $t=0$ and $t=1$, respectively.
These morphisms satisfy $\delta^0\circ\iota=\delta^1\circ\iota=1.$
By \cite[Lemma 2.13]{Ste25}, $\iota$ is a pluripotential weak equivalence of cbba's. 

\begin{defi}
    We define a \emph{functorial path} in the category of operads as the functor \[-\otimes R:\ops\to \ops,\] given on objects by $\Pp[t,\del t,\delb t,\del\delb t]=\Pp\otimes R$ and on morphisms by $\varphi\otimes\operatorname{id}_R$ together with the natural transformations
    \[\begin{tikzcd}
    	\Pp & {\Pp\otimes R} & \Pp
    	\arrow["\iota", from=1-1, to=1-2]
    	\arrow["{\delta^1}"', shift right=1, from=1-2, to=1-3]
    	\arrow["{\delta^0}", shift left=1, from=1-2, to=1-3]
    \end{tikzcd}\]
     such that $\delta^k\circ\iota=\id,$ where $\delta^k=\id\otimes \delta^k\colon\Pp\otimes R\to \Pp\otimes \kk=\Pp$ and $\iota=\id\otimes\iota\colon\Pp=\Pp\otimes\kk\to\Pp\otimes R$. 
\end{defi}

Since $\iota\colon\kk\to R$ is a pluripotential weak equivalence, we have $\iota(n)\colon\Pp(n)=\Pp(n)\otimes \kk\to \Pp(n)\otimes R$ is also a pluripotential weak equivalence of bicomplexes and hence, $\iota\colon\Pp\to\Pp\otimes R$ is a pluripotential weak equivalence of operads. 
Note also that the path preserves weak equivalences. Also, the map $(\delta^0,\delta^1)\colon\Pp\otimes R\to\Pp$ is surjective. The maps $\delta^0$ and $\delta^1$ are surjective pluripotential weak equivalences.

\begin{defi}
    Let $\varphi,\psi\colon\Pp\to\Qq$ be two morphisms of operads. A \emph{homotopy from} $\varphi$ to $\psi$ is given by a morphism of operads $H\colon\Pp\to\Qq\otimes R$ such that $\delta^0\circ H=\varphi$ and $\delta^1\circ H=\psi$. We use the notation $H\colon\varphi\simeq \psi$.
\end{defi}

We say that $\varphi$ and $\psi$ are \emph{homotopic} if there exists a homotopy $H\colon\varphi\simeq\psi$.

The homotopy relation is symmetric and reflexive.
Moreover, it is transitive whenever the source is a cofibrant operad. The proof of this last fact follows from standard homotopy theory arguments (see for instance \cite{CiRoig19}). 
Given operads $\Oo$ and $\Pp$, with $\Oo$ cofibrant, denote by  $[\Oo,\Pp]$ the set of homotopy classes of morphisms of operads $\Oo\to \Pp$.
The lifting property of Proposition \ref{liftingprop2} gives:
\begin{prop}
    Let $\Oo$ be a cofibrant operad. Any pluripotential weak equivalence $\omega\colon\Pp\to \Qq$  of operads induces a bijection $\omega_*\colon[\Oo,\Pp]\to[\Oo,\Qq]$.
    \label{LiftingUniqueUptoHomCof}
\end{prop}

\subsection{Minimal models}\label{subsec:MinialModelsSulivan}
We next define minimal operads as cofibrant operads satisfying an additional decomposability condition on their differentials, and prove their existence and uniqueness. 

Recall that the free operad is weight graded and that $\TT(M)^{(r)}$ denotes the $\SS\,$-module of operations of weight $r$, see Remark~\ref{rema:weight-grading}.

\begin{defi}\label{defi:minimal}
    A \emph{minimal} operad is a cofibrant operad \(\Mm\) such that the differentials are minimal, in the sense that \(\del\delb(\Mm)\subseteq \Mm^{(\geq 2)}\).
\end{defi}
The following lemma is a key result for building minimal models using principal extensions, we will use the cone of a map of bicomplexes defined in Section~\ref{subsec:preliminarsbico}.
\begin{lemm}\label{lemm:principalextcone}
    Let \(\alpha:\Mm\to\Pp\) be a morphism of operads where
    \begin{itemize}
        \item \(\Mm\) is a free operad, 
        \item \(\alpha:\Mm\to\Pp\) is a pluripotential weak equivalence in arities \(< n\). 
    \end{itemize}
    Then, there is an arity \(n\) principal extension \(\Mm'=\Mm\sqcup_{d+\ov d}\TT(E)\) 
    and a map of operads \(\alpha'\colon\Mm'\to\Pp\) such that it is a pluripotential weak equivalence for arities \(<n\) and
    \begin{enumerate}
        \item \label{principalextcone1} the bicomplex \(C(\alpha'(n))\) satisfies the \(\del\delb\)-property, 
        \item \label{principalextcone2}if \(C(\alpha(n))\) already satisfies the \(\del\delb\)-property, then \(\alpha'(n):\Mm'(n)\to\Pp(n)\) is a pluripotential weak equivalence for arities \(\leq n\).
    \end{enumerate}
\end{lemm}
\begin{proof}
    Consider the bigraded $\SS\,$-mod $E:=\BC(C(\alpha(n))),$ where \[C(\alpha(n))=\Mm(n)[-1,-1]\oplus\Mm(n)[-1,0]\oplus\Mm(n)[0,-1]\oplus \Pp(n)\] denotes the cone of the morphism of bicomplexes $\alpha(n)$. Consider a section $s:=(d,-d_1,-d_2,f)$ of the projection 
    \[Z(C(\alpha(n)))\to \BC(C(\alpha(n))).\]
    Note that $s(e)\in Z(C(\alpha(n)))$ implies 
    \[
    (\del d,\del d_1, \del d_2-d, \del f-\alpha d_1)=0\Rightarrow
            \begin{cases}
                \del d=0\\
                \del d_1=0\\
                \del d_2=d\\
                \del f=\alpha d_1\\
            \end{cases}
    \]
    and
    \[
    (\delb d, \delb d_1+d,\delb d_2,\delb f-\alpha d_2)=0\Rightarrow
            \begin{cases}
                \delb d=0\\
                \delb d_1 = -d\\
                \delb d_2=0\\
                \delb f =\alpha d_2\\
            \end{cases}
    \]

    This shows that $(d_1+d_2)$ defines a differential $\del_{M'}+\delb_{M'}$ on $\Mm'=\Mm\sqcup_{d_1+ d_2}\TT(E)$ and $f$ extends $\alpha$ by Lemma~\ref{lemm:univpropprincext}, i.e., there exists $\alpha':\Mm'\to \Pp$ such that $\alpha'_{|\Mm}=\alpha$ and $\alpha'{|E}=f$.

    We now prove (\ref{principalextcone1}), that is we prove that $C(\alpha'(n))$ satisfies the $\del\delb$-property, which is equivalent to
    \begin{equation}
       \id_*:\BC(C(\alpha'(n)))\to \A(C(\alpha'(n)))
        \label{LemmaTecnicExistence}
    \end{equation}
    being an isomorphism, see Section~\ref{subsec:preliminarsbico}. 
    
    It is enough to see that $\widetilde{H}_{\Bb\Cc}(C(\alpha'(n)))=0$, where we recall that
    \[\widetilde{H}_{\Bb\Cc}:=\frac{\Img \del\cap \ker\delb + \Img \delb\cap \ker\del}{\Img\del\delb}.\]
    We have
    \[C(\alpha'(n))\cong \ele\otimes \Mm'(n)\oplus \Pp(n)=\ele\otimes (\Mm(n)\oplus E)\oplus \Pp(n),\]
    then, every element in $C(\alpha'(n))$ can be expressed as $(a+e,a'+e',a''+e'',c)$.
    Let $x=[(a+e,a'+e',a''+e'',c)]\in \BC(C(\alpha'(n))),$ the cycle condition gives   
    \[\begin{cases}
        \del (a+e)=\delb (a+e)=0,\\
        -\del (a'+e') = 0,\\
        \delb(a'+e')=a+e,\\
        -\del(a''+e'')=a+e,\\
        \delb(a''+e'')=0,\\
        \del c = -\alpha'(a'+e'),\\
        \delb c = -\alpha'(a''+e'').
    \end{cases}\] Which implies $e=0$, since $\del(a'+e')\in \Mm(n)$. Graphically, this means
    
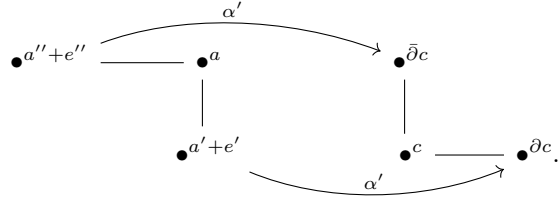
\[\begin{tikzcd}
    	{\bullet^{a''+e''}} & {\bullet^{a}} && {\bullet^{\delb c}} \\
    	& {\bullet^{a'+e'}} && {\bullet^{c}} & {\bullet^{\del c}}
    	\arrow[shift left, no head, from=2-2, to=1-2]
    	\arrow[no head, from=1-1, to=1-2]
    	\arrow[shift left, no head, from=2-4, to=1-4]
    	\arrow[no head, from=2-4, to=2-5]
    	\arrow["\alpha'", curve={height=-18pt}, from=1-1, to=1-4]
    	\arrow["\alpha'", curve={height=18pt}, from=2-2, to=2-5].
    \end{tikzcd}\]
    We want to show that if a representative of a Bott-Chern class of $C(\alpha'(n))$ is in $\Img\del+\Img\del$ then it is in $\Img\del\delb$ and thus, its class is trivial.
    Suppose there is an element $(z,z',z'',c')\in C(\alpha'(n))$ such that \[\del(z,z',z'',c')=(a,a'+e',a''+e'',c),\] then
    \[\begin{cases}
        \del z=a,\\
        -\del z' = a'+e',\\
        -\del z'' - z =a''+e'',\\
        \del c' +\alpha'(z') = c,\\
    \end{cases}\]
    
    which implies $e'=0$ since $\del z'\in \Mm(n)$ and $\del\delb z'=\delb a' = a$. Now, since $e''\in E=\BC(C(\alpha(n)))$, we have, by definition, that
    \[e''=[s(e'')]=[(-\del\delb a'',\del a''+a,\delb a'',-\delb c-\alpha(a''))].\] 
    Note that
    \[(-\delb z'-a'',0,0,c')\in C(\alpha(n))\]
    and
    \begin{equation*}
        \begin{gathered}
            \del\delb(-a''-\delb z',0,0,c')
            =(-\del\delb a'',\del a''+\del\delb z',\delb a'',\del\delb c'-\alpha(a'')-\alpha(\delb z'))=\\
            =(-\del\delb a'',\del a''+a,\delb  a'',-\delb c-\alpha(a''))=s(e'').
        \end{gathered}
    \end{equation*}
    This implies $e''=0$. Therefore, $x=[(a,a',a'',c)]\in \BC(C(\alpha(n)))=E$ and hence $\del\delb(x,0,0,0)=(a,a',a'',c)$ in $C(\alpha'(n))$, indicating that $x$ is trivial. A similar argument shows that if there is an element $(z,z',z'',c')\in C(\alpha'(n))$ such that \(\delb(z,z',z'',c')=(a,a'+e',a''+e'',c),\) then we also have that \(x=[(a,a'+e',a''+e'',c)]\) is trivial.  
    Thus, we can conclude $\widetilde{H}_{\Bb\Cc}(C(\alpha'(n)))=0$. 

    We now prove (\ref{principalextcone2}). Suppose
   $C(\alpha(n))$ satisfies the \(\del\delb\)-property. We want to show \(\alpha'(n)\) is a pluripotential weak equivalence for $m=0,\dots,n.$ Since $\alpha'(m)=\alpha(m)$ for $m=0,\dots,n-1$, it is a pluripotential weak equivalence by induction hypothesis. Let us prove that \[\alpha'(n)=(\alpha \oplus f):\Mm'(n)=\Mm(n)\oplus E\to \Pp(n)\] is a pluripotential weak equivalence.

    Consider the map \[j:\Mm(n)\hookrightarrow \Mm'(n)=\Mm(n)\oplus E,\]
    and the morphism \[(\id\oplus \alpha'):C(j)=\ele\otimes \Mm(n)\oplus (\Mm(n)\oplus E)\to C(\alpha(n))=\ele\otimes \Mm(n)\oplus \Pp(n).\]
    We will prove that $\BC(\id\oplus\alpha')$ and $\A(\id\oplus\alpha')$ are isomorphisms. 
    
    Let $[(x,x',x'',y+e)]\in \BC(C(j))$, then 
    \[\begin{cases}
        \delb x' =x,\\
        -\del x''=x,\\
        j(x')=x'=-\del(y+e),\\
        j(x'')=x''=-\delb(y+e).
    \end{cases}\]
    Therefore, 
    \begin{equation*}
        \begin{gathered}
            [(x,x',x'',y+e)]=[(\del\delb (y+e),-\del(y+e),-\del(y+e),y+e)]=\\=[(\del\delb e,-\del e,-\del e,e)]+[(\del\delb y,-\del y,-\delb y,y)]=\\=[(\del\delb e,\del e,-\del e,e)]+[\del\delb(y,0,0,0)].
        \end{gathered}
    \end{equation*}
    This implies that every element in $\BC(C(j))$ has a representative $[(\delb\del e,\del e,-\delb e,e)]$, where $e\in E$. If $(\id\oplus \alpha')(\del\delb e,-\del e,-\delb e,e)=(\del\delb e,-\del e,-\delb e,\alpha'e)=s(e)=\del\delb(y,y',y'',c)$, then $e=0$. This proves $\BC(\id\oplus\alpha')$ is injective.
    
    Let $e\in \BC(C(\alpha(n)))$, we can choose a representative $e=[(x,x',x'',c)]$, which implies $\del \delb(e)=x$, $\del(e)=-x'$, $\delb(e)=-x''$ and $\alpha'(e)=c$. Then \[[(x,x',x'',e)]\mapsto[(x,x',x'',c)].\]
    This proves $\BC(\id\oplus\alpha')$ is surjective. 
    
    Now, since the map \(C(\alpha(n))\) satisfies the \(\del\delb\)-property, $\id_*\colon\A(C(\alpha(n)))\to \BC(C(\alpha(n)))$ is an isomorphism. We have the following commutative diagram:
\[\begin{tikzcd}[ampersand replacement=\&]
	{\BC(C(j))} \& {\A(C(j))} \\
	{\BC(C(\alpha(n)))} \& {\A(C(\alpha(n)))}
	\arrow["{\id_*}", from=1-1, to=1-2]
	\arrow["\cong"', from=1-1, to=2-1]
	\arrow[from=1-2, to=2-2]
	\arrow["\cong"',"{\id_*}", from=2-1, to=2-2]
\end{tikzcd}\]

    which implies $\id_*:\BC(C(j))\to\A(C(j))$ is injective and, hence, an isomorphism, see \ref{subsec:preliminarsbico}. This gives us the isomorphism \(\A(\id\oplus\alpha')\).

    By Lemma~\ref{lemm:long-exact-seq}, there is a commutative diagram, where the rows are exact:
    \[
    \begin{adjustbox}{max width=\textwidth,center}
    \begin{tikzcd}[ampersand replacement=\&,column sep=tiny]
    	{\A^{p-1,q-1}(\Mm(n))} \& {\A^{p-1,q-1}(\Mm'(n))} \& {\A^{p-1,q-1}(C(j))} \& {\BC^{pq}(\Mm(n))} \& {\BC^{pq}(\Mm'(n))} \& {\BC^{pq}(C(j))} \\
    	{\A^{p-1,q-1}(\Mm(n))} \& {\A^{p-1,q-1}(\Pp(n))} \& {\A^{p-1,q-1}(C(\alpha(n)))} \& {\BC^{pq}(\Mm(n))} \& {\BC^{pq}(\Pp(n))} \& {\BC^{pq}(C(\alpha(n)))}
    	\arrow[from=1-1, to=1-2]
    	\arrow[from=1-2, to=1-3]
    	\arrow[from=1-3, to=1-4]
    	\arrow[from=1-4, to=1-5]
    	\arrow[from=1-5, to=1-6]
    	\arrow[from=2-5, to=2-6]
    	\arrow[from=2-4, to=2-5]
    	\arrow[from=2-3, to=2-4]
    	\arrow[from=2-2, to=2-3]
    	\arrow["\alpha'", from=1-2, to=2-2]
    	\arrow["(\id\oplus\alpha')", "\cong"', from=1-3, to=2-3]
    	\arrow["\cong"', from=1-4, to=2-4]
    	\arrow["\alpha'", from=1-5, to=2-5]
    	\arrow["(\id\oplus\alpha')", "\cong"', from=1-6, to=2-6]
    	\arrow["\cong"', from=1-1, to=2-1]
    	\arrow[from=2-1, to=2-2]
    \end{tikzcd}
    \end{adjustbox}
    \]   
    by the Five Lemma, $\A(\alpha'(n))$  and $\BC(\alpha'(n))$ are isomophisms, which shows that 
    $\alpha'(n)$ is a pluripotential weak equivalence. 
\end{proof}

\begin{theo}[Existence of minimal models]\label{theo:MinimalModelsSullivan}
    Let $\Pp$ be an operad such that $\Pp(0)=0$ and $\Pp(1)=\kk$. 
    There exists a minimal operad $\Mm$ and a pluripotential weak equivalence $\Mm\xrightarrow[]{\sim}\Pp$.
\end{theo}
\begin{proof}
We will define, inductively on arity $m\geq 2$, 
    a morphism of operads 
    \[
    \alpha_{m}:\Mm_{m}\to\Pp
    \] in such a way that:
    \begin{enumerate}
        \item[(i)] $\Mm_{m}$ is a minimal operad generated in arities $<m$,
        \item[(ii)] $\alpha_{m}\colon\Mm_{m}\to \Pp$ is a pluripotential weak equivalence in arities $<m$.
    \end{enumerate}

  For our base case, we begin by adding all zig-zags in arity 2:
  let $E$ be the $\bbsmod$-module $E=\ABC(\Pp(2))$ and $s_2:\ABC(\Pp(2))\to \Pp(2)$ a section of a projection $\Pp(2)\cong \ABC(\Pp(2))\oplus \Pp(2)_{sq}\twoheadrightarrow \ABC(\Pp(2))$. We set
    \[\Mm_2=\TT(E),\quad \del_{2_{|E}}=\del_{\ABC}\quad\delb_{2_{|E}}=\delb_{\ABC}\quad\alpha_2:\Mm_2\to\Pp,\;
    \alpha_{2_{|E}}=s_2.\]
    Note that $\del_{2_{|E}}\delb_{2_{|E}}=0$. The map $\alpha_2$ is a morphism of operads and is a pluripotential weak equivalence in arities $\leq 2$.
   
    Assume we have defined 
    $\alpha_{m}\colon\Mm_{m}\to\Pp$ for all $m<n$. Let us now build $\alpha_n:\Mm_n\to \Pp$. We apply Lemma~\ref{lemm:principalextcone} to construct an arity \(n\) principal extension \(\Mm'=\Mm_{n-1}\sqcup_{d'+\ov{d'}}\TT(E')\) and obtain a map \(\alpha':\Mm'\to \Pp\) such that \(C(\alpha'(n))\) satisfies the \(\del\delb\)-property.

    Then, $\Mm'$ is minimal since it is cofibrant and $d_{\Mm'}(\Mm')=(\del_{\Mm'}+\delb_{\Mm'})(\Mm')\subseteq \Mm'^{(\geq 2)}$. Indeed,
    \begin{itemize}
        \item $d_{\Mm'|\Mm_{n-1}}=d_{\Mm_{n-1}}$, but $\del_{\Mm_{n-1}},\delb_{\Mm_{n-1}}$ are minimal by induction hypothesis,
        \item $d_{\Mm'|E'}:E'\xrightarrow[]{d'+\ov d'}\Mm_{n-1}(n)\subseteq \Mm'^{(\geq 2)}.$
    \end{itemize}

    We now apply again Lemma~\ref{lemm:principalextcone} to the map \(\alpha'\colon\Mm'\to\Pp\), which satisfies that \(C(\alpha'(n))\) has the \(\del\delb\)-property. This produces a map \(\alpha_n\colon\Mm_n=\Mm'\sqcup_{d_n+\ov d_n}\TT(E)\to \Pp\) which is a pluripotential weak equivalence in arities \(\leq n\). 

    Again, $\Mm_n$ is minimal since it is cofibrant and $\del_{\Mm_n}\delb_{\Mm_n}(\Mm_n)\subseteq \Mm_n^{(\geq 2)}$. Indeed,
    \begin{itemize}
        \item $(\del_{\Mm_n}\delb_{\Mm_n})_{|\Mm'}=\del_{\Mm'}\delb_{\Mm'}$, but $\del_{\Mm'},\delb_{\Mm'}$ are already minimal and
        \item $(\del_{\Mm_n}\delb_{\Mm_n})_{|E}:E\xrightarrow[]{\ov d_n}\Mm_{n-1}(n)\oplus E'\xrightarrow[]{\del\oplus d'}\Mm_{n-1}(n)\subseteq \Mm_{n}
        ^{(\geq 2)}.$ 
    \end{itemize}
    This completes the proof.
\end{proof}

\begin{prop}
A morphims $\phi\colon\Mm\to \Nn$ of operads between two minimal operads $\Mm=(\TT(M),\del,\delb)$ and $\Nn=(\TT(N),\del,\delb)$ is an isomorphism if and only if it is a pluripotential weak equivalence.
\label{weakEqIsIsoForMinimal}
\end{prop}

\begin{proof}
Let us first introduce some notation.
Let $M=\{M(n)\}_{n\geq 1}$ be a $\bbsmod$-module and let $\Mm=\TT(M)$.  We will
denote $M_{<k}$ the $\SS\,$-module defined by
$$M_{< k}(n):=
\begin{cases}
M(n),\text{ if }k< n,\\
0,\text{ otherwise,}
\end{cases}$$
and $M_{\leq k}$ the $\SS\,$-module defined by
$$M_{\leq k}(n):=
\begin{cases}
M(n),\text{ if }k\leq n,\\
0,\text{ otherwise.}
\end{cases}$$
Finally, $\Mm_{\leq k}:=\TT(M_{\leq k})$.

Now, one implication is clear since an isomorphism is already a weak equivalence. Let us prove the other implication: We will prove, by induction, that the restriction $\phi_n:\Mm_{\leq n}\to \Nn_{\leq n}$, is an isomorphism for each $n$. This will imply that $\phi:\TT(M)\to \TT(N)$ is an isomorphism too. 

Since $\Mm=(\TT(M),\del,\delb)$ and $\Nn(\TT(N),\del,\delb)$ are minimal, we have that $\del\delb(\Mm(2))=0$ and $\del\delb(\Nn(2))=0$. Since a pluripotential weak equivalence between minimal bicomplexes is an isomorphism, $\phi_2:\Mm(2)\to \Nn(2)$ is an isomorphism. Now, suppose we have already proved that $\phi_n$ is an isomorphism. We have two short exact sequences
\[0\to \Mm_{\leq n}(n+1)\to \Mm(n+1)\to M(n+1)\to 0,\]
and
\[0\to \Nn_{\leq n}(n+1)\to \Nn(n+1)\to N(n+1)\to 0.\]

As a result, again by Lemma~\ref{lemm:long-exact-seq}, we obtain the following two commutative diagrams, where the rows are exact:
\[
\begin{adjustbox}{max width=\textwidth,center}
\begin{tikzcd}[ampersand replacement=\&,column sep=tiny]
    	{\A(\Mm_{\leq n}(n+1))} \& {\A(\Mm(n+1))} \& {\A(M(n+1))} \& {\BC(\Mm_{\leq n}(n+1))} \& {\BC(\Mm(n+1))}\\
    	{\A(\Nn_{\leq n}(n+1))} \& {\A(\Nn(n+1))} \& {\A(N(n+1))} \& {\BC(\Nn_{\leq n}(n+1))} \& {\BC(\Nn(n+1))}
    	\arrow[from=1-1, to=1-2]
    	\arrow[from=1-2, to=1-3]
    	\arrow[from=1-3, to=1-4]
    	\arrow[from=1-4, to=1-5]
    	\arrow[from=2-4, to=2-5]
    	\arrow[from=2-3, to=2-4]
    	\arrow[from=2-2, to=2-3]
    	\arrow["\cong"', from=1-2, to=2-2]
    	\arrow[from=1-3, to=2-3]
    	\arrow["\cong"', from=1-4, to=2-4]
    	\arrow["\cong"', from=1-5, to=2-5]
    	\arrow["\cong"', from=1-1, to=2-1]
    	\arrow[from=2-1, to=2-2]
    \end{tikzcd}
    \end{adjustbox}
    \]   
and
\[\begin{adjustbox}{max width=\textwidth,center}\begin{tikzcd}[ampersand replacement=\&,column sep=tiny]
    	{\BC(\Mm_{\leq n}(n+1))} \& {\BC(\Mm(n+1))} \& {\BC(M(n+1))} \& {H_{\Bb_{p,q}}(\Mm_{\leq n}(n+1))} \& {H_{\Bb_{p,q}}(\Mm(n+1))}\\
    	{\BC(\Nn_{\leq n}(n+1))} \& {\BC(\Nn(n+1))} \& {\BC(N(n+1))} \& {H_{\Bb_{p,q}}(\Nn_{\leq n}(n+1))} \& {H_{\Bb_{p,q}}(\Nn(n+1))}
    	\arrow[from=1-1, to=1-2]
    	\arrow[from=1-2, to=1-3]
    	\arrow[from=1-3, to=1-4]
    	\arrow[from=1-4, to=1-5]
    	\arrow[from=2-4, to=2-5]
    	\arrow[from=2-3, to=2-4]
    	\arrow[from=2-2, to=2-3]
    	\arrow["\cong"', from=1-2, to=2-2]
    	\arrow[from=1-3, to=2-3]
    	\arrow["\cong"', from=1-4, to=2-4]
    	\arrow["\cong"', from=1-5, to=2-5]
    	\arrow["\cong"', from=1-1, to=2-1]
    	\arrow[from=2-1, to=2-2]
    \end{tikzcd}\end{adjustbox}\]  
by the Five Lemma, it follows that $\phi(n+1)\colon M(n+1)\to N(n+1)$ is a pluripotenital weak equivalence.
The minimality of $\Mm$ implies that the differentials $\del$ and $\delb$ on \[M(n+1)\cong\Mm(n+1)/\Mm_{\leq n}(n+1)\] verify $\del\delb=0$, and the same holds true for $\Nn$. Again, this implies $\phi(n+1)\colon M(n+1)\to N(n+1)$ is an isomorphism. We have proved that $\phi_{n+1}$ induces an isomorphism on the generators, hence, the map \[\phi_{n+1}:\TT(M(\leq n+1))\to \TT(N(\leq n+1))\] is also an isomorphism.
\end{proof}

\begin{defi}
    A \textit{minimal model} of an operad $\Pp$ in bicomplexes is given by a minimal operad $\Mm$ together with a pluripotential weak equivalence $\Mm\xrightarrow{\sim} \Pp$.
\end{defi}

\begin{theo}[Uniqueness of minimal models]
Minimal models are unique up to isomorphism. 
\end{theo}
\begin{proof}
    Let $\alpha:\Mm\xrightarrow[]{\sim}\Pp$ and $\alpha':\Nn\xrightarrow[]{\sim}\Pp$ two minimal models of an operad $\Pp$. By Proposition~\ref{LiftingUniqueUptoHomCof}, there is a map $\phi:\Mm\to\Nn$ unique up to homotopy such that the following diagram commutes up to homotopy
    \[\begin{tikzcd}
    	& \Nn \\
    	\Mm & \Pp.
    	\arrow["\phi", dashed, from=2-1, to=1-2]
    	\arrow["\sim", shift right, from=2-1, to=2-2]
    	\arrow["{\rotatebox{90}{$\sim$}}", from=1-2, to=2-2]
    \end{tikzcd}\]
    We obtain $\alpha'(n)\circ \phi(n)\simeq \alpha(n)$ for every $n$ and, by Lemma~\ref{prop:HomotopicMapsAreEqualInHomology}, $\phi$ is a pluripotential weak equivalence. By Proposition~\ref{weakEqIsIsoForMinimal}, $\phi:\Mm\to\Nn$ is an isomorphism.
\end{proof}

\section{Koszul duality}\label{sec:koszulduality}

In this section, we introduce the notion of a pluripotential twisting morphism and develop pluripotential Koszul duality for quadratic operads in \(\mathrm{Vect}_\kk\). To do so, we make use of the inflation functor defined in \cite{MaSo}, which allows us to use many of the computations in the dg-case.

\subsection{Twisting morphisms and twisted composite products}\label{subsec:TwistingMorph}
We introduce operadic twisting morphisms for operads in bicomplexes.
As in the classical setting, they play a central role in the bar-cobar constructions. We then use twisting morphisms to define twisted composite products, which lead to the notion of Koszul operads. 

Let $\Pp$ and  $\Cc$ be an operad and a cooperad in bicomplexes. 
Consider the $\mathrm{bb}\text{-}\SS\,$-module $\underline{\Hom}(\Cc, \Pp)$ given by
\[
\underline{\Hom}(\Cc,\Pp)(n)^{p,q}:= \underline{\Hom}(\Cc(n), \Pp(n))^{p,q}
\]
with right action of the symmetric group given by:
\[
f^\sigma(x) := f(x^{\sigma^{-1}})^\sigma.
\]

As in the dg setting, the pre-Lie product on $\underline{\Hom}(\Cc, \Pp)$ is defined as
\begin{gather}
    f\star g :=\Cc\xrightarrow{\Delta_{(1)}} \Cc \circ_{(1)} \Cc\xrightarrow{f \circ_{(1)} g}\Pp \circ_{(1)} \Pp
\xrightarrow{\gamma_{(1)}}\Pp.
\label{preLieprod}
\end{gather}
The space of invariant elements of $\underline{\Hom}(\Cc, \Pp)$ under the symmetric action
\[\underline{\Hom}_\SS(\Cc, \Pp)\]
is stable under the pre-Lie product. The two components $\del$ and $\delb$ are derivations with respect to the pre-Lie product $\star$ (see \cite[Proposition 6.4.5]{LoVa12}), and we obtain
\begin{prop}
The product space of $\SS$-equivariant maps
    \[
    \left( \underline{\Hom}_\SS(\Cc, \Pp), \star, \del, \delb \right)
    \]
    is a bidifferential bigraded pre-Lie algebra.
\end{prop}

Recall that a twisting morphism in the dg case is an equivariant map \(\varphi:\Cc\to \Pp\) of degree \(-1\) satisfying
\[[d,\varphi]+\varphi\star\varphi=0.\]
With this in mind, we introduce the following bigraded version of a twisting morphism.
\begin{defi}\label{defi:twisting}
A \textit{twisting morphism} $\varphi=(\varsh,\varsv):\Cc\to\Pp$ consists of two maps $\varsh,\varsv\in \underline{\Hom}_\SS(\Cc,\Pp)$ of bidegrees $(1,0)$ and $(0,1)$, respectively, satisfying:
    \[\begin{cases}
        [\del,\varsh]+\varsh\star\varsh=0,\\
        [\delb,\varsh] + \varsh\star\varsv+[\del,\varsv] + \varsv\star\varsh=0,\\
        [\delb,\varsv] + \varsv\star\varsv=0.
    \end{cases}\]
\end{defi}

\begin{rema}\label{rema:totaltwisting}
    Let \(\varphi:\Cc\to\Pp\) be an equivariant map such that it splits into two components of bidegrees \((1,0)\) and \((0,1)\), \(\varphi=\varsh+\varsv\). Then, \(\varphi\) is a twisting morphism in the dg sense with respect to the total differential \(d=\del+\delb\) if, and only if, \((\varsh,\varsv)\) is a twisting morphism in the bigraded sense.
\end{rema}

A \emph{left module} over $\Pp$ is a $\bbsmod$-module $M$ equipped with a left action 
\[\lambda:\Pp\circ M\to M\] which commutes with the differentials.

Let $N$ be a $\bbsmod$-module. The \emph{free left $\Pp$ module} on $N$ is given by
\[(\Pp\circ N, \del_{\Pp\circ N},\delb_{\Pp\circ N}),\]
where the action $\lambda$ is given by the operadic composition. 

Any linear map $\alpha: \Cc\to \Pp$ from a cooperad $(\Cc,\Delta,\epsilon)$ to an operad $(\Pp,\gamma,u)$, defines a map 
\[\Cc\xrightarrow{\Delta}\Cc\circ \Cc\xrightarrow{\alpha\circ\id}\Pp\circ \Cc\]
which, by Proposition 6.3.9 in \cite{LoVa12},  induces a unique derivation on $\Pp\circ \Cc$, explicitly given by
\[d=(\gamma_{(1)}\circ \id_\Cc)(\id_\Pp\circ'[(\alpha\circ\id_{\Cc})\Delta].\]

Moreover, for any $\alpha,\beta\in  \underline{\Hom}_\SS(\Cc,\Pp)$, the following relation is satisfied
\[d_{\alpha\star\beta} = d_{\alpha}\circ d_\beta,\]
see, for instance, \cite[Proposition 6.4.7]{LoVa12}.

Let $\alpha\colon\Cc\to\Pp$ be a twisting morphism. Then, we define the \emph{left twisted composite product}  as
\[\Pp\circ_\alpha\Cc:=(\Pp\circ\Cc,\del_{\Pp\circ\Cc}+d_{\alpha^{10}},\delb_{\Pp\circ\Cc}+d_{\alpha^{01}}).\]
\begin{prop}\label{prop:TwistedComposite}
The differentials of the twisted composite product satisfy the bicomplex equations:
\[
\begin{cases}
     (\del_{\Pp\circ\Cc}+d_{\alpha^{10}})^2=0,\\ (\delb_{\Pp\circ\Cc}+d_{\alpha^{01}})^2=0,\\
     (\del_{\Pp\circ\Cc}+d_{\alpha^{10}})(\delb_{\Pp\circ\Cc}+d_{\alpha^{01}}) + (\delb_{\Pp\circ\Cc}+d_{\alpha^{01}})(\del_{\Pp\circ\Cc}+d_{\alpha^{10}})=0.    
\end{cases}
\]

\end{prop}
\begin{proof}
    By Remark~\ref{rema:totaltwisting}, \(\varphi=\varphi^{10}+\varphi^{01}\) is a dg twisting morphism. We can endow \(\Pp\circ\Cc\) with a total differential of degree 1 \[d+d_\varphi=\del_{\Pp\circ\Cc}+d_{\varphi^{10}}+\delb_{\Pp\circ\Cc}+d_{\varphi^{01}}.\] Thus, we have \((d+d_\varphi)^2=d_{[d,\varphi]+\varphi\star\varphi}\), which is zero because \(\varphi\) is a dg twisting morphism. Since \(d+d_{\varphi}\) splits into the two components
    \[\del_{\Pp\circ\Cc}+d_{\varphi^{10}}\quad\text{and}\quad\delb_{\Pp\circ\Cc}+d_{\varphi^{01}}\]
    of bidegree \((1,0)\) and \((0,1)\), by degree reasons we obtain the desired result.
\end{proof}

\subsection{Stairway to heaven}\label{sec:staircase}
We introduce a family of minimal bicomplexes, which will serve as a minimal generalization of the shift functor, and state some properties. This family is the key ingredient for constructing the inflation functor defined in \cite{MaSo}.

Let us consider the following family of bicomplexes, $\{\ele_n\}_{n\in \ZZ_{\leq 0}}$, where
\[
\ele_n^{p,q} :=
\begin{cases}
    \kk & \text{if } p, q \leq 0 \text{ and } p + q =  n, \\
    \kk & \text{if } p, q \leq -1 \text{ and } p + q =  n - 1, \\
    0 & \text{otherwise}.
\end{cases}
\]

Since $\ele_n^{p,q}$ is either $\kk$ or trivial, in the first case,  ${(p,q)}_n$ will denote the unit in $\kk$, while in the second case it will denote the zero element. The differentials in the bicomplex $\ele_n $ are defined on the generators by
\[
\del (p,q)_n = q \cdot (p+1, q)_n \quad \text{and} \quad \delb (p,q)_n = -p \cdot (p, q+1)_n.
\]

\begin{exam}
The bicomplexes $\ele_0$, $\ele_{-1}$ and $\ele_{-2}$ are illustrated below.
Denote by $\kk_{(p,q)}$ the graded vector space $\kk$ concentrated in bidegree $(p,q)$. Then
    
\[  
    \ele_{0} = \begin{tikzcd}
    \kk_{(0,0)}
    \end{tikzcd} \qquad
    \ele_{-1} = \begin{tikzcd}
    \kk  &   \\
     \kk_{(\mep1,\mep1)} \arrow[u, "\cdot 1"] \arrow[r, "\cdot (-1)"] & \kk 
    \end{tikzcd} \qquad
    \ele_{-2} = \begin{tikzcd}
    \kk  &  &  \\
     \kk_{(\mep2,\mep1)} \arrow[u, "\cdot 2"] \arrow[r, "\cdot (-1)"] & \kk & \\
      & \kk_{(\mep1,\mep2)}\arrow[r, "\cdot (-2)"] \arrow[u, "\cdot 1"] & \kk
    \end{tikzcd}
\]
\end{exam}

\begin{rema}
    Note that the bicomplex \(\ele_{\mep1}\) is precisely the bicomplex \(\ele\) in the definition of the shift functor $\ele: \bicok\to \bicok$ defined in Section~\ref{subsec:preliminarsbico}. Moreover, \(\ele_{-n}\) is pluripotential weakly equivalent to the \(n\)-fold tensor product 
    of \(\ele\), that is,
    \[
    \ele_{-n} \;\simeq\; \ele^{\otimes n}.
    \]
    The key advantage of \(\ele_{-n}\) is that it is a minimal bicomplex (i.e., \(\del\delb=0\)), whereas \(\ele^{\otimes n}\) is not.
\end{rema}

We introduce the following bicomplex maps
\[\iota_{k,\ell}:\ele_{k+\ell}\to \ele_k\otimes\ele_\ell\]
of bidegree $(0,0)$ for every $k,\ell\in\ZZ_{\leq 0}$ given by
\[\iota_{k,\ell}(p,q)_{k+\ell}=\sum_{\substack{\alpha+\delta=p\\\beta+\gamma=q}}(-1)^{(\alpha+\beta)(\delta+\gamma+\ell)}(\alpha,\beta)_k\otimes(\delta,\gamma)_\ell.\]
The following is proved in \cite{MaSo}
\begin{prop}\label{prop:iotacoass}
The maps are \textit{coassociative} in the sense that
\[(\iota_{k_1,k_2}\otimes \id)\iota_{k_1+k_2,k_3}=(\id\otimes\iota_{k_2,k_2})\iota_{k_1,k_2+k_3}\]
for any \(k_1,k_2,k_3\in \ZZ_{\leq 0}\).
\end{prop}

We can define a generalized map 
\[\iota_{k_0,\dots,k_r}:\ele_{k_0+\dots+k_r}\to \ele_{k_0}\otimes\dots\otimes \ele_{k_r}\]
for \(k_0,\dots,k_r\in\ZZ_{\leq 0}\) by recursively applying \(\iota_{k,\ell}\) which gives
\[\iota_{k_0,\dots,k_r}(p,q)_n=\sum_{\substack{p_0+\dots+p_r=p\\q_0+\dots+q_r=q}}(-1)^{\varepsilon}(p_0,q_0)_{k_0}\otimes\dots\otimes(p_r,q_r)_{k_r},\]
where 
\(\varepsilon:=\sum_{j=1}^r \left( \sum_{i=0}^{j-1}(p_i+q_i) \right)(p_j+q_j+k_j)\) and \(n=k_0+\dots+k_r\).

It follows from the proof of Proposition 4.1 in \cite{MaSo}, for the case \(r=2\), that:
\begin{prop}
    The maps \(\iota_{k_0,\dots,k_r}\) are pluripotential weak equivalences for every \(k_0,\dots,k_r\in\ZZ_{\leq 0}\).
\end{prop}

We define two linear maps, that will be used in Section~\ref{subsec:KoszulMinimalModels},
\[
\mathfrak{s}_n:\ele_{n}\to\ele_{n+1}\quad\text{and}\quad\ov{\mathfrak{s}}_n:\ele_{n}\to\ele_{n+1}
\]
of bidegree \((1,0)\) and \((0,1)\), respectively, for \(n\leq -1\) given by 
\[\mathfrak{s}(p,q)_n=(-1)^{p+q+n}(p+1,q)_{n+1},\quad \ov{\mathfrak{s}}(p,q)_n=(-1)^{p+q+n}(p,q+1)_{n+1}.
\]
\begin{rema}\label{rema:fraksdel}
    These maps satisfy the following
\[
\del \mathfrak{s}+\mathfrak{s}\del=0,\quad\delb\ov{\mathfrak{s}}+\ov{\mathfrak{s}}\delb=0\quad\text{and}\quad \del\ov{\mathfrak{s}}+\delb \mathfrak{s}+\ov{\mathfrak{s}}\del+\mathfrak{s}\delb=0.
\]
\end{rema}

The maps \(\iota\), \(\mathfrak{s}\) and \(\ov{\mathfrak{s}}\) satisfy compatibility conditions:
\begin{prop}\label{prop:fraksiota} The following hold
        \begin{enumerate}
        \item \((\mathfrak{s}_{k-1}\otimes\id)\iota_{k-1,\ell}=(-1)^{k}(\id\otimes \mathfrak{s}_{\ell-1})\iota_{k,\ell-1}=\iota_{k,\ell}\mathfrak{s}_{k+\ell-1}\),\label{item:1propiotas}
        \item \((\ov{\mathfrak{s}}_{k-1}\otimes\id)\iota_{k-1,\ell}=(-1)^{k}(\id\otimes \ov{\mathfrak{s}}_{\ell-1})\iota_{k,\ell-1}=\iota_{k,\ell}\ov{\mathfrak{s}}_{k+\ell-1}\).\label{item:2propiotas}
    \end{enumerate}
\end{prop}
\begin{proof}
        Let us prove (\ref{item:1propiotas}), (\ref{item:2propiotas}) is analogous. On the one hand, we have
    \begin{align*}
        (\mathfrak{s}\otimes\id)\iota_{k-1,\ell}(p,q)_{k+\ell-1}&=\sum_{\substack{p_1+p_2=p\\q_1+q_2=q}}(-1)^{d_1+k-1+d_1(d_2+\ell)}(p_1+1,q_1)_{k}\otimes(p_2,q_2)_{\ell}\\
        & =\sum_{\substack{p_1+p_2=p+1\\q_1+q_2=q}}(-1)^{d_1+k+d_2+\ell+d_1(d_2+\ell)}(p_1,q_1)_{k}\otimes(p_2,q_2)_{\ell},
    \end{align*}
        where \(d_i=p_i+q_i\). On the other hand,
    \begin{align*}
        (\id\otimes \mathfrak{s})\iota_{k,\ell-1}(p,q)_{k+\ell-1}&=\sum_{\substack{p_1+p_2=p\\q_1+q_2=q}}(-1)^{d_2+\ell-1+d_1(d_2+\ell-1)+d_1}(p_1,q_1)_{k}\otimes(p_2+1,q_2)_{\ell}\\
        & =\sum_{\substack{p_1+p_2=p+1\\q_1+q_2=q}}(-1)^{d_1+d_2+\ell+d_1(d_2+\ell)}(p_1,q_1)_{k}\otimes(p_2,q_2)_{\ell}.
    \end{align*}
    Finally,
    \begin{align*}
        \iota_{k,\ell}\mathfrak{s}(p,q)_{k+\ell-1}&=\sum_{\substack{p_1+p_2=p+1\\q_1+q_2=q}}(-1)^{d_1+d_2+k+\ell+d_1(d_2+\ell)}(p_1,q_1)_{k}\otimes(p_2,q_2)_{\ell}.\qedhere
    \end{align*}
\end{proof}

\subsection{The inflation functor}
We next review the definition of the \textit{inflation functor}
constructed in \cite{MaSo}. It is a functor
\[\Inf: \Ch^{\leq 0}_\kk \to \bicok^{\doq}\]
where \(\Ch^{\leq 0}\) is the category of cochain complexes over \(\kk\) concentrated in non-positive degrees and \(\bico^{\doq}\) denotes the category of bicomplexes concentrated in the 3rd quadrant, that is \(B^{p,q}=0\) for \(p> 0\) or \(q> 0\). Given a cochain complex concentrated in negative degrees \((C,d)\)
it is defined as follows:
\[\Inf(C) = \left( \bigoplus_{n \in \mathbb{Z}_{\leq0}} \ele_n \otimes C^n, \del, \delb \right),\]
where we are considering $C^n$ as a bicomplex sitting in degree $(0,0)$ and \(\ele_n\) is the bicomplex defined in Section~\ref{sec:staircase}. The differentials \(\del\) and \(\delb\) are defined by:
\begin{equation}
    \label{eq:infdiff}
    \del((p,q)_n \otimes c) = q(p+1,q)_n \otimes c + (-1)^{p+q+n} (p+1, q)_{n+1} \otimes d c,
\end{equation}
\[\delb((p,q)_n \otimes c) = -p(p,q+1)_n  \otimes c + (-1)^{p+q+n} (p, q+1)_{n+1} \otimes d c.\]
The inflation functor is symmetric colax monoidal, and maps quasi-isomorphisms to pluripotential weak equivalences (see \cite{MaSo}).

We now adapt the inflation functor to \(\SS\,\mep\)modules of cochain complexes and bicomplexes:
\[
\Inf \colon \SS\,\mep\Ch^{\leq 0}_\kk \longrightarrow \SS\,\mep\bicok^{\doq}.
\]

Let 
\(
\rho \colon \SS_n \to \mathrm{Aut}_\kk(C),
\)
be a right \(\SS_n\) action on a cochain complex \(C\). The inflation functor induces a group morphism
\[
\mathrm{Aut}_\kk(C) \to \mathrm{Aut}_\kk(\Inf(C))
\]
given by
\[
\Inf(f)((p,q)_n \otimes c) = (p,q)_n \otimes f(c).
\]
By composition, we obtain a group morphism
\[
\begin{tikzcd}[ampersand replacement=\&]
	{\SS_n} \\
	{\mathrm{Aut}_\kk(C)} \& {\mathrm{Aut}_\kk(\Inf(C))}
	\arrow["\rho"', from=1-1, to=2-1]
	\arrow[dashed, from=1-1, to=2-2]
	\arrow[from=2-1, to=2-2]
\end{tikzcd}
\]
which defines a right action on \(\mathrm{Aut}_\kk(\Inf(C))\).

Let \(M\) be in \(\dgsmod\), then we define
\[\Inf(M)(n):=\Inf(M(n)).\]

The categories \((\dgsmod, \circ, \mathrm{I})\) and  \((\bicosmod^{\doq}, \circ, \mathrm{I})\) are monoidal. We have:
\begin{prop}\label{prop:SimInfColax}
    The functor 
    \[
    \Inf\colon\dgsmod\to \bicosmodn
    \]
    is  colax monoidal. Moreover, for all $M$ and $N$ in $\dgsmod$ the comparison map
    \[ \Inf(M \circ N) \to \Inf(M) \circ \Inf(N)\]
    is a weak equivalence, i.e. an arity-wise pluripotential weak equivalence.
\end{prop}
\begin{proof}
   We will construct natural equivariant maps
    \[
    \psi_{M,N}(n):\Inf(M\circ N)(n)\to (\Inf(M)\circ\Inf( N))(n)
    \]
    for any   \(\mathrm{dg}\text{-}\SS\,\)-modules  \(M\) and \( N\) in non-positive degrees,
making the following diagram commute:
\[\begin{tikzcd}[ampersand replacement=\&]
	{\Inf\big((M\circ N)\circ P\big)} \& {\Inf\big(M\circ( N\circ P)\big)} \\
	{\Inf(M\circ N)\circ\Inf( P)} \& {\Inf(M)\circ\Inf( N\circ P)} \\
	{\big(\Inf(M)\circ\Inf( N)\big)\circ\Inf( P)} \& {\Inf(M)\circ\big(\Inf( N)\circ\Inf( P)\big)}
	\arrow["\cong", from=1-1, to=1-2]
	\arrow["{\psi_{M\circ N, P}}"', from=1-1, to=2-1]
	\arrow["{\psi_{M, N\circ P}}", from=1-2, to=2-2]
	\arrow["{\psi_{M, N}\circ\id}"', from=2-1, to=3-1]
	\arrow["{\id\circ\psi_{ N, P}}", from=2-2, to=3-2]
	\arrow["\cong", from=3-1, to=3-2]
\end{tikzcd}\]
and satisfying the unitality conditions
 with respect to $\id:\Inf(\mathrm{I})=\mathrm{I}\to\mathrm{I}$. The horizontal maps in the above square are induced by the symmetric monoidal structures on cochain complexes and bicomplexes, respectively.
    
First, we have
    \[
    \Inf(M\circ N)(n)= \bigoplus_{s\geq 0}\bigoplus_{\substack{i_1+\dots+i_k=n\\{s_0}+s_1+\dots+s_k=s}}\ele_sM(k)^{s_0}\otimes_{\SS_k}\left(\mathrm{Ind}_{\SS_{i_1}\times\dots\times\SS_{i_k}}^{\SS_n}  N(i_1)^{s_1}\otimes\dots\otimes  N(i_k)^{s_k}\right)
    \]
    and 
    \begin{align*}
       &(\Inf(M)\circ\Inf( N))(n)=\\
       &\bigoplus_{s\geq 0}\bigoplus_{\substack{i_1+\dots+i_k=n\\s_0+s_1+\dots+s_k=s}}\ele_{s_0}M(k)^{s_0}\otimes_{\SS_k}\left(\mathrm{Ind}_{\SS_{i_1}\times\dots\times\SS_{i_k}}^{\SS_n} \ele_{s_1} N(i_1)^{s_1}\otimes\dots\otimes \ele_{s_k} N(i_k)^{s_k}\right).
    \end{align*}
    
    To compare these two expressions, we use the maps
    \[
    \iota_{s_0,s_1,\dots,s_k}:\ele_s\to \ele_{s_0}\otimes\ele_{s_1}\otimes\dots\otimes\ele_{s_k}
    \]
    defined in Section~\ref{sec:staircase}, which are given by
    \[
    (p,q)_r\mapsto\sum_{\substack{p_0+\dots+p_k=p\\q_0+\dots+q_k=q}}(-1)^{\sum_{i}(d_0+\dots+d_{i-1})(d_{i}+r_i)}(p_0,q_0)_{s_0}\otimes\dots\otimes(p_k,q_k)_{s_k}
    \]
Elements of \(\Inf\big((M\circ N)\big)(n)\) are linear combinations of formal composites
\[
(p,q)_s\otimes(\nu_{0};\nu_{1},\dots,\nu_{k})\sigma
\]
where \(s=\mid \nu_0\mid + \dots+\mid \nu_k\mid\), \(\sigma\in\SS_n\), \(\nu_{0}\in M(k)\) and \(\nu_{1}\in N(i_1),\dots,\nu_{k}\in  N(i_k)\) with \(i_1+\dots+i_k=n\). With this description we define the comparison map
\[
\psi_{M, N}:\Inf(M\circ N)\to\Inf(M)\circ\Inf( N)
\]
by letting
\[
\psi_{M, N}(n)\big((p,q)_s\otimes(\nu_{0};\nu_{1},\dots,\nu_{k})\sigma\big):=\sum_{\substack{p_0+\dots+p_k=p\\q_0+\dots+q_k=q}}(-1)^{\varepsilon}(\nu_{0}^{p_0,q_0};\nu_{1}^{p_1,q_1},\dots,\nu_{k}^{p_k,q_k})\sigma
\]
where \(\varepsilon=\sum_{i}(d_0+\dots+d_{i-1})(d_{i}+r_i)\) and \(\nu_{r}^{p_r,q_r}:=(p_r,q_r)_{d_r}\otimes \nu_r\) and \(d_r=|\nu_r|\). The map \(\psi(n)\) is obviously equivariant. 

The fact that it is a map of bicomplexes and the required properties of $\psi_{M, N}$ are now directly deduced from the symmetric monoidality of the inflation functor for the case of chain complexes of 
\(\kk\)-modules, Proposition~4.1 in \cite{MaSo}.
Namely, the above diagram commutes and $\psi_{M, N}$ is a weak equivalence.
\end{proof}

\begin{prop}\label{prop:SimInfWeakEq}
    The inflation functor
    \[
    \Inf\colon\dgsmod\to \bicosmod^{\doq}
    \]
    sends quasi-isomorphisms to pluripotential weak equivalences.
\end{prop}
\begin{proof}
    Let \(\vp\colon M\to N\) be a quasi-isomorphism in \(\dgsmod\), i.e. a sequence of quasi-isomorphisms of cochain complexes \(\vp(n)\colon M(n)\to N(n)\) compatible with the \(\SS_n\) action. Then, \[\Inf(\vp)\colon\Inf(M)\to \Inf(N)\] is a sequence of maps of bicomplexes 
    \[\Inf(\vp)(n)\colon\Inf(M(n))\to \Inf(N(n))\] compatible with the action of \(\SS_n\). By Proposition~3.9 in \cite{MaSo}, the functor \(\Inf\colon \Ch^{\leq 0}\to\bicok^{\doq}\) sends quasi-isomorphisms to pluripotential weak equivalences and we obtain the desired result.
\end{proof}

Of course, the inflation functor can be extended in the same way to sequences 
\(\mathcal{M} = \{\mathcal{M}(n)\}_{n \ge 0}\)
without the \(\mathbb{S}_n\)-action, so it may also be used in the corresponding 
non-symmetric setting.

\subsection{Koszul dual cooperad}\label{subsec:KoszulDualCoop}
We next define the pluripotential Koszul dual cooperad of a quadratic operad \(\Pp\) in $\kk$-vector spaces. Recall that, by definition, an \emph{operadic quadratic data} \((E,R)\) in \(\mathrm{Vect}_\kk\) consists of an \(\SS\)-module \(E\) and a sub-\(\SS\)-module \(R \subseteq \TT(E)^{(2)}\), $E$ then denotes the generating vector space and $R$ the set of relations. A \emph{quadratic operad} associated to the quadratic data \((E,R)\) is the operad \(\Pp(E,R):=\TT(E)/(R)\).

For the remainder of the section, \(\Pp\) will denote a quadratic operad in \(\mathrm{Vect}_\kk\), we will view this operad as an operad in 
$\Ch_\kk$ or in $\bicok$ by placing it in degree 0 or bidegree $(0,0)$, with trivial differentials.  We will denote by \(\Pp^{\antshrk}=\Cc(sE,s^2R)\) the dg Koszul dual cooperad of \(\Pp\) and by $\Delta$ its cooperad structure (see Section 7.2 in \cite{LoVa12}).

Following the classical case, one is led to define
the pluripotential Koszul dual cooperad of $\Pp$ as
 $\Cc(\ele E,\ele_{\mep2} R)$.
We will instead use the inflation functor, so that we can recycle the classical computations of Koszul dual cooperads.

\begin{defi}Define the  \emph{pluripotential Koszul dual cooperad} of \(\Pp\) as
    \[
    \Pp^{\antshrk}_{pp}:=\Inf(\Pp^{\antshrk})
    \] 
    with the cooperad structure given by
    \[
    \Delta:\Inf(\Pp^{\antshrk})\xrightarrow{\Inf(\Delta)}\Inf(\Pp^{\antshrk}\circ\Pp^{\antshrk})\xrightarrow{\psi_{\Pp^{\antshrkind},\Pp^{\antshrkind}}}\Inf(\Pp^{\antshrk})\circ\Inf(\Pp^{\antshrk}),
    \]
    where $\psi$ denotes the colax natural transformation of Proposition~\ref{prop:SimInfColax}.
\end{defi}

\begin{rema}
    Note that \(\Pp^{\antshrk}\) can be viewed as a sub-cooperad of \(\Pp^{\antshrk}_{pp}\), given by the elements of bidegrees \((*,0)\) and, similarly, of bidegrees \((0,*)\).
\end{rema}

\begin{exam}
    The pluripotential Koszul dual of \(As\) (as a non symmetric operad) is given by
    \[
    As^{\antshrk}_{pp}(n)=\ele_{1-n}.
    \]
    Denote by $\mu_n^{p,q}$ an element of arity \(n\) and bidegree \((p,q)\). Then, the cooperad structure is given by
    \[
    \Delta(\mu_n^{p,q})=\sum_{i_1+\dots+i_k=n}\sum_{\substack{p_0+\dots+p_k=p\\q_0+\dots+q_k=q}}(-1)^{\theta+\varepsilon}(\mu_k^{p_0,q_0};\mu_{i_1}^{p_1,q_1},\dots,\mu_{i_k}^{p_k,q_k}),
    \] 
    where $\theta=\sum_{1\leq r<s\leq k}i_r(i_s+1)$ and $\varepsilon=\sum_{j=1}^k(p_0+q_0+\dots+p_{j-1}+q_{j-1})(p_j+q_j+i_j-1)$. Therefore, the infinitesimal decomposition gives
    \[
    \Delta_{(1)}(\mu_n^{p,q})=\sum_{\substack{k+\ell=n+1\\1\leq i\leq k}}\sum_{\substack{p_1+p_2=p\\q_1+q_2=q}}(-1)^{(i-1)(\ell+1)+\varepsilon}\mu_k^{p_1,q_1}\circ_i\mu_\ell^{p_2,q_2}
    ,
    \]
    where \(\varepsilon=(p_1+q_1)(p_2+q_2+\ell+1).\)
For the symmetric version, we have \[
    As^{\antshrk}_{pp}(n)=\ele_{1-n}\otimes\kk[\SS_n],
    \]
    with the \(\SS_n\) action being trivial on \(\ele_{1-n}\).
\end{exam}

\begin{rema}\label{rema:ppKoszulDualWeight}
    Note that, since \(\Pp^{\antshrk}\) is a sub-cooperad of the cofree cooperad, it is naturally weight graded. Then, the pluripotential Koszul dual \(\Pp^{\antshrk}_{pp}\) inherits a weight given by
    $\omega((p,q)_n\otimes \mu)=\omega(\mu)$
    where \(\mu\) is an element of \(\Pp^{\antshrk}\).
\end{rema}
We will next define a twisting morphism 
$\kappa:\Pp^{\antshrk}_{pp}\to\Pp$. 
Recall from Section~\ref{sec:staircase} that $\mathfrak{s}_1:\ele\to \kk$ is the map of bidegree $(1,0)$ projecting $\ele$ to $\kk$. Likewise, $\ov{\mathfrak{s}}_1:\ele \to\kk$ is the corresponding map of bidegree $(0,1)$.
We construct the following two equivariant maps:
\[
\kappa^{10} \colon \Pp^{\antshrk}_{pp}=\Inf(\Cc(sE,s^2R))\twoheadrightarrow\Inf(sE)=\ele\otimes E\xrightarrow{\mathfrak{s}_1\otimes \id} E\hookrightarrow \Pp=\Pp(E,R)
\]
and
\[
\kappa^{01} \colon \Pp^{\antshrk}_{pp}=\Inf(\Cc(sE,s^2R))\twoheadrightarrow\Inf(sE)=\ele\otimes E\xrightarrow{\ov{\mathfrak{s}}_1\otimes \id} E\hookrightarrow \Pp=\Pp(E,R).
\]

\begin{rema}\label{rema:kappa}
    Recall that there is a dg twisting morphism, which by abuse of notation we will still denote by \(\kappa\),
    \[\kappa: \Pp^{\antshrk}=\Cc(sE,s^2R)\twoheadrightarrow sE\rightarrow E\hookrightarrow \Pp=\Pp(E,R).
    \]
    In fact, \(\kappa^{10}=\mathfrak{s}\otimes\kappa\) and \(\kappa^{01}=\ov{\mathfrak{s}}\otimes\kappa\), where \(\mathfrak{s}(p,q)_n=\mathfrak{s}_1\) if \(n=1\) and zero otherwise and similarly for \(\ov{\mathfrak{s}}\).
\end{rema}

\begin{prop}
    The pair $\kappa=(\kappa^{10},\kappa^{01})\colon\Pp^{\antshrk}_{pp}\to\Pp$ is a twisting morphism.
\end{prop}
\begin{proof}
    A direct computation shows
    \[[\del,\kappa^{10}]=[\delb,\kappa^{01}]=[\del,\kappa^{01}]+[\delb,\kappa^{10}]=0.\]
    On the other hand, $\kappa$ is only non-zero in weight 1. So the pre-Lie products $\kappa^{i,1-i}\star\kappa^{j,1-j}$, with $i,j\in\{0,1\}$
    are zero except maybe in weight 2. But in weight 2 we have
    \[
    \Pp^{\antshrk(2)}_{pp}\to\Pp^{\antshrk(1)}_{pp}\circ_{(1)}\Pp^{\antshrk(1)}_{pp}\xrightarrow{\kappa^{**}\circ_{(1)}\kappa^{**}}(R)\rightarrow\TT(E)^{(2)}/(R)=\Pp
    \]
    which is the zero map.
\end{proof}

Recall that a dg-operad $\Pp$ is said to be \textit{Koszul} if and only if 
its associated Koszul complex is acyclic, so there is a quasi-isomorphism 
\[\Pp\circ_{\kappa}\Pp^{\antshrk}\stackrel{\sim}{\lra}\mathrm{I},\]
where \(\Pp\circ_\kappa\Pp^{\antshrk}\) is the Koszul complex of a quadratic dg operad (see Section 7.4 in \cite{LoVa12}).
In particular, this definition depends on the chosen class of weak equivalences, which in the dg case are just quasi-isomorphisms. In the pluripotential case, we may consider the analogous definition with the corresponding weak equivalences. 
To this end, we construct the twisted composite product \(\Pp\circ_\kappa\Pp_{pp}^{\antshrk}\) associated to the twisting morphism \(\kappa:\Pp_{pp}^{\antshrk}\to\Pp\), see Section~\ref{subsec:TwistingMorph} for details. The following proposition shows that if $\Pp$ is a Koszul operad in the dg sense, concentrated in degree 0, then it is also Koszul in the pluripotential sense.

\begin{prop}Let $\Pp$ be a quadratic operad in $\kk$-vector spaces that is Koszul as a dg-operad. 
Then, there is a pluripotential weak equivalence 
\[ \Pp\circ_{\kappa}\Pp^{\antshrk}_{pp}\stackrel{\sim}{\lra}\mathrm{I}.\]
\end{prop}
\begin{proof} By Proposition~\ref{prop:SimInfWeakEq}, the inflation functor preserves weak equivalences. 
    Since \(\Pp\) is Koszul as a dg-operad, we have a pluripotential weak equivalence
    \[
    \Inf(\Pp\circ_{\kappa}\Pp^{\antshrk})\xrightarrow{\sim}\mathrm{I}.
    \]

We define a map of bicomplexes \(\varphi:\Inf(\Pp\circ_{\kappa}\Pp^{\antshrk})\rightarrow \Pp\circ_\kappa\Pp^{\antshrk}_{pp}\)
using the
colax structure of the inflation functor given in Proposition~\ref{prop:SimInfColax}
    \[
\varphi:=\psi_{\Pp,\Pp^{\antshrkind}}:\Inf(\Pp\circ_{\kappa}\Pp^{\antshrk})\rightarrow \Pp\circ_\kappa\Inf(\Pp^{\antshrk})=\Pp\circ_\kappa\Pp^{\antshrk}_{pp}
    \]
This is a linear map of bigraded vector spaces, and it suffices to see that it is compatible with the twisting differentials.     
    Note that \(\Inf(\Pp)=\Pp\) since it is concentrated in degree 0.
 
    Since $\varphi$ is constructed using the colax monoidal structure of $\Inf$, it is clear that, when we forget the twisting differentials, we have a map of bicomplexes
    \[
    \Inf(\Pp\circ \Pp^{\antshrk})\rightarrow \Pp\circ\Inf(\Pp^{\antshrk})
    \]
    which is a pluripotential weak equivalence by Proposition~\ref{prop:SimInfColax}.
    Denote by \(d_\kappa\) and \(\ov d_\kappa\) the twisting differentials 
    \[
    d_\kappa, \ov d_\kappa:\Inf(\Pp\circ_\kappa\Pp^{\antshrk})\to\Inf(\Pp\circ_\kappa\Pp^{\antshrk})
    \] 
    of bidegree \((1,0)\) and of bidegree \((0,1)\), respectively, obtained after inflation of the Koszul complex. Explicitly, these are given by
    \[
    d_\kappa=\mathfrak{s}_n\otimes\left(\gamma_{(1)}\circ\id_{\Pp^{\antshrkind}})(\id_{\Pp}\circ'[(\kappa\circ\id_{\Pp^{\antshrkind}})\Delta]\right)\quad\text{and}\quad\ov{d}_\kappa=\ov{\mathfrak{s}}_n\otimes\left(\gamma_{(1)}\circ\id_{\Pp^{\antshrkind}})(\id_{\Pp}\circ'[(\kappa\circ\id_{\Pp^{\antshrkind}})\Delta]\right).
    \]
    By the definition of twisted composite products (Section~\ref{subsec:TwistingMorph}) and Remark~\ref{rema:kappa}, the 
    derivations 
    \[
    d_{\kappa^{10}},d_{\kappa^{01}}:\Pp\circ\Pp_{pp}^{\antshrk}\to \Pp\circ\Pp_{pp}^{\antshrk}
    \]
    are given by 
    \[
    (\gamma_{(1)}\circ\id_{\Pp^{\antshrkind}_{pp}})(\id_{\Pp}\circ'[(\mathfrak{s}\otimes\kappa)\circ\id_{\Pp^{\antshrkind}_{pp}})\Delta]\quad\text{and}\quad(\gamma_{(1)}\circ\id_{\Pp^{\antshrkind}_{pp}})(\id_{\Pp}\circ'[(\ov{\mathfrak{s}}\otimes\kappa)\circ\id_{\Pp^{\antshrkind}_{pp}})\Delta].
    \]
    Then, the map \(\varphi\) is a map of bicomplexes if \(\varphi d_{\kappa}=d_{\kappa^{10}}\varphi\) and \(\varphi\ov d_{\kappa}=d_{\kappa^{01}}\varphi\).

    Let us check the first identity.
    On the one hand, we have
    \begin{align*}
        \varphi d_{\kappa}\big((p,q)_{n}\otimes (\mu;\nu_{i_1},\dots,\nu_{i_k})\big)=\\
        =\sum_{j=1}^k\sum(-1)^{\theta+\sum_{s=1}^{j-1}d_s}(\iota\mathfrak{s}_n(p,q)_{n})\otimes (\mu\circ_{j}\kappa\nu_{i_j}^{(1)};\nu_{i_1},\dots,\nu_{i_j}^{(2)},\dots,\nu_{i_k}),
    \end{align*}
    where \(\nu_{i_j}^{(2)}=(\nu_{{i_j},1},\dots,\nu_{{i_j},\ell})\), \(\theta\) is the sign given by the decomposition map of \(\Pp^{\antshrk}\), \(d_s=|\nu_{i_s}|\) and 
    \[
    \iota=\iota_{d_1,\dots,{d_j}_1,\dots,{d_j}_\ell,\dots,d_k}:\ele_{{d_1},\dots,d_k}\to \ele_{d_1}\otimes\dots\otimes\ele_{d_k},
    \]
    defined in Section~\ref{sec:staircase}.
    On the other hand,     
    \begin{align*}
        d_{\kappa^{10}}\varphi\big((p,q)_{n}\otimes (\mu_n;\nu_{i_1},\dots,\nu_{i_k})\big)=\\
        =\sum_{j=1}^k(-1)^{\theta}\big((\id^{j-1}\otimes(\mathfrak{s}\otimes \id^{ \ell})\iota_{{d_j}_0,\dots,{d_j}_\ell}\otimes\id^{k-j})\iota_{d_1,\dots,d_k}(p,q)_{n}\big)(\mu\circ_{j}\kappa\nu_{i_j}^{(1)};\nu_{i_1},\dots,\nu_{i_j}^{(2)},\dots,\nu_{i_k}),
    \end{align*}
    where \(d_0=|\nu_{i_j}^{(1)}|\) and \(\mathfrak{s}(p,q)_n=\mathfrak{s}_1\) if \(n=1\) and zero otherwise.
    By coassociativy of the maps \(\iota\), Proposition~\ref{prop:iotacoass}, we have 
    \begin{align*}
        (\id^{\otimes j-1}\otimes(\mathfrak{s}\otimes \id^{\otimes \ell})\iota_{{d_j}_0,\dots,{d_j}_\ell}\otimes\id^{\otimes k-j})\iota_{d_1,\dots,d_k}(p,q)_{n}=\\
        =(\id^{\otimes j-1}\otimes\mathfrak{s}\otimes \id^{\otimes k+\ell-j})\iota_{d_1,\dots,{d_j}_0,\dots,{d_j}_\ell,\dots,d_k}(p,q)_{n},
    \end{align*}
    which by the compatibility relations between the maps \(\iota\) and \(\mathfrak{s}\), Proposition~\ref{prop:fraksiota}, gives
    \begin{align*}
       (-1)^{\sum_s^{j-1}d_s}\iota_{d_1,\dots,{d_j}_1,\dots,{d_j}_\ell,\dots,d_k}\mathfrak{s_n}(p,q)_{n}.
    \end{align*}
    This proves that \(\varphi d_{\kappa}=d_{\kappa^{10}}\varphi\).
     A similar computation shows \(\varphi\ov{d}_{\kappa}=d_{\kappa^{01}}\varphi\).

    This gives us a commutative diagram 
    \[\begin{tikzcd}[ampersand replacement=\&]
	{\Pp\circ_\kappa\Inf(\Pp^{\antshrk})} \& \mathrm{I} \\
	{\Inf(\Pp\circ_{\kappa}\Pp^{\antshrk})}
	\arrow[from=1-1, to=1-2]
	\arrow["\varphi",from=2-1, to=1-1]
	\arrow["\sim", from=2-1, to=1-2]
    \end{tikzcd}\]

    Thus, by the two-out-of-three property, it suffices to show that \(\varphi\) is a pluripotential weak equivalence, which we address in the remainder of the proof. 
    
    Since \(\Pp\) is concentrated in bidegree \((0,0)\), \(\Inf(\Pp\circ_\kappa\Pp^{\antshrk})\) and \(\Pp\circ_\kappa\Inf(\Pp^{\antshrk})\) are \(\bbsmod\)-modules in bounded bicomplexes, see Remark~\ref{rema:boundedbicos}. Therefore, to show \(\varphi\) is a pluripotential weak equivalence is enough to show it is a quasi-isomorphism with respect to \(\del\) and \(\delb\). 
    Let us focus on the $\del$-case, so we forget the differential $\delb$ and view $\varphi$ as a map of complexes with respect to the $\del$-component on both sides. The strategy will actually be to split the component $\del$ into weights, obtaining a new bicomplex structure. Indeed, the natural weight-grading of $\Pp$
    induces a weight-grading on \(\Pp\circ\Pp^{\antshrk}\) and the differential $\del$
    of $\Inf(\Pp\circ_\kappa\Pp^{\antshrk})$
    splits into two components of weight 0 and 1, respectively, $\del=\del_0+\del_1$,
    where $\del_1=d_\kappa$. So we get that \(\Inf(\Pp\circ\Pp^{\antshrk})\) is a bicomplex with total differential $\del$.
    A similar argument shows that \(\Pp\circ\Pp^{\antshrk}_{pp}\) with its $\delb$-differential is also a bicomplex. Moreover, $\varphi$ is a map of bicomplexes with respect to the weight-grading.
    By  Proposition~\ref{prop:SimInfColax}, $\varphi$ is a quasi-isomorphism with respect to the weight 0 component $\del_0$ of $\del$.
    Therefore, a standard spectral sequence argument shows it is a quasi-isomorphism with respect to the total differential. This proves that $H_\del(\varphi)$ is an isomorphism.
    A parallel argument shows $H_\delb(\varphi)$ is also an isomorphism.
\end{proof}

\subsection{Minimal models via Koszul duality}\label{subsec:KoszulMinimalModels}
We present a concrete and computationally effective method for constructing pluripotential minimal models for operads in $\mathrm{Vect}_\kk$ that are Koszul as a dg-operad, taking advantage of the dg setting. Fix \(\Pp\) a quadratic operad in \(\mathrm{Vect}_\kk\).

A \emph{coaugmented cooperad} is a cooperad \(\Cc\) equipped with a morphism 
\(\eta : \mathrm{I} \to \Cc\) of cooperads. Denote by \(\ov\Cc\) the coaugmentation coideal, that is \(\ov\Cc:=\Coker(\eta)\).

Let $s$ denote the $\mathbb{S}$-module in graded vector spaces defined by
\(
(0,\, \kk s,\, 0,\, \dots),
\)
where $\kk s$ is concentrated in degree $-1$.

Recall that the cobar construction of a coaugmented dg-cooperad  \(\Cc\) (see, for instance,  Section 6.5 \cite{LoVa12}) is given by
\[\cbr(\Cc)=(\TT(s\ov{\Cc}),d_1+d_2)\]
where \(s\Cc=s\otimes \Cc\), \(d_1\) is the derivation extending \(d_\Cc\) and \(d_2\) is the derivation extending 
\[\delta:s\ov\Cc\xrightarrow{\Delta_s\otimes \Delta_{(1)}} s\otimes s\otimes \ov\Cc\circ_{(1)}\ov\Cc\xrightarrow{\cong}s\ov\Cc\circ_{(1)}s\ov\Cc\]
where \(\Delta_s(s)=-s\otimes s\).

We define a quasi-free operad in bicomplexes
\[
\Pp_{\infty}^{pp}:=\left(\TT(\Inf(s\ov{\Pp^{\antshrk}})),\del_1+\del_2,\delb_1+\delb_2\right)
\]
where \(\del_1\) and \(\delb_1\) are the derivations extending the differentials of \(\Inf(s\ov{\Pp^{\antshrk}})\). 
The derivations $\del_2$ and $\delb_2$ are defined as follows.
Recall from Section~\ref{sec:staircase} the maps of bidegree \((1,0)\)
\[
\mathfrak{s}_n : \ele_n \longrightarrow \ele_{n+1}
\]
defined for \(n \leq 0\). 
Then \(\del_2\) is the derivation extending the following map:
\[
\Inf(s\ov{\Pp^{\antshrk}})\xrightarrow{\mathfrak{s}\otimes\delta}\Inf(s\ov{\Pp^{\antshrk}}\circ_{(1)}s\ov{\Pp^{\antshrk}})\xrightarrow{\psi_{s\Pp^{\antshrkind},s\Pp^{\antshrkind}}} \Inf(s\ov{\Pp^{\antshrk}})\circ_{(1)} \Inf(s\ov{\Pp^{\antshrk}})
\]
where $\mathfrak{s}$ is induced by the maps $\mathfrak{s}_n$, $\delta$ is the generating map for the derivation of the cobar construction given above, and  \(\psi_{s\Pp^{\antshrkind},s\Pp^{\antshrkind}}\) is the colax structure of the inflation functor, see Proposition~\ref{prop:SimInfColax}.

Explicitly, let \(\mu\) be an element in \(\Pp^{\antshrk}\) and denote \(\mu^{p,q}:=(p,q)_{| \mu|+1}\otimes s\mu\in\Inf(s\ov{\Pp^{\antshrk}})\). We have:
\[
(\del_1+\del_2)(\mu^{p,q})=q\mu^{p+1,q}+\sum_i\sum_{\substack{p_1+p_2=p+1\\q_1+q_1=q}}(-1)^{\varepsilon+|\mu_{i}'|}({\mu_i'}^{p_1,q_1}\circ_{e_i}{\mu_{i}''}^{p_2,q_2})\sigma_i
\]
where $\varepsilon=p+q+| \mu|+1+(p_1+q_1)(p_2+q_2+| \mu_{i}''|+1)$. Similarly, we use the maps \(\ov{\mathfrak{s}}\) to construct $\delb_2$ as the derivation extending
\[
\Inf(s\ov{\Pp^{\antshrk}})\xrightarrow{\ov{\mathfrak{s}}\otimes\delta}\Inf(s\ov{\Pp^{\antshrk}}\circ_{(1)}s\ov{\Pp^{\antshrk}})\xrightarrow{\psi_{s\Pp^{\antshrkind},s\Pp^{\antshrkind}}} \Inf(s\ov{\Pp^{\antshrk}})\circ_{(1)} \Inf(s\ov{\Pp^{\antshrk}}).
\]
In this case, we have:
\[
(\delb_1+\delb_2)(\mu^{p,q})=-p\mu^{p,q+1}+\sum_i\sum_{\substack{p_1+p_2=p\\q_1+q_1=q+1}}(-1)^{\varepsilon+|\mu_{i}'|}({\mu_i'}^{p_1,q_1}\circ_{e_i}{\mu_{i}''}^{p_2,q_2})\sigma_i.
\]

\begin{prop}The tuple $(\Pp^{pp}_\infty,\del,\delb)$ satisfies the bicomplex equations
  \[
    \del^2=\delb^2=\del\delb+\delb\del=0
    .\]
\end{prop}
\begin{proof}
We will prove that \(\del^2 = 0\). The other two identities follow by analogous arguments. 
Recall that the colax structure of the inflation functor involves the maps introduced in Section~\ref{sec:staircase}
\[
\iota_{k_1,\dots,k_n} \colon \ele_{k_1+\dots+k_n} \longrightarrow 
\ele_{k_1} \otimes \dots \otimes \ele_{k_n}.
\]
For the computations below, we write 
\(d_i=|s\mu_i|\).
    \begin{align*}
        \del\del((p,q)_{n}\otimes s\mu)
        &=\sum(-1)^{d_1+1}\left(\iota_{d_1,d_2}\mathfrak{s}\del_\Lpic(p,q)_{n}\right)\otimes s\mu_{(1)}\circ_{i}s\mu_{(2)}\\
        &+\sum(-1)^{d_1+1}\left(\del_\Lpic\iota_{d_1,d_2}\mathfrak{s}(p,q)_{n}\right)\otimes s\mu_{(1)}\circ_{i}s\mu_{(2)}\\
        &+\sum(-1)^{d_2+1}\left((\iota_{d_1,d_2}\mathfrak{s}\otimes\id)\iota_{d_1+d_2+1,d_3}\mathfrak{s}(p,q)_{n}\right)\otimes s\mu_{(1)}\circ_i s\mu_{(2)}\circ_j s\mu_{(3)}\\
        &+\sum(-1)^{d_1+d_2}(\id\otimes\iota_{d_2,d_3}\mathfrak{s})\iota_{d_1,d_2+d_3+1}\mathfrak{s}(p,q)_{n}\otimes s\mu_{(1)}\circ_i s\mu_{(2)}\circ_j s\mu_{(3)}.
    \end{align*}
    The two first summands vanish because \(\iota_{k,\ell}\) are maps of bicomplexes and by Remark~\ref{rema:fraksdel}. The two last summands vanish by the relations satisfied by \(\iota_{k,\ell}\) and \(\mathfrak{s}\), Proposition~\ref{prop:fraksiota}.
\end{proof}

\begin{rema}
    There is a naive cobar construction for a coaugmented cooperad \(\Cc\) in \(\bicok\) using the shift \(\antiele:\bicok\to\bicok\), defined in Section~\ref{subsec:preliminarsbico}, that mimics the cobar construction in the dg case. However, since the shift does not preserve minimality of bicomplexes, we will typically obtain an operad \(\Omega\Cc\) which is not minimal as an operad in \(\bicok\), see Definition~\ref{defi:minimal}. The idea of the construction presented here is shifting before and then apply the inflation functor, so that minimality is ensured. 
\end{rema}

\begin{theo}[Minimal model of a Koszul operad]\label{theo:MinimalModelsKoszul}
The operad
    \(\Pp_\infty^{pp}\) is the pluripotential minimal model of \(\Pp\).
\end{theo}
\begin{proof}
We first show that $\Pp_\infty^{pp}$ is a minimal operad, see Section~\ref{subsec:MinialModelsSulivan}.
We need to show that it is cofibrant and 
 $\del\delb$ decomposes, that is:
   \[\del\delb(\Inf(s\ov{\Pp^{\antshrk}}))\subseteq (\Pp_\infty^{pp})^{(\geq 2)}.\] 
Note first that  \(\Pp_\infty^{pp}\) is cofibrant since it is the colimit of a sequence of principal extensions starting from the initial operad
    \[
    \mathrm{I}\to\TT(\mathrm{I}\oplus V_2)\to\dots\to\TT(\Pp_{n-1}\oplus V_n)\to\dots\to \mathrm{colim}_n \Pp_n=\Pp_{\infty}^{pp}
    \]
    where the extensions are given by
    \[V_{2n}=\BC(\Inf(s\ov{\Pp^{\antshrk}}(n)))\text{ and }V_{2n+1}=\A(\Inf(s\ov{\Pp^{\antshrk}}(n))).\]
   By construction, \(\del_2\) and \(\delb_2\) decompose. Note that \(\del_1\) and \(\delb_1\) do not decompose in general but in each arity, \(\Inf(s\ov{\Pp^{\antshrk}})\) is a minimal bicomplex, and so \(\del_1\delb_1=0\). This shows  
that $\del\delb$ decomposes
   and so the operad is minimal.

   By  writing $\Pp=\Pp(E,R)$, define
    \(  \alpha\colon \Pp_\infty^{pp}\to \Pp\)  
    as the map extending 
    \[
    \Inf(s\ov{\Pp^{\antshrk}})=\Inf(s\Cc(s^{-1}E,s^{-2}R))\twoheadrightarrow \Inf(E)=E\hookrightarrow \Pp.\]
It only remains to show that $\alpha$ is a pluripotential weak equivalence. Since all bicomplexes are bounded, it suffices to show it induces isomorphisms in $H_\del$ and $H_\delb$, see Remark~\ref{rema:boundedbicos}.
Note that, by considering $\Pp^{\antshrk}$ in bidegree $(*,0)$, we may view it as a sub-cooperad of $\Pp_{pp}^{\antshrk}$. Moreover, we have $H_\del(\Pp_{pp}^{\antshrk}(n))=\Pp^{\antshrk}(n)$
so this inclusion of cooperads is an $H_\del$-isomorphism. 
This allows us to identify $\Pp_\infty$ as the suboperad of $\Pp_\infty^{pp}$ 
generated by $H_\del(\Inf(s\ov{ \Pp^{\antshrk}}))$. In particular, we have an inclusion $\Pp_\infty\hookrightarrow \Pp_\infty^{pp}$, which we claim is an $H_\del$-isomorphism. Indeed, \(\Pp_\infty^{pp}\) is weight-graded since it is a free operad. The differential $\del$ of \(\Pp_\infty^{pp}\) splits into the two components $\del=\del_1+\del_2$ of weight 0 and 1, respectively. So we get that \(\Pp_\infty^{pp}\) is a bicomplex with total differential $\del$. In a similar way, \(\Pp_\infty\) is a bicomplex with total differential \(\del=\del_2\). Moreover, the inclusion 
$\Pp_\infty\hookrightarrow \Pp_\infty^{pp}$ is a map of bicomplexes with respect to the weight-grading and it is a quasi-isomorphism with respect to the weight 0 component. Therefore, a standard spectral sequence argument shows that \(\Pp_\infty\hookrightarrow \Pp_\infty^{pp}\) is an $H_\del$-isomorphism.

The diagram
\[\begin{tikzcd}[ampersand replacement=\&]
	{\Pp_{\infty}^{pp}} \\
	{\Pp_\infty} \& \Pp
	\arrow["{\alpha}",from=1-1, to=2-2]
	\arrow["{\del}"',"{\sim}", from=2-1, to=1-1]
	\arrow["{\del}"',"{\sim}", from=2-1, to=2-2]
\end{tikzcd}\]
commutes and so $\alpha$ is an $H_\del$-isomorphism.
An analogous argument embedding $\Pp^{\antshrk}$ in bidegree $(0,*)$ gives that $\alpha$ is an $H_\delb$-isomorphism.
\end{proof}

\begin{exam}
    Let \(\Pp=As\) (as a non-symmetric operad), then \(\ppinfa=\TT(\Inf(sAs^{\antshrk}))\) and the differentials give
    \[\del\mu_n^{p,q}=q\mu_n^{p+1,q}+\sum_{\substack{k+\ell=n+1\\1\leq i\leq k}}\sum_{\substack{p_1+p_2=p+1\\q_1+q_2=q}}(-1)^{i(\ell+1)+n+\varepsilon}\mu_k^{p_1,q_1}\circ_i\mu_\ell^{p_2,q_2}\]
    and
    \[\delb\mu_n^{p,q}=-p\mu_n^{p,q+1}+\sum_{\substack{k+\ell=n+1\\1\leq i\leq k}}\sum_{\substack{p_1+p_2=p\\q_1+q_2=q+1}}(-1)^{i(\ell+1)+n+\varepsilon}\mu_k^{p_1,q_1}\circ_i \mu_\ell^{p_2,q_2},\]
    where \(\varepsilon=p+q+n+(p_1+q_1)(p_2+q_2+\ell)\) and we have set \(\mu_n^{p,q} := (p,q)_{2-n} \otimes s\mu_n\) for \(\mu_n \in \mathrm{As}^{\antshrk}(n)\).
    Algebras over this operad are precisely the pluripotential \(A_\infty\)-algebras introduced in Section~\ref{subsec:AinftyAlgs}.
\end{exam}
  
\begin{exam}
If \(\Pp=Lie\), we obtain \(L_\infty^{pp}=\TT(\Inf(sLie^{\antshrk}))\) and the differentials give

\[
\del\ell_n^{p,q}=q\ell_n^{p+1,q}+\sum_{r+s=n+1}\sum_{\sigma \in \mathrm{Sh}_{r,\,s}^{-1}}\sum_{\substack{p_1+p_2=p+1\\q_1+q_2=q}}(-1)^{r(s+1)+\varepsilon}\operatorname{sgn}(\sigma)(\ell_r^{p_1,q_1}\circ_1\ell_s^{p_2,q_2})\sigma
\]
and
\[
\delb\ell_n^{p,q}=-p\ell_n^{p,q+1}+\sum_{r+s=n+1}\sum_{\sigma \in \mathrm{Sh}_{r,\,s}^{-1}}\sum_{\substack{p_1+p_2=p\\q_1+q_2=q+1}}(-1)^{r(s+1)+\varepsilon}\operatorname{sgn}(\sigma)(\ell_r^{p_1,q_1}\circ_1 \ell_s^{p_2,q_2})\sigma,
\]
where \(\varepsilon=p+q+n+(p_1+q_1)(p_2+q_2+s)\). We obtain thus, that a pluripotential \(L_{\infty}\)-algebra structure on a bicomplex \((L,\del,\delb)\) is a family of operations
\[
\ell_{n}^{p,q} : L^{\otimes n} \longrightarrow L, \qquad n = 2,3,\ldots,
\]
of bidegrees \(1-n\leq p+q\leq 2-n\), such that each \(\ell_{n}^{p,q}\) is anti-symmetric and such that
\[
[\del,\ell_{n}^{p,q}]
=q\ell_n^{p+1,q}+
\sum_{\substack{r+s = n+1}}
\;\sum_{\sigma \in \mathrm{Sh}_{r,\,s}^{-1}}\sum_{\substack{p_1+p_2=p+1\\q_1+q_2=q}}
\operatorname{sgn}(\sigma)\,
(-1)^{r(s+1)+\varepsilon}\,
\ell_{r}^{p_1,q_1} \circ 
\bigl( \ell_{s}^{p_2,q_2} \otimes \id^{\otimes (r-1)} \bigr)
\sigma
\]
and
\[
[\delb,\ell_{n}^{p,q}]
=-p\ell_n^{p,q+1}+
\sum_{\substack{r+s = n+1}}
\;\sum_{\sigma \in \mathrm{Sh}_{r,\,s}^{-1}}\sum_{\substack{p_1+p_2=p\\q_1+q_2=q+1}}
\operatorname{sgn}(\sigma)\,
(-1)^{r(s+1)+\varepsilon}\,
\ell_{r}^{p_1,q_1} \circ 
\bigl( \ell_{s}^{p_2,q_2} \otimes \id^{\otimes (r-1)} \bigr)
\sigma
\]
hold.
\end{exam}

\section{Homotopy theory of \texorpdfstring{\(\Pp\)}{P-infinity}-algebras}\label{sec:HomThPalg}
In this section we define the bar and cobar costructions for pluripotential \(\Pp\)-algebras and show that when \(\Pp\) is Koszul, the homotopy category of \(\Pp\)-algebras is equivalent to the homotopy category of \(\Pp_\infty\)-algebras with their \(\infty\)-morphisms. 

\subsection{Bar and Cobar constructions for operadic algebras}\label{subsec:BarCobar}
Let $\Pp$ and $\Cc$ be an operad and a cooperad in bicomplexes.
For the sake of simplicity, we will 
restrict to \emph{connected cooperads}, that is \(\Cc(0)=0\) and \(\Cc(1)=\kk\),
which is sufficient for our interests. This condition may be removed assuming conilpotency on coalgebras (see Section 5.8.4 in \cite{LoVa12}).

Using the notion of twisting morphism
$\varphi=(\varphi^{10},\varphi^{01}):\Cc\to\Pp$ 
introduced in Section~\ref{subsec:TwistingMorph}, we will
associate a 
pair of adjoint functors
\[\Omega_\varphi:\Cc\text{-coalgebras}\rightleftarrows{\Pp\text{-algebras}:\br_\varphi}\]
relating algebras and coalgebras in bicomplexes.
The construction is analogous to the classical dg setting with the obvious difference introduced by the bicomplex structures. We will follow the construction as explained by Getzler and Jones in \cite{GeJo94}.

Let $(A,\del_A,\delb_A,\gamma)$ be a $\Pp$-algebra. Consider the cofree $\Cc$-coalgebra $\Cc(A)$ and the induced square zero coderivations
\[\del_1:=\del_\Cc\circ\id_A +\id_\Cc\circ'\del_A\quad\text{and}\quad\delb_1:=\delb_\Cc\circ\id_A +\id_\Cc\circ'\delb_A.\]

For any linear map $\alpha:\Cc\to \Pp$, consider the unique coderivation $d^c_\alpha$ extending 
\[\Cc(A)\xrightarrow{\alpha\circ\id}\Pp(A)\xrightarrow{\gamma}A\]
as explained in Section~\ref{subsec:CoopsCoalgs}.
It is explicitly given by
\[d^c_\alpha := (\id_\Cc\circ (\id_A;\gamma))([(\id_\Cc \circ_{(1)} \alpha)\Delta_{(1)}]\circ \id_A).\]

Note that, for any two linear maps $\alpha,\beta:\Cc\to\Pp$, the following relation is satisfied
\[d^c_\alpha d^c_\beta=d^c_{\alpha\star\beta},\]
see for instance \cite{GeJo94}.

Recall from Definition~\ref{defi:twisting}, that an operadic twisting morphism $\varphi$ is given by two linear maps $\varsh$ and $\varsv$ which we will use to construct a functor
\[\br_\varphi:\Palg\to\Cc\text{-Coalg}\]
called the \emph{bar construction with respect to $\varphi$}. It is given on objects by
\[\br_{\varphi}(A):= (\Cc(A),\del_1+d^c_{\varsh},\delb_1+d^c_{\varsv}).\]

By definition, $B_\varphi(A)$ is a $\Cc$-coalgebra with two additional coderivations.
The following proposition shows that these are indeed anti-commuting codifferentials.
\begin{prop}\label{barbicomplex}
   The bar construction $B_\varphi(A)$ satisfies the bicomplex equations:
    \[
    \begin{cases}
        (\del_1+d_{\varphi^{10}})^2=0,\\
        (\delb_1+d_{\varphi^{01}})^2=0,\text{ and}\\
        (\del_1+d_{\varphi^{10}})(\delb_1+d_{\varphi^{01}})+(\delb_1+d_{\varphi^{01}})(\del_1+d_{\varphi^{10}})=0.
    \end{cases}
    \]
\end{prop}
\begin{proof}
The proof is parallel to that of Proposition~\ref{prop:TwistedComposite}.
    First, note that, as in the classical case, for any linear map $\alpha:\Cc\to\Pp$
    we have \[\del_1d^c_{\alpha}+d^c_{\alpha}\del_1=d^c_{\del\alpha+\alpha\del}\text{ and }\delb_1d^c_{\alpha}+d^c_{\alpha}\delb_1=d^c_{\delb\alpha+\alpha\delb}.\]
    Since $\varphi$ is a twisting morphism, the following relations are satisfied
    \[[\del,\varphi^{10}]+\varphi^{10}\star\varphi^{10}=0,\quad [\delb,\varphi^{01}]+\varphi^{01}\star\varphi^{01}=0\quad\text{and}\quad [\del,\varphi^{01}] + \varphi^{01}\star\varphi^{10}+
    [\delb,\varphi^{10}]+ \varphi^{10}\star\varphi^{01}=0.\]
    Together with the fact that \(d^c_{\alpha} d^c_{\beta}=d^c_{\alpha\star\beta}\) imply \[(\del_1+d^c_{\varphi^{10}})^2=0\text,\quad(\delb_1+d^c_{\varphi^{01}})^2=0,\] and
    \begin{align*}
     (\del_1+d^c_{\varphi^{10}})(\delb_1+d^c_{\varphi^{01}})+(\delb_1+d^c_{\varphi^{01}})(\del_1+d^c_{\varphi^{10}})=&\\
     =d^c_{[\del,\varphi^{01}]+ \varphi^{01}\star\varphi^{10}}+d^c_{[\delb,\varphi^{10}]+ \varphi^{10}\star\varphi^{01}}=&\,0.
    \end{align*}
\end{proof}

Let $(C,\Delta,\del,\delb)$ be a $\Cc$-coalgebra. Consider the free $\Pp$-algebra $\Pp(C)$ and the induced square zero derivations
\[\del_1:=\del_\Pp\circ\id_C +\id_\Pp\circ'\del_C\quad\text{and}\quad\delb_1:=\delb_\Pp\circ\id_C +\id_\Pp\circ'\delb_C.\]

Let $\alpha:\Cc\to \Pp$ be a linear map and consider the unique derivation $d^a_\phi$ extending the following linear map
\[C\xrightarrow{\Delta}\Cc(C)\xrightarrow{\alpha\circ\id}\Pp(C),\]
it is explicitly given by
\[d^a_\alpha := (\gamma_{(1)}\circ\id_C)(\id_\Pp\circ'[(\alpha\circ\id)\Delta]).\]

Again, for any two maps $\alpha,\beta:\Cc\to\Pp$, the following relation is satisfied
    \[d^a_\alpha d^a_\beta=d^a_{\alpha\star\beta}.\]

Let $\varphi:\Cc\to \Pp$ be an operadic twisting morphism. We define the \emph{cobar construction of $C$ with respect to $\varphi$} as
\[\Omega_{\varphi}(C):= (\Pp(C),-\del_1-d^a_{\varphi^{10}},-\delb_1 -d^a_{\varphi^{01}}).\]
The following  is proven analogously to 
Proposition~\ref{barbicomplex}.
\begin{prop}  The cobar construction $\cbr_\varphi(C)$ satisfies the bicomplex equations.
\end{prop}

We will now show that the two functors constructed above form an adjunction. Let $\varphi:\Cc\to \Pp$ be an operadic twisting morphism, $(C,\Delta,\del,\delb)$ a $\Cc$-coalgebra and $(A,\gamma,\del,\delb)$ a $\Pp$-algebra. We define two unary operations $\star_\varphi$ and $\ov{\star}_\varphi$ in $\underline{\Hom}(C,A)$ of degree $(1,0)$ and $(0,1)$, respectively:
\[\star_{\varphi}(\alpha): C\xrightarrow{\Delta}\Cc(C)\xrightarrow{\varsh\circ \alpha}\Pp(A)\xrightarrow{\gamma}A\]
and
\[\ov\star_{\varphi}(\alpha): C\xrightarrow{\Delta}\Cc(C)\xrightarrow{\varsv\circ \alpha}\Pp(A)\xrightarrow{\gamma}A.\]

A map $\alpha: C\to A$ of degree $(0,0)$ is a \emph{twisting morphism with respect to }$\varphi$ if it satisfies the following two relations
\[[\del,\alpha]+\star_\varphi(\alpha)=0\quad\text{and}\quad[\delb,\alpha]+\ov\star_\varphi(\alpha)=0.\]

\begin{prop}\label{prop:BarCobarAdj}
    Let $\varphi:\Cc\to \Pp$ be a twisting morphism. For every $\Pp$-algebra and every $\Cc$-coalgebra there exist natural bijections 
    \[\Hom_{\text{$\Pp$-Alg}}(\Omega_\varphi C, A)\cong \Tw_\varphi(C,A)\cong \Hom_{\text{$\Cc$-Coalg}}(C,\br_\varphi A).\]
\end{prop}
\begin{proof}
    Since $\Omega_\varphi$ is a free $\Pp$-algebra, a map $\Tilde{\alpha}:\Omega_\varphi(C)\to A$ is uniquely determined by its restriction to the generators $\alpha:C\to A$. The fact that $\Tilde{\alpha}$ commutes with the differentials is equivalent to the commutativity of the following two diagrams
\[\begin{tikzcd}[ampersand replacement=\&]
	C \&\& A \\
	{C+\Pp(C)} \&\& A
	\arrow["{\alpha}", from=1-1, to=1-3]
	\arrow["{\del_C+(\varsh\circ\id)\Delta}"', from=1-1, to=2-1]
	\arrow["{\del_A}", from=1-3, to=2-3]
	\arrow["{\alpha+\gamma(\id_\Pp\circ\alpha)}"', from=2-1, to=2-3]
\end{tikzcd}
\quad\quad
\begin{tikzcd}[ampersand replacement=\&]
	C \&\& A \\
	{C+\Pp(C)} \&\& A\;,
	\arrow["{\alpha}", from=1-1, to=1-3]
	\arrow["{\delb_C+(\varsv\circ\id)\Delta}"', from=1-1, to=2-1]
	\arrow["{\delb_A}", from=1-3, to=2-3]
	\arrow["{\alpha+\gamma(\id_\Pp\circ\alpha)}"', from=2-1, to=2-3]
\end{tikzcd}
\]
which is equivalent to the two equations \[[\del,\alpha]+\star_\varphi(\alpha)=0\quad\text{and}\quad[\delb,\alpha]+\ov\star_\varphi(\alpha)=0.\]

On the other hand, since $\br_\varphi(A)$ is cofree, a map $\Tilde{\alpha}:C\to\br_\varphi(A)$ is characterized by its restriction to the cogenerators, $\alpha:C\to A$. Then, $\Tilde{\alpha}$ commutes with the differentials if and only if the following two diagrams commute
\[\begin{tikzcd}[ampersand replacement=\&]
	C \&\& {\br_\varphi(A)} \\
	C \&\& A
	\arrow["{(\id_\Cc\circ\alpha)\Delta}", from=1-1, to=1-3]
	\arrow["{\del_C}"', from=1-1, to=2-1]
	\arrow["{\del_A+\gamma(\varsh\circ\id)}", from=1-3, to=2-3]
	\arrow["\alpha", from=2-1, to=2-3]
\end{tikzcd}
\quad
\begin{tikzcd}[ampersand replacement=\&]
	C \&\& {\br_\varphi(A)} \\
	C \&\& A\;,
	\arrow["{(\id_\Cc\circ\alpha)\Delta}", from=1-1, to=1-3]
	\arrow["{\delb_C}"', from=1-1, to=2-1]
	\arrow["{\delb_A+\gamma(\varsv\circ\id)}", from=1-3, to=2-3]
	\arrow["\alpha", from=2-1, to=2-3]
\end{tikzcd}
\]
which gives \[[\del,\alpha]+\star_\varphi(\alpha)=0\quad\text{and}\quad[\delb,\alpha]+\ov\star_\varphi(\alpha)=0.\qedhere\]
\end{proof}

\subsection{Codifferential viewpoint and pluripotential \texorpdfstring{$\infty$}{infinity}-morphisms}\label{subsec:PinftyAlgsCoders}
In analogy with the classical case, we can describe
\(\Pp_\infty^{pp}\)-algebras in terms of codifferentials on a cofree coalgebra, where \(\Pp_\infty^{pp}\) denotes the pluripotential minimal model of a Koszul operad \(\Pp\), see Section~\ref{subsec:KoszulMinimalModels}. In this case, we will obtain two anti-commuting codifferentials.
This is done analogously to the classical case, by defining a convenient twisting morphism \[\iota:\Pp^{\antshrk}_{pp}\longrightarrow \Pp_\infty^{pp}\] and then resourcing to the bar construction. Classically, $\iota$ is just defined by suspension. 
In our case, we will make use of the maps
\[
\mathfrak{s}_n:\ele_{n}\to\ele_{n+1}\quad\text{and}\quad\ov{\mathfrak{s}}_n:\ele_{n}\to\ele_{n+1}
\]
of bidegree \((1,0)\) and \((0,1)\), respectively, defined in Section~\ref{sec:staircase}. Explicitly, we  consider the equivariant maps
\[
\iota^{10}\colon\Pp^{\antshrk}_{pp}\xrightarrow{\mathfrak{s}\otimes s} \Pp_\infty^{pp}
\quad\text{
and}\quad
\iota^{01}\colon\Pp^{\antshrk}_{pp}\xrightarrow{\ov{\mathfrak{s}}\otimes s} \Pp_\infty^{pp}
\]
given by
\[
(i,j)_{n}\otimes \mu\mapsto (-1)^{i+j+n}(i+1,j)_{n+1}\otimes s\mu
\]
and
\[
(i,j)_{n}\otimes \mu\mapsto (-1)^{i+j+n}(i,j+1)_{n+1}\otimes s\mu,
\]
respectively.
\begin{prop}
    The pair \(\iota:=(\iota^{10},\iota^{01})\) defines a twisting morphism. 
\end{prop}
\begin{proof} 
This follows from the relations satisfied by \(\mathfrak{s}\), \(\ov{\mathfrak{s}}\), and the differentials, as described in Remark~\ref{rema:fraksdel}, together with the relations involving \(\mathfrak{s}\), \(\ov{\mathfrak{s}}\), and the maps \(\iota_{k,\ell}\colon\ele_{k+\ell}\to\ele_k\otimes\ele_\ell\) given in Proposition~\ref{prop:fraksiota}.
\end{proof}

The bar construction given in Section~\ref{subsec:BarCobar} associated to the twisting morphism \(\iota\) gives a functor
\[
\br_\iota:\pinfal\to\Pp^{\antshrk}_{pp}\mep\mathrm{Coalg}.
\]
As a consequence, any \(\Pp_\infty^{pp}\)-algebra  can be understood as two codifferentials on the \(\Pp^{\antshrk}_{pp}\)-cofree coalgebra  which anti-commute. 
Moreover, this viewpoint naturally gives a definition of $\infty$-morphism in the pluripotential case:

\begin{defi}
    Let \(A\) and \(B\) be two \(\Pp_\infty^{pp}\)-algebras. A \emph{pluripotential \(\infty\)-morphism}, \(\ppinfm\)-morphism for short, \(f:A\rightsquigarrow B\) is a coalgebra morphism between the bar constructions
    \[
    f:\br_\iota(A)\to\br_\iota(B).
    \]
\end{defi}
\begin{rema}
    Any morphism of cofree coalgebras is characterized by its projection to the cogenerators. So a \(\ppinfm\)-morphism of \(\Pp_{\infty}^{pp}\)-algebras \(f:A\rightsquigarrow B\) is equivalently given by a map
    \[
    \bar{f}:\Pp^{\antshrk}_{pp}(A)\to B
    \]
    satisfying
    \[(\delta_1^{B}+\delta_{\iota^{10}}^{B}) f=\ov f(\del_1^{A}+d_{\iota^{10}}^A)\quad\text{and}\quad(\ov\delta_1^{B}+\delta_{\iota^{01}}^{B})f=\ov f(\delb_1^{A}+d_{\iota^{01}}^A),\]
    where \(\del_1+d_{\iota^{10}}\) and \(\delb_1+d_{\iota^{01}}\) are the codifferentials of the bar construction defined in Section~\ref{subsec:BarCobar} and \(\delta_1^{B}+\delta_{\iota^{10}}^{B}\) and \(\ov\delta_1^{B}+\delta_{\iota^{01}}^{B}\) the projection of such codifferentials to the cogenerators.
    \end{rema}
    
    We call a \(\ppinfm\)-morphism a \emph{\(\ppinfm\)-weak equivalence} if the arity 1 component, that is \(f_1:A\to B\), is a pluripotential weak equivalence.
    
\begin{exam}
    Taking \(\Pp=As\) (as a non symmetric operad) we obtain that a \(\ppinfm\)-morphism \(f:A\rightsquigarrow B\) of \(\ppinfa\)-algebras is a family of maps 
    \[
    f_n:\ele_n\otimes A^{\otimes n} \to B,\quad\text{for }n\geq 1 .
    \]
    The map \(f\) has to satisfy
    \[
    (\del_1^{B}+d_{\iota^{10}}^{B})f=f(\del_1^{A}+d_{\iota^{10}}^A)\quad\text{and}\quad(\delb_1^{B}+d_{\iota^{01}}^{B})f=f(\delb_1^{A}+d_{\iota^{01}}^A).
    \]
    Unraveling the equation we obtain the equations of a \(\ppinfm\)-morphism in Definition~\ref{defi:inftymorph} of Section~\ref{subsec:AinftyAlgs}.
\end{exam}

Let \(\pinfalm\) denote the category whose objects are \(\Pp_{\infty}^{pp}\)-algebras and whose morphisms are given by \(\ppinfm\)-morphisms. The following is automatic.

\begin{prop}\label{prop:fullyfatihfulCoder}
    The functor \(\br_\iota\) extends to a fully faithful embedding \(\br_\iota\)
    \[
    \br_\iota:\pinfalm\to\{\text{quasi-free }\Pp_{pp}^{\antshrk}\text{-}\mathrm{Coalg}\}.
    \]
\end{prop}

\subsection{Equivalence of homotopy categories}
Throughout this section, $\Pp$ is a quadratic operad in 
\(\mathrm{Vect}_\kk\) that is Koszul as dg-operad.
 Let $\Pp_\infty^{pp}$ denote its pluripotential minimal model given in Section~\ref{subsec:KoszulMinimalModels}.
Recall that $\pinfalm$ denotes the category of $\Pp_\infty^{pp}$-algebras  with $\ppinfm$-morphisms, and $\Palg$ the category of $\Pp$-algebras (in bicomplexes) with strict morphisms.
The main result of this section is the equivalence of homotopy categories
    \[\mathrm{Ho}(\Palg)\cong\mathrm{Ho}(\pinfalm).\]
Consider the twisting morphism
 \[\kappa:\Pp^{\antshrk}_{pp}\to\Pp\]
defined in Section~\ref{subsec:KoszulMinimalModels} for (co)operads in bicomplexes.
By Proposition~\ref{prop:BarCobarAdj}, this gives an adjunction
\[\begin{tikzcd}[ampersand replacement=\&]
	{\Omega_\kappa:\Pp^{\antshrk}_{pp}\text{-Coalg}} \& {\Pp\text{-Alg}:\br_\kappa}
	\arrow[harpoon, from=1-1, to=1-2]
	\arrow[shift left, harpoon, from=1-2, to=1-1]
\end{tikzcd}\]   
between coalgebras and algebras in bicomplexes.
The counit of this adjunction is given as follows: for $A$ a $\Pp$-algebra in bicomplexes, it is  
the map of algebras extending 
\[\Pp^{\antshrk}_{pp}\circ A\xrightarrow{\varepsilon\circ\id_A}\mathrm{I}\circ A\cong A\]
where $\varepsilon$ denotes the counit of the cooperad $\Pp^{\antshrk}_{pp}$. That is, 
\[
\varepsilon_\kappa: \cbr_\kappa\br_\kappa A =\Pp \circ \Pp^{\antshrk}_{pp} \circ A\xrightarrow{\id_\Pp\circ\varepsilon\circ\id_A}\Pp(A)\xrightarrow{\gamma_A}A.
\]
The unit is given by the map of coalgebras extending the map
\[C\cong\mathrm{I} \circ C\xrightarrow{\eta\circ \id_C}\Pp\circ C,\]
where $\eta$ denotes the unit of the operad $\Pp$. That is,
\[
C \xrightarrow{\Delta_C} \Pp^{\antshrk}_{pp} \circ C 
\xrightarrow{\id_{\Pp^{\text{¡}}_{pp}} \circ \eta \circ \id_C} 
\Pp^{\antshrk}_{pp} \circ \Pp \circ C 
= \br_\kappa \cbr_\kappa C.
\]
We have
\begin{prop}\label{prop:CbrkappaBrkappaWeakeq}
Let \(\Pp\) be a Koszul operad. For every $\Pp$-algebra $A$, the
counit 
\[\varepsilon_{\kappa}:\Omega_\kappa\br_\kappa A\xrightarrow{\sim}A\]
is a pluripotential weak equivalence of $\Pp$-algebras.
\end{prop}
\begin{proof}
By definition, we have
    \[\Omega_\kappa\br_\kappa A:=(\Pp\circ\Pp^{\antshrk}_{pp}\circ A, \del_1 + d^{c}_{\kappa^{10}}+d^{a}_{\kappa^{10}},\delb_1+d^{c}_{\kappa^{01}}+d^{a}_{\kappa^{01}}),\] 
    where $\del_1$ and $\delb_1$ are the induced differentials of $A$ on $\Pp\circ\Pp^{\antshrk}_{pp}\circ A$, $d^{c}_{\kappa^{10}}$ is the induced differential by that of the bar construction
    \[
    \Pp^{\antshrk}_{pp}(A)\to \Pp^{\antshrk}_{pp}\circ_{(1)}\Pp^{\antshrk}_{pp}(A)\xrightarrow{(\id_\Cc \circ_{(1)} \kappa^{10})\circ\id_A} (\Pp^{\antshrk}_{pp}\circ_{(1)}\Pp)\circ A\xrightarrow{(\id;\gamma_A)}\Pp^{\antshrk}_{pp}(A),
    \]
    and $d^{a}_{\kappa^{10}}$ is the differential extending the following map
    \[
    \Pp^{\antshrk}_{pp}(A)\to\Pp^{\antshrk}_{pp}\circ\Pp^{\antshrk}_{pp}(A)\xrightarrow{\kappa^{10}\circ\id\circ\id} \Pp\circ\Pp^{\antshrk}_{pp}(A).
    \]
    The differentials $d^{c}_{\kappa^{01}}$ and $d^{a}_{\kappa^{01}}$ are defined in the same manner but using $\kappa^{01}$. Since $\Pp$ and $\Pp^{\antshrk}_{pp}$ are weight graded, $\cbr_\kappa\br_\kappa A$ is weight graded, and the weight is defined by the sum of the weights of $\Pp$ and $\Pp^{\antshrk}_{pp}$.    
    We define a bounded below and exhaustive increasing filtration
    \[
    F_r(\Omega_\kappa\br_{\kappa}A):=\{\omega\mid\text{weight of }\omega\leq r\}.
    \]
    The differentials of \(\cbr_\kappa\br_\kappa A\) preserve the filtration. Indeed, \(\del_1\) and \(\delb_1\) are weight preserving as well as \(d_{\kappa^{10}}^a\) and \(d_{\kappa^{01}}^a\).  The differentials \(d_{\kappa^{10}}^c\) and \(d_{\kappa^{01}}^c\) are not weight preserving but decrease the weight, in particular we have 
    \[d^c_{\kappa^{**}}(F_r(\cbr_\kappa\br_\kappa A))\subseteq F_{r-1}(\cbr_\kappa\br_{\kappa} A).\] 
    These observations and the fact that \(d^a_{\kappa^{10}}\) and \(d^a_{\kappa^{01}}\) are constructed as the differentials of the twisted composite products (Section~\ref{subsec:TwistingMorph}) imply  
    \[\bigoplus_{r\geq 0}F_r(\cbr_\kappa\br_\kappa A)/F_{r-1}(\cbr_\kappa\br_\kappa A)=(\Pp\circ_\kappa\Pp^{\antshrk}_{pp})(A).\]
    Observe that, since \(\Pp\) is Koszul, by the ABC-Künneth formulae (Proposition~\ref{prop:Kunneth}), and using the fact that \(\kk[\SS_n]\) is a semisimple ring, we obtain
    \[
    \ABC\left((\Pp\circ_\kappa\Pp^{\antshrk}_{pp})(A)\right)\cong\mathrm{I}(\ABC(A))\cong \ABC(A).
    \]
    This implies
    \[
    \BC(F_r(\cbr_\kappa\br_\kappa A)/F_{r-1}(\cbr_\kappa\br_\kappa A))=\A(F_r(\cbr_\kappa\br_\kappa A)/F_{r-1}(\cbr_\kappa\br_\kappa A))=0
    \] for $r>0$.
    
    Endow $A$ with the trivial filtration of weight 0. 
    Then, the counit of the adjunction
    \[\varepsilon_{\kappa}:\cbr_\kappa\br_\kappa A\xrightarrow{\gamma_A(\id_\Pp\circ\varepsilon\circ\id_A)}A\]
     is a filtered map. For $r=0$, we have
    \[F_0(\cbr_\kappa\br_\kappa A)=\mathrm{I}\circ\mathrm{I}\circ A\cong A\xrightarrow{\cong}A.\]
   Assume inductively that, for all $0\leq s<r$, we have a pluripotential weak equivalence
    \[
    F_s:=F_s(\Omega_\kappa\br_{\kappa}A)\xrightarrow{\sim}A.
    \]
The map of short exact sequences
\[\begin{tikzcd}[ampersand replacement=\&]
	0 \& {F_{r-1}} \& {F_r} \& {F_r/F_{r-1}} \& 0 \\
	0 \& A \& A \& 0 \& 0
	\arrow[from=1-1, to=1-2]
	\arrow[from=1-2, to=1-3]
	\arrow[from=1-2, to=2-2]
	\arrow[from=1-3, to=1-4]
	\arrow[from=1-3, to=2-3]
	\arrow[from=1-4, to=1-5]
	\arrow[from=1-4, to=2-4]
	\arrow[from=2-1, to=2-2]
	\arrow[from=2-2, to=2-3]
	\arrow[from=2-3, to=2-4]
	\arrow[from=2-4, to=2-5]
\end{tikzcd}\]
    induces a commutative diagram where the rows are exact (see Lemma~\ref{lemm:long-exact-seq})
\[\begin{adjustbox}{max width=\textwidth,center}\begin{tikzcd}[ampersand replacement=\&]
	\dots \& {\A^{p-1q-1}(F_{r-1})} \& {\A^{p-1q-1}(F_r)} \& {\A^{p-1q-1}(F_r/F_{r-1})} \& {\BC^{pq}(F_{r-1})} \& {\BC^{pq}(F_{r})} \& {\BC^{pq}(F_r/F_{r-1})} \& \dots \\
	\dots \& {\A^{p-1q-1}(A)} \& {\A^{p-1q-1}(A)} \& 0 \& {\BC^{pq}(A)} \& {\BC^{pq}(A)} \& 0 \& \dots
	\arrow[from=1-1, to=1-2]
	\arrow[from=1-2, to=1-3]
	\arrow["\cong", from=1-2, to=2-2]
	\arrow[from=1-3, to=1-4]
	\arrow[from=1-3, to=2-3]
	\arrow[from=1-4, to=1-5]
	\arrow[from=1-4, to=2-4]
	\arrow[from=1-5, to=1-6]
	\arrow["\cong", from=1-5, to=2-5]
	\arrow[from=1-6, to=1-7]
	\arrow[from=1-6, to=2-6]
	\arrow[from=1-7, to=1-8]
	\arrow[from=1-7, to=2-7]
	\arrow[from=2-1, to=2-2]
	\arrow["\cong", from=2-2, to=2-3]
	\arrow[from=2-3, to=2-4]
	\arrow[from=2-4, to=2-5]
	\arrow["\cong", from=2-5, to=2-6]
	\arrow[from=2-6, to=2-7]
	\arrow[from=2-7, to=2-8]
\end{tikzcd}\end{adjustbox}\]
By the observation above, we obtain \(F_r(\cbr_\kappa\br_\kappa A)\xrightarrow{\sim}A.\)
Since the filtration is exhaustive, this implies \(
\cbr_\kappa\br_\kappa A\xrightarrow{\sim} A.\)
\end{proof}

Consider the morphism \[\varphi:\Pp_\infty^{pp}\xrightarrow{\sim}\Pp,\]
it induces the inclusion functor
\[\varphi^{*}:\Palg\hookrightarrow \pinfal.\]
Denote by 
\[i:\Palg\hookrightarrow \pinfal \hookrightarrow \pinfalm\]
the obvious inclusion of categories.
\begin{rema}
    Recall the twisting morphism
    \[
    \iota:\Pp^{\antshrk}_{pp} \longrightarrow \Pp^{pp}_\infty
    \]
    introduced in Section~\ref{subsec:BarCobar}. The twisting morphisms \(\kappa\) and \(\iota\) satisfy \(\kappa = \varphi \, \iota.\)
    From this identity, together with the definition of the bar construction (Section~\ref{subsec:BarCobar}), it follows that
    \[
    \br_\kappa = \br_{\iota} \, i.
    \]

\end{rema}

The following proposition is proved in exactly the same way as in the dg case, we include it here for completeness.

\begin{prop}\label{prop:AdjuncioRectif}
    Let $\Pp$ be a Koszul operad. There is an adjunction
\[\begin{tikzcd}[ampersand replacement=\&]
	{\Omega_\kappa\br_{\iota}:\pinfalm} \& {\Palg:i}
	\arrow[harpoon, from=1-1, to=1-2]
	\arrow[shift left, harpoon, from=1-2, to=1-1]
\end{tikzcd}\]   
\end{prop}
\begin{proof}
    Let $A$ be a $\Pp_\infty^{pp}$-algebra, consider its bar construction $\br_\iota(A)$. We have the following morphism given by the unit of the adjunction associated to $\kappa$
    \[
    \br_\iota A\to\br_\kappa\cbr_\kappa(B_\iota A).
    \]
    By the previous remark, we obtain a morphism
        \[
        \br_\iota A\to\br_\iota(i\cbr_\kappa(B_\iota A)),
        \]
    which gives a $\ppinfm$-morphism by Proposition~\ref{prop:fullyfatihfulCoder}
    \[
    \eta(A): A\rightsquigarrow i\Omega_\kappa(\br_\iota A),
    \]
    which induces a natural transformation $\eta:\id\to i\Omega_\kappa\br_\iota$.

    Let $A$ be a $\Pp$-algebra, we have the following morphism of $\Pp$-algebras
    \[\varepsilon_\kappa(A):\Omega_\kappa \br_\iota i(A)=\Omega_\kappa\br_\kappa(A)\to A,\]
    which induces a natural transformation of functors $\varepsilon\colon \Omega_\kappa\br_\iota i\to \id$.

    Let $A$ be a $\Pp_\infty^{pp}$-algebra, the composite
    $\varepsilon(\cbr_\kappa\br_\iota)\circ\cbr_\kappa\br_\iota\eta$
    is equal to 
    \[\cbr_\kappa\br_\iota(A)\xrightarrow{\Omega_\kappa(\eta_\kappa (B_\iota A))}\cbr_\kappa\br_\kappa\cbr_\kappa\br_\iota(A)\xrightarrow{\varepsilon_\kappa(\cbr_\kappa(\br_\iota A))}\cbr_\kappa\br_\iota(A),\]
    which is the identity by the adjunction associated to $\kappa$.
    On the other hand, the composite $i(\varepsilon(A))\circ\eta(i(A))$ is a $\ppinfm$-morphism, whose image under $\br_\iota$ is given by
    \[\br_\kappa A\xrightarrow{\eta_\kappa(\br_\kappa)}\br_\kappa\cbr_\kappa\br_\kappa A\xrightarrow{\br_\kappa(\varepsilon_\kappa)}\br_\kappa A,\]
which is the identity again by the adjunction associated to $\kappa$.
\end{proof}

\begin{theo}[Rectification]\label{theo:Rectification}
    Let $\Pp$ be a Koszul operad. There is a natural weak equivalence 
    \[\id_{\pinfalm}\to \Omega_\kappa\br_\iota.\]
\end{theo}
\begin{proof}
    To any $\Pp_\infty^{pp}$-algebra $A$, we apply the unit of the adjunction from Proposition~\ref{prop:AdjuncioRectif}, namely
    \[
    A \rightsquigarrow i\cbr_\kappa\br_\iota A = \cbr_\kappa\br_\iota A.
    \]
    Recall that this $\ppinfm$-morphism is defined by applying the unit of the adjunction associated to $\kappa$ to $\br_\iota A$. To show that this map is a $\ppinfm$-weak equivalence, we need to prove that its first component is a pluripotential weak equivalence.

    The first component is given by the composite:
    \[
    A \xrightarrow{\cong} \mathrm{I} \circ \mathrm{I} \circ A \hookrightarrow \Pp \circ \Pp^{\antshrk}_{pp} \circ A.
    \]
    We equip $\cbr_\kappa\br_\iota A$ with a filtration by the total weight in $\Pp \circ \Pp^{\antshrk}_{pp}$. As in Proposition~\ref{prop:CbrkappaBrkappaWeakeq}, the associated graded object is
    \[
    \bigoplus_{r\geq 0}F^r(\cbr_\kappa\br_\iota A)/F^{r-1}(\cbr_\kappa\br_\iota A)=\Pp\circ_\kappa\Pp^{\antshrk}_{pp}\circ A.
    \]
    By an argument analogous to that in Proposition~\ref{prop:CbrkappaBrkappaWeakeq}, we conclude that the morphism \( A \rightsquigarrow \cbr_\kappa\br_\iota A \) is indeed a $\ppinfm$-weak equivalence.
\end{proof}

\begin{prop}\label{prop:CbrkappaBriotaWeakeq}
    The rectification functor
    \[\cbr_\kappa\br_\iota:\pinfalm\to\Palg\]
    sends \(\ppinfm\)-weak equivalences to pluripotential weak equivalences.
\end{prop}
\begin{proof}
By Theorem~\ref{theo:Rectification}, we have the following commutative diagram of $\ppinfm$-morphisms
\[\begin{tikzcd}[ampersand replacement=\&]
	{\Omega_\kappa\br_\iota(A)} \& {\cbr_\kappa\br_ \iota(A')} \\
	A \& {A'}.
	\arrow["\Omega_\kappa\br_\iota(f)",from=1-1, to=1-2]
	\arrow["\sim", squiggly, from=2-1, to=1-1]
	\arrow["f",squiggly, from=2-1, to=2-2]
	\arrow["\sim"', squiggly, from=2-2, to=1-2]
\end{tikzcd}\]
We can restrict the diagram to the first components of the $\ppinfm$-morphisms to obtain the following commutative diagram of maps of bicomplexes
\[\begin{tikzcd}[ampersand replacement=\&]
	{\Omega_\kappa\br_\iota(A)} \& {\cbr_\kappa\br_ \iota(A')} \\
	A \& {A'}
	\arrow["\Omega_\kappa\br_\iota(f)",from=1-1, to=1-2]
	\arrow["\sim", from=2-1, to=1-1]
	\arrow["{f_1}", from=2-1, to=2-2]
	\arrow["\sim"', from=2-2, to=1-2]
\end{tikzcd}\]
so  $f$ is a $\ppinfm$-weak equivalence if and only if $\Omega_\kappa\br_\iota(f)$ is as well.
\end{proof}

\begin{theo}[Equivalence of homotopy categories]\label{theo:EqHoCats}
    Let $\Pp$ be a Koszul operad. We have an equivalence of homotopy categories
    \[\mathrm{Ho}(\Palg)\cong\mathrm{Ho}(\pinfalm).\]
\end{theo}
\begin{proof}
    Consider the adjuction of Proposition~\ref{prop:AdjuncioRectif}
    \[\begin{tikzcd}[ampersand replacement=\&]
	{\Omega_\kappa\br_{\iota}:\pinfalm} \& {\Pp\text{-alg}:i.}
	\arrow[harpoon, from=1-1, to=1-2]
	\arrow[shift left, harpoon, from=1-2, to=1-1]
\end{tikzcd}\] 
    The functor $i$ preserves weak equivalences by definition. The functor $\cbr_\kappa\br_\iota$ also preserves weak equivalences by Proposition~\ref{prop:CbrkappaBriotaWeakeq}. This implies they induce an adjunction at the level of homotopy categories.
    Moreover, the counit of the adjunction is given by 
    \[
    \varepsilon_\kappa:\cbr_\kappa\br_\iota i(A)=\cbr_\kappa\br_\kappa(A)\to A,
    \]
    which is a pluripotential weak equivalence of $\Pp$-algebras
    by Propostition~\ref{prop:CbrkappaBrkappaWeakeq}. And the unit of the adjunction is a \(\ppinfm\)-weak equivalence by Theorem~\ref{theo:Rectification}. This proves that the unit and the counit induce an isomorphism at the level of homotopy categories.
\end{proof}

\appendix

\section{Proof of Lemma~\ref{lemm:HTT}} \label{app:induc}

We prove the key lemma underlying the Homotopy Transfer Theorem for bidifferential bigraded algebras.
Please do not print, save the planet.

Recall from Section~\ref{subsec:HTT} the \(\mathfrak{p}\)-kernels 
\[
\mup_n^{p,q}:A^{\otimes n}\to A
\]
defined for \(n\geq 2\). We shall prove the following result.
\begin{namedtheorem}[\ref{lemm:HTT}]
Let $(A,\del,\delb,\mu)$ be a bidifferential bigraded algebra. 
Given a contraction     
\[
    \begin{tikzcd}[ampersand replacement = \&]
    (A,\del,\delb) \arrow[r, rightarrow, shift left, "f"] \arrow[r, shift right, leftarrow, "g" swap] \arrow[loop left, "h_{11}"] \& (B,\del,\delb)
    \end{tikzcd}
\]
define
\begin{equation} \tag{1} \nu_n^{p,q}:=f\mup_n^{p,q} g^{\otimes n}
\end{equation}
and
\begin{equation}\tag{2}
    g_n^{p,q}:=\sum_{\substack{i-\alpha=p\\j-\beta=q}}(-1)^{(\alpha+\beta)(i+j+n)} h_{\alpha\beta} \mup_n^{i,j}g^{\otimes n}
\end{equation}
for all $n\geq 2$.

Then, $(B,\del,\delb,\nu_n^{p,q})$ is a $\ppinfa$-algebra, and $g=(g,g_n^{p,q})$
is a $\ppinfm$-weak equivalence $g:B\rightsquigarrow A$.
\end{namedtheorem}

To do so, we will employ an alternative definition of the $\mathfrak{p}$-kernels for a more compact presentation in the proof. We will use the bicomplexes \(\ele_n\) and the maps
\[
\iota_{k_1,\dots,k_r}:\ele_{k_1+\dots+k_r}\to \ele_{k_1}\otimes\dots\otimes\ele_{k_r}
\]
defined in Section~\ref{sec:staircase}.

Let $(A,\del,\delb,\mu)$ be a bidifferential bigraded algebra and let
    \[
    \begin{tikzcd}[ampersand replacement = \&]
    (A,\del,\delb) \arrow[r, rightarrow, shift left, "f"] \arrow[r, shift right, leftarrow, "g" swap] \arrow[loop left, "h"] \& (B,\del,\delb)
    \end{tikzcd}
    \]
be a contraction.
\begin{rema}\label{rema:H}
    Recall from Section~\ref{sec:staircase}, the maps \[\mathfrak{s}_1,\ov{\mathfrak{s}}_1\colon \ele\to \kk\]
    of bidegree \((1,0)\) and \((0,1)\), respectively.
    A pluripotential homotopy from $f\colon A\to B$ to $g\colon A\to B$ is equivalent to a bidegree preserving map $\bm h\colon \ele\otimes A\to B$ such that
    \[[\del,\bm h]=(f-g)(\mathfrak{s}_1\otimes \id)\quad\text{ and }\quad[\delb, \bm h]=(f-g)(\ov{\mathfrak{s}}_1\otimes \id).\]
\end{rema}

The $\mathfrak{p}$-kernels are equivalent to bidegree preserving linear maps
\[\mup_n\colon \ele_{2-n}\otimes A^{\otimes n}\to A,\]
for $n\geq 2$ given by
\begin{equation*}
\mup_n\colon =\sum_{k+\ell=n}(-1)^{k(\ell+1)}\mu\big(\bm h(\id\otimes\mup_k)\otimes \bm h(\id\otimes\mup_\ell)\big)\Delta^h_{k,\ell}.
\end{equation*}
Here, 
\[\Delta^h_{k,\ell}:=\tau(\iota_{1,2-k,1,2-\ell}\otimes\id^{\otimes n})\colon \ele_{2-n}\otimes A^{\otimes n}\to \ele\otimes\ele_{2-k}\otimes A^{\otimes k}\otimes\ele\otimes\ele_{2-\ell}\otimes A^{\otimes \ell},
\]
for $k,\ell\geq1$ with the formal convention that $\ele_{1}=\antiele$ and $\bm h\big((\mep1,\mep1)\otimes\mup_1((1,1)_1\otimes\_)\big)=\id$ and zero otherwise. The map \(\tau\) denotes the permutation of tensor factors
\[\tau\colon \ele\otimes\ele_{2-k}\otimes\ele\otimes\ele_{2-\ell}\otimes A^{\otimes n}\xrightarrow{\cong}\ele\otimes\ele_{2-k}\otimes A^{\otimes k}\otimes\ele\otimes\ele_{2-\ell}\otimes A^{\otimes \ell}.\]

To simplify even more the notation, from now on, we will write 
\[
{\Delta}^{\del,r}_{k,\ell}:=\tau_r(\iota_{2-k,2-\ell}\mathfrak{s}_{3-k-\ell}\otimes \id^{\otimes k+\ell-1})\colon \ele_{3-k-\ell}\otimes A^{\otimes k+\ell-1}\to\ele_{2-k}\otimes A^{\otimes r-1}\otimes\ele_{2-\ell}\otimes A^{\otimes \ell+k-r},
\]
where \(\mathfrak{s}_{3-k-\ell}\colon \ele_{3-k-\ell}\to\ele_{4-k-\ell}\) is the map defined in Section~\ref{sec:staircase} and \(\tau_r\) is again the obvious permutation of tensor products. Likewise, we define
\[
\Delta^{\delb,r}_{k,\ell}:=\tau_r(\iota_{2-k,2-\ell}\ov{\mathfrak{s}}_{3-k-\ell}\otimes \id^{\otimes k+\ell-1})\colon \ele_{3-k-\ell}\otimes A^{\otimes k+\ell-1}\to\ele_{2-k}\otimes A^{\otimes r-1}\otimes\ele_{2-\ell}\otimes A^{\otimes \ell+k-r}.
\]

\begin{lemm}\label{lemm:pkernelsprimari}
    The $\p$-kernels satisfy the following relations
    \begin{equation}\label{eq:pkerMn}\tag{$\Mm_n$}
        [\del,\mup_n]=\sum_{\substack{k+\ell=n+1\\1\leq r\leq k\\k,\ell\geq 2}}(-1)^{r(1+\ell)+n}\mup_k(\id^{\otimes r}\otimes gf\mup_\ell\otimes \id^{\otimes k-r}){\Delta}^{\del,r}_{k,\ell}
    \end{equation}
    and
    \begin{equation}\label{eq:pkerMnb}\tag{$\ov{\Mm}_n$}
        [\delb,\mup_n]=\sum_{\substack{k+\ell=n+1\\1\leq r\leq k\\k,\ell\geq 2}}(-1)^{r(1+\ell)+n}\mup_k(\id^{\otimes r}\otimes gf\mup_\ell\otimes \id^{\otimes k-r})\Delta^{\delb,r}_{k,\ell}
    \end{equation}
    for every $n\geq 2$.
\end{lemm}
\begin{proof} We will only prove the relations~\ref{eq:pkerMn}, the relations~\ref{eq:pkerMnb} are analogous.
    We have $\mup_2=\mu$ and it satisfies the relation $(\Mm_2)$. Assume inductively that $\mup_{m}$ satisfies the equation $(\Mm_m)$ for all $m< n$. By definition,
    \begin{align*}
        [\del,\mup_n]=&\sum_{k+\ell=n}(-1)^{k(\ell+1)}\del\mu\big(\bm h(\id\otimes\mup_k)\otimes \bm h(\id\otimes\mup_\ell)\big)\Delta^h_{k,\ell}\\
        &-\sum_{k+\ell=n}(-1)^{k(\ell+1)}\mu\big(\bm h(\id\otimes\mup_k)\otimes \bm h(\id\otimes\mup_\ell)\big)\Delta^h_{k,\ell}\del,
    \end{align*}
    where, by abuse of notation, $\del$ also denotes the induced differential on the tensor product. Using $\del\mu=\mu\del$ and $\del\Delta^h_{k,\ell}=\Delta^h_{k,\ell}
\del$ yields
\begin{align*}
    [\del,\mup_n]=\\
    &+\sum_{k+\ell=n}(-1)^{k(\ell+1)}\mu\big(\del \bm h(\id\otimes \mup_k)\otimes \bm h(\id\otimes \mup_\ell)\big)\Delta^h_{k,\ell}\tag{P1}\label{P1}\\
    &+\sum_{k+\ell=n}(-1)^{k(\ell+1)}\mu\big(\bm h(\id\otimes \mup_k)\otimes \del \bm h(\id\otimes \mup_\ell)\big)\Delta^h_{k,\ell}\tag{P2}\label{P2}\\
    &-\sum_{k+\ell=n}(-1)^{k(\ell+1)}\mu\big(\bm h(\id\otimes \mup_k)\del\otimes \bm h(\id\otimes \mup_\ell)\big)\Delta^h_{k,\ell}\tag{P3}\label{P3}\\
    &-\sum_{k+\ell=n}(-1)^{k(\ell+1)}\mu\big(\bm h(\id\otimes \mup_k)\otimes \bm h(\id\otimes \mup_\ell)\del\big)\Delta^h_{k,\ell}\tag{P4}\label{P4}. 
\end{align*}
Applying the induction hypothesis
    \[
    -\mup_{k}\del=-\del\mup_k+\sum_{\substack{k'+\ell'=k+1\\
    1\leq r\leq k'\\k',\ell'>1}}(-1)^{{r}(1+{\ell'})+{k}}\mup_{k'}(\id^{\otimes {r}}\otimes gf\mup_{\ell'}\otimes \id^{\otimes k'-r})\Delta^{\del,r}_{{k'},{\ell'}}
    \]
for \(k\geq 2\) to \ref{P3} and \ref{P4} gives
\begin{align*}
    \text{\ref{P3}}=&\\
    &-\sum_{\substack{k+\ell=n\\k>1}}(-1)^{k(\ell+1)}\mu\big(\bm h\del(\id\otimes\mup_k)\otimes \bm h(\id\otimes\mup_\ell)\big)\Delta^h_{k,\ell} \tag{P3.1}\label{P3.1}\\
    &+\sum_{\substack{k+\ell=n\\k>1}}\sum_{\substack{k'+\ell'=k+1\\
    1\leq r\leq k'\\k',\ell'>1}}(-1)^{k(\ell+1)+r(1+\ell')+k}\mu\big(\bm h(\id,\mup_{k'}(\id^{\otimes {r}}, gf\mup_{\ell'},\id^{\otimes {k'-r}})\Delta^{\del,r}_{{k'},{\ell'}}), \bm h(\id, \mup_\ell)\big)\Delta^h_{k,\ell} \tag{P3.2}\label{P3.2}\\
    &-(-1)^n\mu\big(\del\otimes \bm h(\id\otimes \mup_{n-1})\big)\Delta^h_{1,n-1} \tag{P3.3}\label{P3.3}
\end{align*}
and
\begin{align*}
    \text{\ref{P4}}=&\\
    &-\sum_{\substack{k+\ell=n\\\ell>1}}(-1)^{k(\ell+1)}\mu\big(\bm h(\id\otimes\mup_k)\otimes \bm h\del(\id\otimes\mup_\ell)\big)\Delta^h_{k,\ell} \tag{P4.1}\label{P4.1}\\
    &+\sum_{\substack{k+\ell=n\\\ell>1}}\sum_{\substack{k'+\ell'=\ell+1\\
    1\leq r\leq k'\\k',\ell'>1}}(-1)^{k(\ell+1)+r(1+\ell')+\ell}\mu\big(\bm h(\id,\mup_k), \bm h(\id,\mup_{k'}(\id^{\otimes {r}}, gf\mup_{\ell'}, \id^{\otimes {k'-r}})\Delta^{\del,r}_{{k'},{\ell'}})\big)\Delta^h_{k,\ell}\tag{P4.2}\label{P4.2}\\
    &-\mu\big(\bm h(\id\otimes\mup_{n-1})\otimes \del\big)\Delta^h_{n-1,1}.\tag{P4.3}\label{P4.3}
\end{align*}
Now grouping \ref{P1}+\ref{P3.1}+\ref{P3.3} and \ref{P2}+\ref{P4.1}+\ref{P4.3} leads to
\begin{align*}
    \text{\ref{P1}+\ref{P3.1}+\ref{P3.3}+\ref{P2}+\ref{P4.1}+\ref{P4.3}}=&\sum_{\substack{k+\ell=n\\k>1}}(-1)^{k(\ell+1)}\mu\big((\del \bm h-\bm h\del)(\id\otimes\mup_k)\otimes \bm h(\id\otimes\mup_\ell)\big)\Delta^h_{k,\ell}\\
    &+\sum_{\substack{k+\ell=n\\l>1}}(-1)^{k(\ell+1)}\mu\big(\bm h(\id\otimes\mup_k)\otimes (\del \bm h-\bm h\del)(\id\otimes \mup_\ell)\big)\Delta^h_{k,\ell}.
\end{align*}
From Remark~\ref{rema:H}, the identity $[\del,\bm h]=(gf-\id)(\mathfrak{s}_1\otimes \id)$ implies
\begin{align*}
    &\sum_{\substack{k+\ell=n\\k>1}}(-1)^{k(\ell+1)}\mu\big(gf(\mathfrak{s}_1\otimes\mup_k)\otimes \bm h(\id\otimes\mup_\ell)\big)\Delta^h_{k,\ell}\\
    &-\sum_{\substack{k+\ell=n\\k>1}}(-1)^{k(\ell+1)}\mu\big((\mathfrak{s}_1\otimes \mup_k)\otimes \bm h(\id\otimes\mup_\ell)\big)\Delta^h_{k,\ell}\\
    &+\sum_{\substack{k+\ell=n\\\ell>1}}(-1)^{k(\ell+1)}\mu\big(\bm h(\id\otimes \mup_k)\otimes gf(\mathfrak{s}_1\otimes\mup_\ell)\big)\Delta^h_{k,\ell}\\
    &-\sum_{\substack{k+\ell=n\\\ell>1}}(-1)^{k(\ell+1)}\mu\big(\bm h(\id\otimes \mup_k)\otimes (\mathfrak{s}_1\otimes\mup_\ell)\big)\Delta^h_{k,\ell}.
\end{align*}
Expanding the definitions of $\mup_k$ and $\mup_\ell$, we obtain
\begin{align*}
    &-\sum_{\substack{k+\ell=n\\k>1}}(-1)^{k(\ell+1)}\mu\big((\mathfrak{s}_1\otimes\mup_k)\otimes \bm h(\id\otimes \mup_\ell)\big)\Delta^h_{k,\ell}\\
    &-\sum_{\substack{k+\ell=n\\\ell>1}}(-1)^{k(\ell+1)}\mu\big(\bm h(\id\otimes\mup_k)\otimes(\mathfrak{s}_1\otimes\mup_\ell)\big)\Delta^h_{k,\ell}=\\
    =&\,-\sum_{\substack{k'+\ell'+\ell=n}}(-1)^{\varepsilon_1}\mu\Big(\mu\big(\bm h(\id,\mup_{k'}), \bm h(\id,\mup_{\ell'})\big),\bm h(\id,\mup_\ell)\Big)(\mathfrak{s}_1,\Delta^h_{k',\ell'},\id^{\otimes 2+\ell})\Delta^h_{k'+\ell',\ell}\\
    &-\sum_{\substack{k+k'+\ell'=n}}(-1)^{\varepsilon_2}\mu\Big(\bm h(\id,\mup_k),\mu\big(\bm h(\id,\mup_{k'}), \bm h(\id,\mup_{\ell'})\big)\Big)(\id^{\otimes 2+k},\mathfrak{s}_1,\Delta^{h}_{k',\ell'})\Delta^{h}_{k,k'+\ell'}=0,
\end{align*}
where $\varepsilon_1=(k'+\ell')(\ell+1)+k'(\ell'+1)$ and $\varepsilon_2=k(k'+\ell'+1)+k'(\ell'+1)$. Here, we have employed the associativity of $\mu$ and the relation
\[(\mathfrak{s}_1,\Delta^h_{k_1,k_2},\id^{\otimes 2+k_3})\Delta^h_{k_1+k_2,k_3}=(-1)^{k_1+1}(\id^{\otimes 2+k_1},\mathfrak{s}_1,\Delta^{h}_{k_2,k_3})\Delta^{h}_{k_1,k_2+k_3}.\]

We rewrite 
\begin{align*}
    &\sum_{\substack{k+\ell=n\\k>1}}(-1)^{k(\ell+1)}\mu\big(gf(\mathfrak{s}_1\otimes\mup_k)\otimes \bm h(\id\otimes\mup_\ell)\big)\Delta^h_{k,\ell}\\
    &+\sum_{\substack{k+\ell=n\\\ell>1}}(-1)^{k(\ell+1)}\mu\big(\bm h(\id\otimes\mup_k)\otimes gf(\mathfrak{s}_1\otimes\mup_\ell)\big)\Delta^h_{k,\ell}
\end{align*}
    as 
\begin{align*}\tag{P5}\label{P5}
    &\sum_{\substack{k+\ell=n\\k>1}}\mu\big(\id\otimes \bm h(\id\otimes\mup_\ell)\big)\Delta^h_{1,\ell}(\id\otimes gf\mup_k\otimes\id^{\otimes \ell})\Delta^{\del,1}_{\ell+1,k}\\
    &+\sum_{\substack{k+\ell=n\\\ell>1}}(-1)^{k(\ell+1)+k+1}\mu\big(\bm h(\id\otimes\mup_k)\otimes \id)\big)\Delta^h_{k,1}(\id^{\otimes k+1}\otimes i p\mup_\ell)\Delta^{\del,k+1}_{k+1,\ell}.
\end{align*}
The expression \ref{P3.2} becomes
\begin{align*}
        &\sum_{\substack{k'+\ell'+\ell=n+1\\1\leq r\leq k'\\k',\ell'>1}}(-1)^{\varepsilon_1}\mu\big(\bm h(\id\otimes\mup_{k'})\otimes \bm h(\id\otimes\mup_\ell)\big)\Delta^h_{k',\ell}(\id^{\otimes {r}}\otimes gf\mup_{\ell'}\otimes \id^{\otimes {k'-r+\ell}})\Delta^{\del,r}_{{k'+\ell},{\ell'}},
        \end{align*}
where $\varepsilon_1=r(1+\ell')+k'+\ell+1+(k'+\ell'+1)(\ell+1)+\ell'(\ell+1)+1$.

Similarly, we rewrite \ref{P4.2} as
\begin{align*}
    &\sum_{\substack{k+k'+\ell'=n+1\\1\leq r\leq k'\\k',\ell'>1}}(-1)^{\varepsilon_2}\mu\big(\bm h(\id\otimes\mup_k)\otimes \bm h(\id\otimes\mup_{k'})\big)\Delta^h_{k,k'}(\id^{\otimes {k+r}}\otimes gf\mup_{\ell'}\otimes \id^{\otimes {k'-r}})\Delta^{\del,k+r}_{{k+k'},{\ell'}},
\end{align*}
where $\varepsilon_2=k(k'+\ell')+r(1+\ell')+k'+\ell'+k+1$.
Finally, reindexing and summing \ref{P5}+\ref{P3.2}+\ref{P4.2} completes the proof.
\end{proof}

\begin{rema}\label{rema:ppAinftyAEscales}
    A pluripotential $A_\infty$-algebra is equivalently defined as a bicomplex $(A, \del, \delb)$ together with $n$-ary operations
    \[
    \bm\mu_n \colon  \ele_{n-2} \otimes A^{\otimes n} \to A
    \]
    for each $n \geq 2$, each of bidegree $(0,0)$, satisfying the following conditions:
    \begin{align*}
        [\del, \bm\mu_n] &= \sum_{\substack{k + \ell = n + 1 \\ 1\leq r\leq k}} (-1)^{r(1 + \ell) + n} \, \bm\mu_k \left( \id^{\otimes r} \otimes \bm\mu_\ell \otimes \id^{\otimes k-r} \right) \Delta^{\del,r}_{k,\ell}, \\
        [\delb, \bm\mu_n] &=\sum_{\substack{k + \ell = n + 1 \\ 1\leq r\leq k}} (-1)^{r(1 + \ell) + n} \, \bm\mu_k \left( \id^{\otimes r} \otimes \bm\mu_\ell \otimes \id^{\otimes k-r} \right) \Delta^{\delb,r}_{k,\ell}.
    \end{align*}
\end{rema}

\begin{lemm}\label{lemm:pkerAinfy}
    The $\p$-kernels induce a $\ppinfa$-structure on $B$ by the formula \[\bm\nu_n=f\circ\mup_n\circ (\id\otimes g^{\otimes n}).\]
\end{lemm}
\begin{proof}
    By Remark~\ref{rema:ppAinftyAEscales}, the operations $\bm\nu_n$ constitute an $\ppinfa$-structure if and only if 
    \begin{align*}
        [\del, \bm\nu_n] -\sum_{\substack{k + \ell = n + 1 \\ 1\leq r\leq  k}} (-1)^{r(1 + \ell) + n} \, \bm\nu_k \left( \id^{\otimes r} \otimes \bm\nu_\ell \otimes \id^{\otimes k-r} \right)\Delta^{\del,r}_{k,\ell}=0,
    \end{align*}
    and 
     \begin{align*}
        [\delb, \bm\nu_n] -\sum_{\substack{k + \ell = n + 1 \\ 1 \leq r\leq  k}} (-1)^{r(1 + \ell) + n} \, \bm\nu_k \left( \id^{\otimes r} \otimes \bm\nu_\ell \otimes \id^{\otimes k-r} \right)\Delta^{\delb,r}_{k,\ell}=0.
     \end{align*}
    Expanding the definition of $\bm\nu_n$ in the first expression, we obtain
    \begin{align*}
        \del f\mup_n(\id\otimes g^{\otimes n})-f\mup_n(\del\otimes g^{\otimes n})-\sum_{j=1}^nf\mup_n(\id\otimes g^{\otimes j-1}\otimes g\del\otimes g^{n-j})& \\
        -\sum_{\substack{k + \ell = n + 1 \\ 1\leq r\leq k}} (-1)^{r(1 + \ell) + n} \, f\mup_k \left( \id\otimes g^{\otimes r-1} \otimes gf\mup_\ell(\id\otimes g^{\otimes \ell}) \otimes g^{\otimes k-r} \right)\Delta^{\del,r}_{k,\ell}
    \end{align*}
    which gives
    \begin{align*}
        f\bigg(\del\mup_n-\sum_{j=1}^{n+1}\mup_n(\id^{\otimes j-1}\otimes \del\otimes \id^{\otimes n+1-j})&\\
        -\sum_{\substack{k + \ell = n + 1 \\ 1\leq r\leq  k\\k,\ell\geq 2}} (-1)^{r(1 + \ell) + n} \, \mup_k \left( \id^{\otimes r} \otimes gf\mup_\ell\otimes \id^{\otimes k-r} \right) \Delta^{\del,r}_{k,\ell}\bigg)(\id\otimes g^{\otimes n})&=0
    \end{align*}
by Lemma~\ref{lemm:pkernelsprimari}. Here, we use that $g$ and $f$ are maps of bicomplexes and that 
\[(\id\otimes g^{\otimes r}\otimes \id\otimes g^{\otimes \ell+s})\Delta^{\del,r}_{k,\ell}=\Delta^{\del,r}_{k,\ell}(\id\otimes g^{\otimes n}).\]
A similar proof shows
\begin{align*}
    [\delb, \bm\nu_n] -\sum_{\substack{k + \ell = n + 1 \\ 1 \leq r\leq  k}} (-1)^{r(1 + \ell) + n} \, \bm\nu_k \left( \id^{\otimes r} \otimes \bm\nu_\ell \otimes \id^{\otimes k-r} \right)\Delta^{\delb,r}_{k,\ell}=0.
\end{align*}
\end{proof}

\begin{rema}\label{rema:ppAinftyMEscales}
    Let $(A,\bm\mu_*^A)$ and $(B,\bm\mu_*^B)$ be two pluripotential $A_\infty$-algebras. Similarly to Remark~\ref{rema:ppAinftyAEscales}, a \(\ppinfm\)-morphism $f\colon A \rightsquigarrow B$ is equivalently defined as a family of bidegree-preserving maps
    \[
    \bm f_n \colon  \ele_{n-1}\otimes A^{\otimes n} \to B,\quad n\geq 1,
    \]
    which satisfy the following relations:

\begin{align*}
[\del, \bm f_n] =& -\sum_{i_1+\cdots+i_k=n} (-1)^{\theta(i_1,\dots,i_k)}\, \bm\mu^B_k\Bigl(\id\otimes\bm f_{i_1}\otimes \cdots \otimes \bm f_{i_k}\Bigr)\Delta^{\del}_{k,i_1+1,\dots,i_k+1}\\[1mm]
&-\sum_{\substack{k+\ell=n+1\\ 1\leq r\leq k}} (-1)^{r(1+\ell)+n}\, \bm f_k\Bigl(\id^{\otimes r}\otimes\bm\mu^A_\ell\otimes \id^{\otimes k-r}\Bigr) \Delta^{\del,r}_{k+1,\ell}.
\end{align*}

\begin{align*}
[\delb, \bm f_n] =& -\sum_{i_1+\cdots+i_k=n} (-1)^{\theta(i_1,\dots,i_k)}\, \bm\mu^B_k\Bigl(\id\otimes\bm f_{i_1}\otimes \cdots \otimes \bm f_{i_k}\Bigr) \Delta^{\delb}_{k,i_1+1,\dots,i_k+1}\\[1mm]
&-\sum_{\substack{k+\ell=n+1\\ 1\leq r\leq k}} (-1)^{r(1+\ell)+n}\, \bm f_k\Bigl(\id^{\otimes r}\otimes\bm\mu^A_\ell\otimes \id^{\otimes k-r}\Bigr) \Delta^{\delb,r}_{k+1,\ell}.
\end{align*}

Here, 
\[
\Delta^{\del}_{k,i_1+1,\dots,i_k+1}:=(\tau_{i_1,\dots,i_k})(\iota_{2\mep k,1\mep i_1,\dots,1\mep i_k})\mathfrak{s}_{1\mep n}:\ele_{1\mep n}\otimes A^{\otimes n}\to \ele_{2\mep k}\otimes \ele_{1\mep i_1}\otimes A^{\otimes i_1}\otimes \cdots \otimes \ele_{1\mep i_k}\otimes A^{\otimes i_k}
\]
and
\[
\Delta^{\delb}_{k,i_1+1,\dots,i_k+1}:=(\tau_{i_1,\dots,i_k})(\iota_{2\mep k,1\mep i_1,\dots,1\mep i_k})\ov{\mathfrak{s}}_{1\mep n}:\ele_{1\mep n}\otimes A^{\otimes n}\to \ele_{2\mep k}\otimes \ele_{1\mep i_1}\otimes A^{\otimes i_1}\otimes \cdots \otimes \ele_{1\mep i_k}\otimes A^{\otimes i_k},
\]
where \(\tau_{i_1,\dots,i_k}\) is the isomorphism that permutes the tensor factors.
The sign  $\theta$ is given by \[\theta(i_1,\dots,i_k):=\sum_{1\leq r<s\leq k}i_r(i_s+1).\]
\end{rema}

\begin{lemm}\label{lemm:pkerMinfty}
The $\mathfrak{p}$-kernels define a $\ppinfm$-morphism $g:B\rightsquigarrow A$ via the formula
\[\bm g_n:=\bm h(\id\otimes \mup_n)(\iota_{\mep1,2\mep n}\otimes g^{\otimes n}),\]
for $n\geq 2$ and $\bm g_1=g$.
\end{lemm}
\begin{proof}
    Let us compute
    \[\del\bm g_n-\bm g_n\del=\del \bm h(\id\otimes \mup_n)(\iota_{\mep1,2\mep n}\otimes g^{\otimes n})-\bm h(\id\otimes \mup_n)(\iota_{\mep1,2\mep n}\otimes g^{\otimes n})\del.\]
    By Lemma~\ref{lemm:pkernelsprimari}
    \begin{align*}
        \del\bm g_n-\bm g_n\del=\\
        =&\;\del \bm h(\id\otimes \mup_n)(\iota_{\mep1,2\mep n}\otimes g^{\otimes n})-\bm h\del(\id\otimes \mup_n)(\iota_{\mep1,2\mep n}\otimes g^{\otimes n})\\[1mm]
        &+\sum_{\substack{k+\ell=n+1\\1\leq r\leq k \\k,\ell\geq 2}}(-1)^{\theta}\bm h(\id\otimes\mup_k(\id^{\otimes r}, gf\mup_\ell, \id^{\otimes k-r}))(\id\otimes\Delta^{\del,r}_{k,\ell}\otimes\id^{\otimes n})(\iota_{\mep1,2\mep n}\otimes g^{\otimes n}),
    \end{align*}
    with \(\theta=r(1+\ell)+n.\)
    Now we use the relation $[\del,\bm h]=(gf-\id)(\mathfrak{s}_1\otimes\id)$ in Remark~\ref{rema:H} and that \[(\id\otimes\Delta^{\del,r}_{k,\ell})\iota_{\mep1,2\mep n}=-(\iota_{\mep1,2\mep k}\otimes\id)\Delta_{k+1,\ell}^{\del,r}\] to obtain
    \begin{align*}
        gf(\mathfrak{s}_1\otimes \mup_n)(\iota_{\mep1,2\mep n}\otimes g^{\otimes n})-(\mathfrak{s}_1\otimes \mup_n)(\iota_{\mep1,2\mep n}\otimes g^{\otimes n})\\[1mm]
        -\sum_{\substack{k+\ell=n+1\\1\leq r\leq k\\k,\ell\geq 2}}(-1)^{r(1+\ell)+n}\bm h(\id\otimes\mup_k)(\iota_{\mep1,2\mep k}\otimes\id^{\otimes k})(\id^{\otimes r}\otimes gf\mup_\ell\otimes \id^{\otimes k-r}))\Delta_{k+1,\ell}^{\del,r} (\id\otimes g^{\otimes n})
    \end{align*}
        which gives
        \begin{align*}
            gf(\mup_n)(\id\otimes g^{\otimes n})\Delta^{\del,1}_{2,n}-\sum_{k+\ell=n} (-1)^{k(\ell+1)}\, \mu\Bigl(\bm h(\id\otimes \mup_k)\otimes \bm h(\id\otimes\mup_{\ell})\Delta^h_{k,\ell}\Bigr)\Delta^{\del,1}_{2,n}\otimes (\id\otimes g^{\otimes n})\\
            -\sum_{\substack{k+\ell=n+1\\ 1\leq r\leq k \\k,\ell\geq 2}}(-1)^{r(1+\ell)+n}\bm g_k(\id^{\otimes r}\otimes \bm\nu_\ell\otimes\id^{\otimes k-r})\Delta^{\del}_{k+1,\ell}=\\
            =-\sum_{k+\ell=n} (-1)^{k(\ell+1)}\, \mu\Bigl(\bm g_k\otimes \bm g_{\ell}\Bigr)\Delta^{\del}_{2,k+1,\ell+1}\\
            -\sum_{\substack{k+\ell=n+1\\ 1\leq r \leq k}}(-1)^{r(1+\ell)+n}\bm g_k(\id^{\otimes r}\otimes \bm\nu_\ell\otimes\id^{\otimes k-r})\Delta^{\del}_{k+1,\ell,r}.
        \end{align*}
A similar computation shows
\begin{align*}
    \delb\bm g_n-\bm g_n\delb
    =&-\sum_{k+\ell=n} (-1)^{k(\ell+1)}\, \mu\Bigl(\bm g_k\otimes \bm g_{\ell}\Bigr)\Delta^{\delb}_{2,k+1,\ell+1}\\
    &-\sum_{\substack{k+\ell=n+1\\ 1\leq r \leq k}}(-1)^{r(1+\ell)+n}\bm g_k(\id^{\otimes r}\otimes \bm\nu_\ell\otimes\id^{\otimes k-r})\Delta^{\delb}_{k+1,\ell,r}.\qedhere
\end{align*}
\end{proof}

Lemma~\ref{lemm:pkerAinfy} and Lemma~\ref{lemm:pkerMinfty} imply the following, which is equivalent to Lemma~\ref{lemm:HTT} by Remarks~\ref{rema:ppAinftyAEscales} and \ref{rema:ppAinftyMEscales}.

\begin{lemm}
    Let $(A,\del,\delb,\mu)$ be a bidifferential bigraded algebra and let
    \[
    \begin{tikzcd}[ampersand replacement = \&]
    (A,\del,\delb,\mu) \arrow[r, rightarrow, shift left, "f"] \arrow[r, shift right, leftarrow, "g" swap] \arrow[loop left, "h"] \& (B,\del,\delb)
    \end{tikzcd}
    \]
be a contraction. The $\mathfrak{p}$-kernels 
\[\mup_n:\ele_{2-n}\otimes A^{\otimes n}\to A\]
given by
\begin{equation*}
\mup_n:=\sum_{k+\ell=n}(-1)^{k(\ell+1)}\mu\big(\bm h(\id\otimes\mup_k)\otimes \bm h(\id\otimes\mup_\ell)\big)\Delta^h_{k,\ell}
\end{equation*}
for $n\geq 2$ induce 
\begin{enumerate}
    \item a \(\ppinfa\)-structure on \(B\) via \(\bm\nu_n=f\circ\mup_n\circ (\id\otimes g^{\otimes n})\), and 
    \item a \(\ppinfm\)-weak equivalence \(g:B\rightsquigarrow A\) via 
    \(\bm g_n:=\bm h(\id\otimes \mup_n)(\iota_{\mep1,2\mep n}\otimes g^{\otimes n})\) for \(n\geq 2\) and \(\bm g_1=g.\)
\end{enumerate}
\end{lemm}

\bibliographystyle{alpha}
\bibliography{biblio}

\end{document}